\documentclass{amsart}
\usepackage{amsmath,amssymb,amsfonts}
\IfFileExists{yhmath.sty}{\usepackage{yhmath}}{}
\usepackage{xcolor}
\usepackage{hyperref}
\usepackage{tikz-cd}
\usepackage{amscd}
\usepackage[all]{xy}
\usepackage{enumerate}
\def\Z{{\mathbb Z}}
\def\Zl{{\mathbb Z_\ell}}
\def\Q{{\mathbb Q}}
\def\Ql{{\mathbb Q_\ell}}

\def\C{{\mathbb C}}

\def\F{{\mathbb F}}

\def\V{{\mathbb V}}

\def\cA{{\mathcal A}}
\def\cC{{\mathcal C}}

\def\cG{{\mathcal G}}
\def\cN{{\mathcal N}}
\def\cP{{\mathcal P}}

\def\G{{\Gamma}}

\def\J{{\mathcal J}}
\def\L{{\mathbb L}}
\def\M{{\mathcal M}}
\def\Mbar{\overline{\M}}

\def\U{{\mathcal U}}

\def\g{{\mathfrak g}}
\def\n{{\mathfrak n}}
\def\p{{\mathfrak p}}

\def\u{{\mathfrak u}}

\def\an{\mathrm{an}}

\def\et{\mathrm{\acute et}}
\def\hyp{\mathrm{hyp}}
\def\orb{\mathrm{orb}}
\def\top{\mathrm{top}}
\def\un{\mathrm{un}}
\def\rel{{\mathrm{rel}(\ell)}}
\newcommand{\ad}{\operatorname{ad}}
\newcommand{\Ad}{\operatorname{Ad}}
\newcommand{\Gr}{\operatorname{Gr}}
\newcommand{\Hom}{\operatorname{Hom}}
\newcommand{\Lie}{\operatorname{Lie}}
\newcommand{\Out}{\operatorname{Out}}
\newcommand{\Der}{\operatorname{Der}}
\newcommand{\OutDer}{\operatorname{OutDer}}

\newcommand{\GL}{\operatorname{GL}}
\newcommand{\Sp}{\operatorname{Sp}}
\newcommand{\Pic}{\operatorname{Pic}}
\newcommand{\Spec}{\operatorname{Spec}}
\newcommand{\Loc}{\operatorname{Loc}}
\newcommand{\Rep}{\operatorname{Rep}}
\newcommand{\im}{\operatorname{im}}
\newcommand{\id}{\operatorname{id}}
\newcommand{\contr}{\operatorname{contr}}
\newtheorem{theorem}{Theorem}[section]
\newtheorem{lemma}[theorem]{Lemma}
\newtheorem{proposition}[theorem]{Proposition}
\newtheorem{corollary}[theorem]{Corollary}
\newtheorem{bigtheorem}{Theorem}
\newtheorem{bigcorollary}[bigtheorem]{Corollary}
\theoremstyle{definition}
\newtheorem{definition}[theorem]{Definition}
\theoremstyle{remark}
\newtheorem{remark}[theorem]{Remark}

\newif\ifshowchanges
\showchangestrue

\begin{document}

\title[Relative completions and separating boundary divisors]
{Relative completions and separating boundary divisors in moduli of curves}
\author{Ma Luo}
\address{School of Mathematical Sciences, Key Laboratory of MEA (Ministry of Education), Shanghai Key Laboratory of PMMP, East China Normal University, Shanghai}
\email{mluo@math.ecnu.edu.cn}
\author{Tatsunari Watanabe}
\address{Mathematics Department, Embry-Riddle Aeronautical University, Prescott}
\email{watanabt@erau.edu}
\thanks{The first author is supported in part by Shanghai Pilot Program for Basic Research and by Science and Technology Commission of Shanghai Municipality (No. 22DZ2229014).}

\begin{abstract}
We study relative completions of fundamental groups of moduli stacks
of curves of compact type and of partial compactifications of the
smooth moduli stack. For $g\geq3$, the compact-type completion has
commutative prounipotent radical with Lie algebra
$\Lambda^3_0H\oplus H^{\oplus n}$. Fix $1\leq h\leq g-1$.
We show that adjoining the one-node loci of the separating divisors
whose genus-$h$ side carries at most two markings gives the same
relative completion. The proof uses explicit separating-twist
calculations in weight $-2$, pointed Johnson homomorphisms, and the
lantern relation. Functorial specialization gives the corresponding
results in positive characteristic for odd $\ell$ different from the
characteristic. As a consequence, extension to compact type of
iterated extensions of algebraic symplectic local systems can be
tested on these selected divisors. We also prove that any collection
of separating divisors giving the same relative completion contains
at least $\binom n2$ divisors.
\end{abstract}
\maketitle
\setcounter{tocdepth}{1}
\tableofcontents

% REWRITTEN INTRODUCTION. Printed in black; local edits to retained
% sections, rather than whole new sections, are marked in red.
\section{Introduction}\label{sec:introduction}
Separating degenerations govern the relative completion of the
moduli stack of curves of compact type. Passing from smooth curves
to curves of compact type kills the monodromy around every
separating boundary divisor. Our main result shows that, for relative completion with
respect to symplectic monodromy, these relations are already imposed
by a much smaller collection: one may prescribe the genus on one
side of the node and use only the divisors for which that side
carries at most two marked points. For unpointed curves, the
one-node locus of any single separating divisor suffices.

Suppose that $g\geq3$ and $n\geq0$. Let $\M_{g,n}$ denote the moduli
stack of smooth curves of genus $g$ with $n$ ordered marked points,
and let $\M^c_{g,n}$ denote the moduli stack of stable $n$-pointed
curves of compact type. The complement of $\M_{g,n}$ in $\M^c_{g,n}$
consists of the separating boundary divisors. We study the relative
completions of their fundamental groups and of the partial
compactifications obtained by adjoining selected separating divisors.
The relative Jacobian extends over compact type and supplies the
symplectic monodromy representations used throughout. We write
\[
\M^{c,\an}_{g,n}:=(\M^c_{g,n/\C})^{\an}
\]
for the complex analytification of the compact-type moduli stack.

Let $\G_{g,n}$ be the mapping class group of a closed oriented surface
$\Sigma_g$ with $n$ ordered marked points, and put
$H=H_1(\Sigma_g,\Q)$. Denote the standard symplectic representation by
\[
\rho_{g,n}:\G_{g,n}\longrightarrow\Sp(H).
\]
Its kernel is the Torelli group $T_{g,n}$. Let $K_{g,n}$ be the
subgroup generated by Dehn twists about separating simple closed
curves. Boggi's description of the compact-type fundamental group gives
\[
\pi_1^{\orb}(\M^{c,\an}_{g,n})
\cong\G_{g,n}/K_{g,n};
\]
see \cite[Prop.~3.1]{boggi_compact_type}. Moreover,
$T_{g,n}/K_{g,n}$ is free abelian
\cite[Thm.~3.18]{boggi_compact_type}.

Write $\cG_{g,n}$ for the relative completion of $\rho_{g,n}$,
$\U_{g,n}$ for its prounipotent radical, and $\u_{g,n}$ for the Lie
algebra of $\U_{g,n}$. Hain's results give \cite[Thm.~9.11]{hain_rational}
\[
H_1(\u_{g,n})\cong\Lambda^3_0H\oplus H^{\oplus n}.
\]
Here $\Lambda^3_0H$ denotes the weight $-1$ realization
$(\Lambda^3H/(\theta\wedge H))(-1)$ of the irreducible
$\Sp(H)$-module $V(\lambda_3)$. The image of $T_{g,n}$ is Zariski
dense in $\U_{g,n}$, and separating twists have zero image in
$H_1(\u_{g,n})$. It follows from Boggi's abelian quotient that the
prounipotent radical of the compact-type relative completion is
commutative. Its Lie algebra is therefore the unchanged
abelianization $\Lambda^3_0H\oplus H^{\oplus n}$. We give this
argument in Section~\ref{sec:compact-type-relative-completion}.

We use the following notation in the \'etale setting. Let $k$ be an
algebraically closed field and choose a prime
$\ell\neq\operatorname{char}(k)$, assumed odd if
$\operatorname{char}(k)>0$. Fix a geometric point
$\bar\xi\to\M_{g,n/k}$ corresponding to a smooth curve
$C_{\bar\xi}$, and put
\[
H=H_{\bar\xi,\Ql}:=H^1_{\et}(C_{\bar\xi},\Ql(1)).
\]
The relative Jacobian over $\M^c_{g,n/k}$ defines the symplectic
monodromy representation
\[
\rho^c_{g,n/k}:\pi_1(\M^c_{g,n/k},\bar\xi)
\longrightarrow\Sp(H).
\]
Denote its continuous relative completion by $\cG^c_{g,n/k}$ and
the Lie algebra of its prounipotent radical by $\u^c_{g,n/k}$.

\begin{bigtheorem}\label{thm:compact-type-completion}
Suppose that $g\geq3$ and $n\geq0$. With the conventions above,
$\u^c_{g,n/k}$ is abelian, and there is an isomorphism
\[
\u^c_{g,n/k}\cong\Lambda^3_0H\oplus H^{\oplus n}
\]
of $\Sp(H)$-modules.
\end{bigtheorem}
In characteristic zero, the theorem follows from the preceding
argument and the comparison between discrete and continuous relative
completion. In positive characteristic, we construct an isomorphism
of relative pro-$\ell$ completions compatible with symplectic
monodromy, and then pass to continuous relative completion.
This comparison also applies to the partial compactifications below,
where the fixed-genus theorem identifies a smaller set of boundary
relations with the same effect on relative completion.
% In characteristic zero, the theorem follows from the preceding
% argument and the comparison between discrete and continuous relative
% completion. In positive characteristic, we construct an isomorphism
% of relative pro-$\ell$ completions compatible with symplectic
% monodromy, and then pass to continuous relative completion.
% The same comparison will be needed for the partial compactifications.
% The compact-type calculation thus follows from Boggi's and Hain's
% results in characteristic zero. The fixed-genus theorem below
% identifies a smaller set of boundary relations with the same effect
% on relative completion, and the specialization comparison transfers
% this assertion to positive characteristic.

For a stable separating type $(h,S)$, let $\Delta_{h,S}$ be the
closed boundary divisor whose genus-$h$ component carries precisely
the markings in $S\subseteq\{1,\ldots,n\}$. We identify $(h,S)$ with
$(g-h,S^c)$.

For a separating divisor $D$, let $D^\circ$ denote its locus of
curves with exactly one node, so that the curve has two smooth
components meeting at that node. For a collection $\mathcal D$ of
separating divisors, put
\[
\M_{\mathcal D}^{\circ}:=\M_{g,n}\cup\bigcup_{D\in\mathcal D}D^\circ.
\]
The superscript $\circ$ indicates that curves have at most one node;
the smooth locus is included. The pair $(g,n)$ is fixed and is
suppressed in the subscript.
For $1\leq h\leq g-1$, write
\[
\begin{aligned}
\mathcal D_h&=\{\Delta_{h,S}:S\subseteq\{1,\ldots,n\}\},\\
\mathcal D_{h,\leq2}&=\{\Delta_{h,S}:S\subseteq\{1,\ldots,n\},\ |S|\leq2\}.
\end{aligned}
\]
The corresponding one-node partial compactifications are
$\M_{\mathcal D_h}^{\circ}$ and $\M_{\mathcal D_{h,\leq2}}^{\circ}$.
We write $\M_{\mathcal D/k}^{\circ}$ after base change to $k$.

For any such open $Z$, write
$\cG_{Z/k}$ for the continuous relative completion of its geometric
fundamental group over $\Ql$ with respect to the restriction of
$\rho^c_{g,n/k}$. We use $\cG_Z$ for its rational Betti relative
completion when $Z$ is considered over $\C$.

\begin{bigtheorem}[Fixed-genus boundary theorem]
\label{thm:fixed-genus-main}
Suppose that $g\geq3$, $n\geq0$, and $1\leq h\leq g-1$.
The inclusions induce isomorphisms
\[
\cG_{\M_{\mathcal D_{h,\leq2}}^{\circ}/k}
\xrightarrow{\ \sim\ }
\cG_{\M_{\mathcal D_h}^{\circ}/k}
\xrightarrow{\ \sim\ }\cG^c_{g,n/k}.
\]
The same assertion holds for rational Betti relative completions.
Their prounipotent radicals are commutative and have Lie algebra
$\Lambda^3_0H\oplus H^{\oplus n}$.
\end{bigtheorem}

Thus, for $n=0$, the one-node locus of any one positive-genus
separating divisor suffices. For general $n$, the theorem uses at
most $1+n+\binom n2$ distinct divisors. There can be fewer when
$g=2h$, since complementary subsets label the same divisor.
Thus a collection of quadratic size in $n$ suffices even though the
full fixed-genus family grows exponentially with the number of
markings. The genus $h$ may be fixed in advance.
The assertion is about relative proalgebraic completions; no
isomorphism of full fundamental groups or of relative pro-$\ell$
groups is asserted for these inclusions.

The key step is to detect the separate marked-point and diagonal
contributions using twists of the prescribed genus. The proof in
characteristic zero uses the weight $-2$ computation
\[
\Gr^W_{-2}\u_{g,n}
\cong V(2\lambda_2)\oplus V(\lambda_2)^{\oplus n}
\oplus\Q(1)^{\oplus\binom n2}
\]
of Hain \cite[Prop.~8.6 and Thm.~9.11]{hain_rational}.
The logarithm of a separating twist belongs to $W_{-2}\u_{g,n}$.
We show that the initial forms of the selected genus-$h$ twists
generate every summand in this display. For the unpointed summand,
we use the projection calculation in
\cite[\S6.3]{watanabe_sections} and verify its nonvanishing for
each prescribed $h$. The difference of two pointed realizations of
a separating twist detects one marked-point summand at a time,
using the computation of Luo--Watanabe
\cite[proof of Thm.~9.1 and Lem.~9.3]{LuoWatanabeCollino}.
Finally, the lantern relation expresses a two-point diagonal
meridian in terms of four genus-$h$ twists. This gives each
individual Tate summand.

For a collection $\mathcal D$, write $N_{\mathcal D}$ for the
normal subgroup generated by its separating twists,
$\mathcal N_{\mathcal D}$ for its closed normal image in
$\cG_{g,n}$, and $\mathfrak n_{\mathcal D}:=\Lie\mathcal N_{\mathcal D}$.
Since the weight filtration on $\u_{g,n}$ is its lower central
series, spanning in weight $-2$ implies
$\mathfrak n_{\mathcal D}=W_{-2}\u_{g,n}$ for the selected
fixed-genus collection. Its closed normal image therefore equals
the closed normal image of all separating twists. This proves the Betti assertion of
Theorem~\ref{thm:fixed-genus-main} without a calculation in lower
weights. Section~\ref{sec:image-Kgn} contains the fixed-genus
spanning calculation, which also implies the spanning statement for
all separating twists. Section~\ref{sec:fixed-genus} applies this
calculation to the partial compactifications.

For the characteristic-$p$ comparison, let $\mathcal D$ be any
collection of separating boundary types. Besides the one-node open
$\M_{\mathcal D}^{\circ}$, we consider the open
$\M_{\mathcal D}$ obtained by removing the nonseparating divisor
and all separating divisors not in $\mathcal D$. Their difference
has codimension at least two, so purity identifies their fundamental
groups. A single Boggi--Pikaart level cover supplies a smooth proper
compactification with relative normal-crossings boundary for all the
larger opens. Tame specialization on this cover and finite-cover descent
give a comparison of relative pro-$\ell$ completions, functorial for
enlarging $\mathcal D$. This transfers the inclusion-induced
isomorphisms in Theorem~\ref{thm:fixed-genus-main} to positive
characteristic.

We apply the theorem to extension of local systems. Let $F=\Q$ in
the rational Betti setting and $F=\Ql$ in the \'etale setting above.
Write $\Loc_H(Z;F)$ for the category of finite-dimensional local
systems on $Z$ admitting a finite filtration whose successive
quotients are associated to algebraic representations of $\Sp(H)$.
The filtration is required to exist but is not part of the data.
Relative completion identifies this category with the category of
finite-dimensional algebraic representations of $\cG_Z$ or
$\cG_{Z/k}$, respectively.

\begin{bigcorollary}\label{thm:local-extension-main}
Under the hypotheses of Theorem~\ref{thm:fixed-genus-main}, restriction
induces an equivalence of tensor categories
\[
\Loc_H(\M^c_{g,n};F)\xrightarrow{\ \sim\ }
\Loc_H(\M_{\mathcal D_{h,\leq2}}^{\circ};F).
\]
A local system in $\Loc_H(\M_{g,n};F)$ extends to compact type if and
only if its boundary monodromy is trivial along every
$\Delta_{h,S}$ with $|S|\leq2$. In the \'etale setting, triviality
means trivial action of the full transverse boundary inertia group.
The extension is unique up to a unique isomorphism compatible with
its specified restriction.
\end{bigcorollary}

Algebraic symplectic local systems themselves extend over compact
type through the relative Jacobian. The corollary controls the
additional boundary monodromy of their nonsplit iterated extensions.
The assertion concerns their underlying
local systems; no additional Hodge or arithmetic structure is part of
the statement. The same result holds with all subsets $S$, or with
the corresponding larger opens.

\begin{bigcorollary}\label{cor:divisor-lower-bound-main}
Suppose that $\mathcal D$ consists of $r$ distinct separating boundary
divisors and that
$\M_{\mathcal D}^{\circ}\hookrightarrow\M^c_{g,n}$ induces an
isomorphism on the relative completions above. Then
\[
r\geq\binom n2.
\]
\end{bigcorollary}

Each divisor contributes at most one dimension to the trivial
isotypic component of $\Gr^W_{-2}\u_{g,n}$, which proves the bound.
Together with Theorem~\ref{thm:fixed-genus-main}, it shows that the
upper and lower bounds have the same quadratic leading term. Thus
the quadratic growth of the number of divisors is necessary, although
we do not determine the minimum number of divisors.

The paper is organized as follows. Section~\ref{sec:background}
recalls the mapping class groups and compact-type moduli stacks.
Sections~\ref{sec:relative-pro-l-completion} and
\ref{sec:relative-completion} review relative pro-$\ell$ completion
and relative proalgebraic completion. In
Section~\ref{sec:johnson-subgroup}, we locate separating twists in
the weight filtration. Section~\ref{sec:compact-type-relative-completion}
gives the characteristic-zero compact-type computation, and
Section~\ref{sec:image-Kgn} proves the fixed-genus spanning
result in weight $-2$. Section~\ref{sec:fixed-genus} deduces the
characteristic-zero relative completions of the partial compactifications. Section~\ref{sec:specialization-characteristic-p}
establishes the functorial specialization comparison and completes
the proofs in positive characteristic. The local-system consequences
and the divisor bound are proved in Sections~\ref{sec:local-systems}
and~\ref{sec:divisor-count}.

% \subsection*{AI disclosure}
% The authors used OpenAI's ChatGPT to assist with literature
% searches, check calculations and notation, and improve the
% exposition of their original mathematical ideas and arguments.
% All AI-assisted suggestions were independently reviewed and
% verified by the authors, who take full responsibility for the
% content of the paper.

\section{Mapping Class Groups and Compact-Type Moduli}
\label{sec:background}
% \authorcomment{
% Thank you for putting this background together. I have shortened
% this section, keeping the definitions and results used later in the
% paper. In particular, I have condensed the Jacobian and Torelli
% discussion to the statements needed for monodromy. These edits are
% already incorporated in the text. Parenthesized notes are editorial
% comments; red text outside these notes is part of the manuscript.
% }
\subsection{Mapping class groups and moduli of curves}
Suppose that $g\geq3$ and $n\geq0$. Let $\Sigma_{g,n}$ be a closed,
connected, oriented surface of genus $g$ with $n$ ordered distinct
marked points. The \emph{mapping class group} $\G_{g,n}$ is the group
of isotopy classes of orientation-preserving diffeomorphisms fixing
each marked point, where isotopies also fix the marked points.
Its action on homology gives the surjective symplectic representation
\[
\rho_{g,n}:\G_{g,n}\longrightarrow
\Sp(H_1(\Sigma_g,\Z)).
\]
Its kernel is the \emph{Torelli group} $T_{g,n}$.

Let $\M_{g,n}$ denote the moduli stack over $\Z$ of smooth, proper,
geometrically connected curves of genus $g$ with $n$ ordered distinct
marked points, and write $\M_{g,n/k}$ for its base change to a field $k$.
Over $\C$, its analytification has the presentation
\[
\M_{g,n}^{\an}:=(\M_{g,n/\C})^{\an}
\simeq[\operatorname{Teich}_{g,n}/\G_{g,n}],
\]
where $\operatorname{Teich}_{g,n}$ is Teichm\"uller space. Its contractibility gives
\begin{equation}\label{eq:pionequalsgamma}
\pi_1^{\orb}(\M_{g,n}^{\an})\cong\G_{g,n}.
\end{equation}
We use $\pi_1$ without a superscript for the \'etale fundamental group.
For an algebraically closed field $k$ of characteristic zero,
Betti--\'etale comparison and invariance under algebraically closed
base change give
\[
\pi_1(\M_{g,n/k})\cong\widehat{\G_{g,n}}.
\]
Base points are omitted, with compatible choices understood.

\subsection{Moduli of curves of compact type}
The Deligne--Mumford compactification $\Mbar_{g,n}$ parametrizes
stable $n$-pointed curves of genus $g$.
Stability means that each genus-zero component of the
normalization has at least three special points and each genus-one
component has at least one, where special points are the marked
points and the branches of nodes.
% \authorcomment{
% I have formulated stability using special points to include
% rational tails whose stability depends on their marked points.
% These occur among the boundary divisors considered later.
% }
The stack $\Mbar_{g,n}$ is smooth over $\Z$, and its boundary is a
relative normal-crossings divisor.

A geometric stable curve is of \emph{compact type} if its dual graph
is a tree. Here the vertices correspond to irreducible components
and the edges to nodes. Equivalently, every node is separating, or
its generalized Jacobian is proper, hence an abelian variety.
% \authorcomment{
% I have used ``proper, hence an abelian variety'' to express the
% compactness criterion algebraically, so that it also applies in
% positive characteristic.
% }
The compact-type locus is an open substack
\[
\M_{g,n}\subseteq\M^c_{g,n}\subseteq\Mbar_{g,n}.
\]
Its boundary
$\partial\M^c_{g,n}:=\M^c_{g,n}\setminus\M_{g,n}$
consists of the separating boundary divisors, including those
parametrizing stable rational tails.

For a separating simple closed curve $\alpha$ disjoint from the
markings, write $T_\alpha$ for its Dehn twist. Define the
\emph{Johnson subgroup} by
\begin{equation}\label{eq:K}
K_{g,n}:=
\langle T_\alpha:
\alpha\subset\Sigma_{g,n}
\text{ is separating}\rangle
\unlhd\G_{g,n}.
\end{equation}
It is normal because conjugation carries a separating twist to a
separating twist, and it is contained in $T_{g,n}$.

A meridian around a separating boundary divisor corresponds under
\eqref{eq:pionequalsgamma} to the associated Dehn twist. Since these
meridians normally generate the kernel of the map on orbifold
fundamental groups, one obtains the following description.
\begin{theorem}[{Boggi \cite[Prop.~3.1]{boggi_compact_type}}]
\label{thm:main}
The inclusion $\M_{g,n}\hookrightarrow\M^c_{g,n}$ induces an exact
sequence
\begin{equation}\label{eq:main}
1\longrightarrow K_{g,n}\longrightarrow\G_{g,n}
\longrightarrow
\pi_1^{\orb}(\M^{c,\an}_{g,n})
\longrightarrow1.
\end{equation}
In particular,
\[
\pi_1^{\orb}(\M^{c,\an}_{g,n})
\cong\G_{g,n}/K_{g,n}.
\]
\end{theorem}
% \authorcomment{
% I have made the orbifold notation explicit, reserving $\pi_1$
% for the \'etale fundamental group, and used the analytic notation
% defined in the introduction. The statement and its labels are retained.
% }
\subsection{The relative Jacobian over compact type}
% \authorcomment{
% I have shortened the Torelli discussion to the statements needed
% for monodromy, omitting the general dimension and differential-sheaf
% formulas. The generalized Jacobian has dimension $g$ for every
% connected nodal curve of arithmetic genus $g$; properness is the
% criterion distinguishing compact type.
% }
Let $C$ be a geometric stable curve of compact type, and let
$C_1,\ldots,C_r$ be its irreducible components. These components
are smooth, and normalization induces an isomorphism
\[
J(C):=\Pic^0(C)\cong\prod_{i=1}^r J(C_i).
\]
The product of the canonical principal polarizations on
the factors gives the canonical principal polarization on $J(C)$.
% \authorcomment{
% I have used the product description of the polarization to avoid
% multidegree choices for a reducible curve. This also addresses the
% degree issue in the earlier theta-divisor description: already for
% a smooth curve, every degree-$g$ line bundle has a nonzero section,
% while the usual theta divisor lies in $\Pic^{g-1}$.
% }
% For a stable family $\pi:C\to S$ of compact-type curves,
% the relative Jacobian $\Pic^0_{C/S}$ is an abelian scheme of relative
% dimension $g$, equipped with a canonical principal polarization.
% This construction is compatible with base change;
% see \cite[\S1.3]{MoonenOort}.
For a stable family $\pi:C\to S$ of compact-type curves of genus $g$,
write
\[
\pi_J:\J_{C/S}:=\Pic^0_{C/S}\longrightarrow S
\]
for its relative Jacobian. 
This is an abelian scheme of relative dimension $g$, equipped
with a canonical principal polarization;
see \cite[\S1.3]{MoonenOort}.
Its formation commutes with arbitrary base change
\cite[Ex.~9.4.4 and Prop.~9.5.20]{KleimanPicard}.
% This is an abelian scheme of relative
% dimension $g$, equipped with a canonical principal polarization,
% and its formation commutes with base change;
% see \cite[\S1.3]{MoonenOort}.
The assignment $C/S\mapsto\J_{C/S}$, with its canonical principal
polarization, defines a morphism
\[
t^c_{g,n}:\M^c_{g,n}\longrightarrow\cA_g,
\]
where $\cA_g$ is the moduli stack of principally polarized abelian
varieties of dimension $g$. On the smooth locus, it is the usual
Torelli morphism after forgetting the markings.
% \authorcomment{
% I have replaced the gluing discussion with a direct citation of
% the result for families. This supplies the base-change compatibility
% needed later and avoids gluing over boundary strata, which form a
% stratification rather than an open cover.
% }
The relative Jacobian over $\M^c_{g,n}$ supplies the monodromy
representations used in Section~\ref{sec:relative-completion}.

\section{Relative pro-\texorpdfstring{$\ell$}{ell} Completion}
\label{sec:relative-pro-l-completion}
Fix a prime number $\ell$. Relative pro-$\ell$ completion is a
profinite construction that will be used below to compare discrete
mapping class groups, geometric fundamental groups, and specialization
homomorphisms. It should be distinguished from the continuous relative
completion introduced in the next section, which is a proalgebraic
construction over $\Q_\ell$.
Let $\Gamma$ be either a discrete group or a profinite group, let $P$
be a profinite group, and let $\rho:\Gamma\to P$ be a continuous
homomorphism with dense image. In applications, we take $P$ to be the
closure of the image of the given monodromy representation.
\begin{definition}\label{def:relative-pro-l-completion}
The \emph{relative pro-$\ell$ completion} of $\Gamma$ with respect to
$\rho$ is the inverse limit
\[
\Gamma^{\rel,\rho}:=\varprojlim_\phi G_\phi,
\]
where the limit is taken over all factorizations
$\Gamma\xrightarrow{\phi}G_\phi\xrightarrow{q_\phi}P$
such that $G_\phi$ is profinite, $\phi$ is continuous with dense image,
$q_\phi$ is continuous, $q_\phi\circ\phi=\rho$, and
$\ker q_\phi$ is pro-$\ell$. When $\rho$ is understood, we write
simply $\Gamma^\rel$.
\end{definition}
There are canonical continuous homomorphisms
$\Gamma\xrightarrow{\iota_\rho}\Gamma^{\rel,\rho}
\xrightarrow{\pi_\rho}P$ whose composite is $\rho$.
The image of $\iota_\rho$ is dense, and
\[
1\longrightarrow\ker\pi_\rho\longrightarrow\Gamma^{\rel,\rho}
\xrightarrow{\pi_\rho}P\longrightarrow1
\]
is exact, with $\ker\pi_\rho$ pro-$\ell$. Thus
$\Gamma^{\rel,\rho}$ need not itself be pro-$\ell$; rather, it is an
extension of $P$ by a pro-$\ell$ group. If $\Gamma$ is profinite, then
$\iota_\rho$ is surjective, since its image is both compact and dense.

The construction is characterized by the following universal property.
Suppose that $G$ is profinite, that $q:G\to P$ is continuous with
pro-$\ell$ kernel, and that $f:\Gamma\to G$ is a continuous lift of
$\rho$. Then there is a unique continuous homomorphism
$\widetilde f:\Gamma^{\rel,\rho}\to G$ such that
$\widetilde f\circ\iota_\rho=f$ and $q\circ\widetilde f=\pi_\rho$.
The lift $f$ is not required to have dense image. When $P=1$, relative
pro-$\ell$ completion is the usual pro-$\ell$ completion
$\Gamma^{(\ell)}$; for profinite $\Gamma$, this means its maximal
pro-$\ell$ quotient. See \cite[\S2.1]{HainMatsumoto2009}.

We record the formal properties that will be used below.
\begin{proposition}[{\cite[Prop.~2.1]{HainMatsumoto2009}}]
\label{prop:relative-pro-l-profinite-invariance}
Suppose that $\Gamma$ is discrete. The homomorphism $\rho$ extends
uniquely to a continuous homomorphism $\widehat\rho:\widehat\Gamma\to P$,
and the natural map $\Gamma\to\widehat\Gamma$ induces a canonical
isomorphism
\[
\Gamma^{\rel,\rho}\xrightarrow{\sim}
\widehat\Gamma^{\rel,\widehat\rho}.
\]
\end{proposition}
\begin{proposition}[{\cite[Prop.~2.3]{HainMatsumoto2009}}]
\label{prop:relative-pro-l-naturality}
For $i=1,2$, let $\rho_i:\Gamma_i\to P_i$ be a continuous homomorphism
with dense image, where $\Gamma_i$ is discrete or profinite and $P_i$
is profinite. Every commutative diagram
\[
\begin{tikzcd}
\Gamma_1\arrow[r,"f"]\arrow[d,"\rho_1"']&\Gamma_2\arrow[d,"\rho_2"]\\
P_1\arrow[r,"a"']&P_2
\end{tikzcd}
\]
induces a unique continuous homomorphism
$f^\rel:\Gamma_1^{\rel,\rho_1}\to\Gamma_2^{\rel,\rho_2}$
compatible with the canonical maps from $\Gamma_i$ and to $P_i$.
These induced homomorphisms are compatible with identities and composition.
\end{proposition}
\begin{lemma}[{\cite[Lem.~2.6]{HainMatsumoto2009}}]
\label{lem:relative-pro-l-change-of-target}
Suppose that $1\to L\to P\xrightarrow{q}\overline P\to1$ is an exact
sequence of profinite groups in which $L$ is pro-$\ell$, and put
$\overline\rho=q\circ\rho$. Then the natural homomorphism
$\Gamma^{\rel,\rho}\xrightarrow{\sim}\Gamma^{\rel,\overline\rho}$
is an isomorphism.
\end{lemma}
\begin{corollary}[Pro-$\ell$ targets]\label{cor:relative-pro-l-pro-l-target}
If $P$ is pro-$\ell$, then the natural homomorphism
$\Gamma^{\rel,\rho}\xrightarrow{\sim}\Gamma^{(\ell)}$ is an isomorphism.
\end{corollary}
\begin{proof}
Apply Lemma~\ref{lem:relative-pro-l-change-of-target} to $P\to1$.
\end{proof}
In particular, one may replace an $\ell$-adic monodromy target by its
mod-$\ell$ image whenever the kernel of reduction is pro-$\ell$; the
relative pro-$\ell$ completion is unchanged. This is the form of the
change-of-target result used in the specialization arguments below;
compare \cite[Example~2.7]{HainMatsumoto2009}.
\begin{proposition}[{\cite[Prop.~2.4]{HainMatsumoto2009}}]
\label{prop:relative-pro-l-right-exactness}
Suppose that the $\Gamma_i$ are either all discrete or all profinite,
that the $P_i$ are profinite, and that there is a commutative diagram
\[
\begin{tikzcd}
1\arrow[r]&\Gamma_1\arrow[r]\arrow[d,"\rho_1"']&
\Gamma_2\arrow[r]\arrow[d,"\rho_2"']&\Gamma_3\arrow[r]\arrow[d,"\rho_3"']&1\\
1\arrow[r]&P_1\arrow[r]&P_2\arrow[r]&P_3\arrow[r]&1
\end{tikzcd}
\]
with exact rows, where each $\rho_i$ has dense image. Then the induced
sequence
\[
\Gamma_1^{\rel,\rho_1}\to\Gamma_2^{\rel,\rho_2}\to
\Gamma_3^{\rel,\rho_3}\to1
\]
is exact.
\end{proposition}
\begin{corollary}[Quotients in the monodromy kernel]
\label{cor:relative-pro-l-quotient}
Let $N\triangleleft\Gamma$ be a normal subgroup contained in $\ker\rho$;
when $\Gamma$ is profinite, assume that $N$ is closed. Set
$\overline\Gamma=\Gamma/N$, and let $\overline\rho:\overline\Gamma\to P$
be the induced homomorphism. Then there is an exact sequence
\begin{equation}\label{eq:relative-pro-l-quotient}
N^{(\ell)}\longrightarrow\Gamma^{\rel,\rho}
\longrightarrow\overline\Gamma^{\rel,\overline\rho}\longrightarrow1.
\end{equation}
\end{corollary}
\begin{proof}
Apply Proposition~\ref{prop:relative-pro-l-right-exactness} to the
exact rows
\[
\begin{tikzcd}
1\arrow[r]&N\arrow[r]\arrow[d]&\Gamma\arrow[r]\arrow[d,"\rho"']&
\overline\Gamma\arrow[r]\arrow[d,"\overline\rho"']&1\\
1\arrow[r]&1\arrow[r]&P\arrow[r,equal]&P\arrow[r]&1.
\end{tikzcd}
\]
The relative pro-$\ell$ completion of $N\to1$ is $N^{(\ell)}$.
\end{proof}
For the mapping class group, put
\[
H_{\Z_\ell}:=H_1(\Sigma_g,\Z_\ell),
\]
with its intersection pairing, and let
$\rho_{g,n,\ell}:\G_{g,n}\to\Sp(H_{\Z_\ell})$ be the $\ell$-adic
symplectic representation. Put
$P_{g,n,\ell}:=\overline{\rho_{g,n,\ell}(\G_{g,n})}$.
We use $\G_{g,n}^\rel$ for completion relative to the resulting dense
homomorphism $\G_{g,n}\to P_{g,n,\ell}$, and similarly for its quotients.
Since $K_{g,n}\subseteq\ker\rho_{g,n,\ell}$,
Corollary~\ref{cor:relative-pro-l-quotient} gives
\begin{equation}\label{eq:relative-pro-l-compact-type-quotient}
K_{g,n}^{(\ell)}\longrightarrow\G_{g,n}^\rel
\longrightarrow(\G_{g,n}/K_{g,n})^\rel\longrightarrow1.
\end{equation}
Using Boggi's identification \cite[Prop.~3.1]{boggi_compact_type}, the
third term may equivalently be written
$\pi_1^{\orb}(\M^{c,\an}_{g,n})^\rel$.
Proposition~\ref{prop:relative-pro-l-profinite-invariance} then allows
us to pass to the profinite completion of this orbifold group without
changing its relative pro-$\ell$ completion.
\begin{remark}\label{rem:relative-pro-l-not-left-exact}
No initial $1$ is asserted in \eqref{eq:relative-pro-l-quotient}.
Relative pro-$\ell$ completion is right exact but not left exact in
general. Thus the first homomorphism need not be injective, and the
pro-$\ell$ kernel of $\Gamma^{\rel,\rho}\to P$ cannot in general be
identified with the ordinary pro-$\ell$ completion of $\ker\rho$
without an additional argument; see
\cite[the discussion following Prop.~2.4]{HainMatsumoto2009}.
In particular, the injectivity in
Lemma~\ref{lem:compact-type-relative-pro-l-exact} requires a separate
argument, supplied there by the finite-cover descent lemma.
\end{remark}

\section{(Continuous) Relative Completion}\label{sec:relative-completion}
We recall the properties of relative completion needed in this paper,
following Hain \cite[\S3]{hain_relwtfilt}, \cite[\S6]{hain_rational}, and
\cite[\S4]{hain_g3}. We then introduce the completions associated to the
smooth and compact-type moduli stacks. Their defining representations
are the monodromy representations of the relative Jacobian over
$\M^c_{g,n}$ and its restriction to $\M_{g,n}$.
Throughout the applications to moduli stacks, we assume that $g\geq3$
and $n\geq0$. All fundamental groups are based, although base points
will usually be omitted from the notation. Identifications involving
different base points are understood to use compatible choices of paths.

\subsection{Relative completion and its continuous version}
\label{subsec:relative-completion-definition}
Let $F$ be a field of characteristic zero, let $R$ be a reductive
algebraic group over $F$, and let $\rho:\Gamma\to R(F)$ be a
Zariski-dense homomorphism from a discrete group.
The \emph{completion of $\Gamma$ relative to $\rho$} consists of an
extension of proalgebraic $F$-groups
\[
1\longrightarrow\U\longrightarrow\cG
\overset{\pi_R}{\longrightarrow}R
\longrightarrow1,
\]
where $\U$ is prounipotent, together with a Zariski-dense homomorphism
$\widetilde\rho:\Gamma\to\cG(F)$ such that
$\pi_R\circ\widetilde\rho=\rho$. It is characterized by the following
universal property.
Suppose that
\[
1\to U'\to G'\overset{\pi_R'}\longrightarrow R\to1
\]
is an extension of proalgebraic $F$-groups with $U'$ prounipotent, and
that $\phi:\Gamma\to G'(F)$ satisfies
$\pi_R'\circ\phi=\rho$.
Then there is a unique morphism of proalgebraic $F$-groups
$\Phi:\cG\to G'$ such that
\[
\pi_R'\circ\Phi=\pi_R
\]
and
$\Phi(F)\circ\widetilde\rho=\phi$.
The group $\U$ is the prounipotent radical of $\cG$.
We write $\g:=\Lie\cG$ and $\u:=\Lie\U$.

Now let $\Gamma$ be a profinite group, let $\ell$ be a prime, and
suppose that $\rho:\Gamma\to R(\Ql)$ is continuous and Zariski dense,
where $R$ is reductive over $\Ql$. The \emph{continuous completion of
$\Gamma$ relative to $\rho$} is defined by the same universal property,
with $\widetilde\rho$ and all competing lifts $\phi$ required to be
continuous; see \cite[\S6.2]{hain_rational}. For a proalgebraic group
$G=\varprojlim_\alpha G_\alpha$ over $\Ql$, the group $G(\Ql)$ is
equipped with the inverse-limit topology induced by the $\ell$-adic
topologies on $G_\alpha(\Ql)$.
When $R$ is trivial, relative completion reduces to \emph{unipotent
completion}. In the profinite setting, it reduces to continuous
unipotent completion over $\Ql$.

\subsection{Functoriality and quotients}
\label{subsec:relative-completion-properties}
Relative completion is functorial in both the group and its reductive
representation \cite[Prop.~3.6]{hain_relwtfilt}. More precisely, suppose
that
\[
\begin{tikzcd}
\Gamma_1\arrow[r,"\rho_1"]\arrow[d,"\phi"']&R_1(F)\arrow[d,"\psi"]\\
\Gamma_2\arrow[r,"\rho_2"]&R_2(F)
\end{tikzcd}
\]
commutes, where $\psi:R_1\to R_2$ is a morphism of reductive groups and
both $\rho_i$ are Zariski dense. Then there is a unique induced morphism
$\Phi:\cG_1\to\cG_2$ compatible with the canonical lifts and with
$\psi$. The same statement holds for continuous relative completion
when $F=\Ql$ and the homomorphisms between profinite groups are
continuous. In either case, this follows directly from the universal
property, applied to the pullback of $\cG_2$ along $\psi$.
In particular, a section of a homomorphism of groups compatible with
the defining representations induces a section of the corresponding
morphism of relative completions.

\begin{proposition}[Quotients and relative completion]
\label{prop:relative-completion-quotient}
Suppose that $N\unlhd\Gamma$ is contained in $\ker\rho$, and let
$\cN$ be the closed normal proalgebraic subgroup of $\cG$ generated by
$\widetilde\rho(N)$. Let $\overline{\cG}$ denote the relative
completion of $\Gamma/N$ with respect to the representation induced
by $\rho$. Then $\cN\subseteq\U$, and the natural morphism
\[
\overline{\cG}\longrightarrow\cG/\cN
\]
is an isomorphism.

The same statement holds for continuous relative completion when
$\Gamma$ is profinite, $N$ is closed, and $\Gamma/N$ is equipped with
the quotient topology.
\end{proposition}
\begin{proof}
Since $N\subseteq\ker\rho$, its image lies in $\U$. As $\U$ is
closed and normal in $\cG$, we have $\cN\subseteq\U$. Thus
$\cG/\cN$ is an extension of $R$ by the prounipotent group
$\U/\cN$.

The composite $\Gamma\to\cG(F)\to(\cG/\cN)(F)$ kills $N$,
so it descends to a lift of the induced representation of
$\Gamma/N$. The universal property of $\overline{\cG}$ gives
a morphism
\[
\alpha:\overline{\cG}\longrightarrow\cG/\cN.
\]

Conversely, composing the canonical lift
$\Gamma/N\to\overline{\cG}(F)$ with the quotient map
$\Gamma\to\Gamma/N$ and applying the universal property of
$\cG$ gives a morphism $\cG\to\overline{\cG}$.
This morphism kills the image of $N$, hence the closed normal
subgroup $\cN$ generated by it, and therefore factors through
\[
\beta:\cG/\cN\longrightarrow\overline{\cG}.
\]

The composites $\beta\circ\alpha$ and $\alpha\circ\beta$ agree
with the identity on the respective Zariski-dense canonical
images of $\Gamma/N$. They are therefore identity morphisms,
so $\alpha$ is an isomorphism.

The same argument applies in the continuous setting, with
continuity of the descended lift following from the quotient
topology on $\Gamma/N$.
\end{proof}

A Levi splitting identifies $\cG\cong\U\rtimes R$.
Such a splitting is generally not canonical. However, the conjugation
action on $H_1(\u):=\u/\overline{[\u,\u]}$ factors canonically through
$R$, since $\U$ acts trivially on this quotient. Here the bar denotes
closure in the pronilpotent topology. Homomorphisms from a pro-vector
space to a finite-dimensional vector space are understood to be
continuous.
% The following degree-one comparison will be used to determine the
% abelianization of the prounipotent radical.
% \begin{proposition}[Degree-one comparison]\label{prop:relative-completion-h1}
% Let $\Gamma$ be finitely generated, let $R$ be reductive over $\Q$, and
% let $\rho:\Gamma\to R(\Q)$ be Zariski dense. Suppose its $\ell$-adic
% realization extends continuously to $\widehat\Gamma$, and let $\u$ be
% the prounipotent Lie algebra of the resulting continuous relative
% completion. For every finite-dimensional $R$-module $V$, there are
% natural isomorphisms
% \[
% \Hom_{R_{\Ql}}^{\cts}(H_1(\u),V_{\Ql})
% \cong H^1_{\cts}(\widehat\Gamma,V_{\Ql})
% \cong H^1(\Gamma,V)\otimes_\Q\Ql,
% \]
% where subscripts $\Ql$ denote scalar extension.
% If $\Gamma=\pi_1^{\orb}(Y^{\an})$ for a smooth connected
% Deligne--Mumford stack $Y/\C$ of finite type, these groups also identify
% with $H^1_{\et}(Y,\mathbb V_{\Ql})$, where $\mathbb V_{\Ql}$ is the
% associated lisse sheaf.
% \end{proposition}
% \begin{proof}
% Tensor Hain's rational degree-one comparison \cite[\S4.1]{hain_g3} with
% $\Ql$ and apply extension of scalars and profinite comparison
% \cite[Thms.~6.2--6.3]{hain_rational}. Here finite generation ensures
% that degree-one group cohomology commutes with scalar extension.
% The final assertion follows from the standard degree-one comparison
% with \'etale cohomology.
% \end{proof}

\subsection{Extension of scalars and profinite comparison}
\label{subsec:relative-completion-base-change}
We distinguish extension of the coefficient field from passage from a
discrete group to its profinite completion.
Let $\cG_F$ be the relative completion of $\rho:\Gamma\to R(F)$.
If $E/F$ is a field extension, then $\cG_F\times_F E$ is the relative
completion of $\Gamma$ with respect to
$\rho_E:\Gamma\to R(F)\to R(E)$; see
\cite[Thm.~6.2]{hain_rational} and \cite[\S4]{hain_g3}.
This identification is natural in $(\Gamma,\rho)$.
For inverse limits of finite-dimensional Lie algebras, scalar extension
is understood in the pro-vector-space sense:
\[
\mathfrak a\widehat\otimes_F E
:=\varprojlim_\alpha(\mathfrak a_\alpha\otimes_F E),
\qquad\mathfrak a=\varprojlim_\alpha\mathfrak a_\alpha.
\]
Thus scalar extension of the relative completion induces
\[\Lie(\cG_F\times_F E)\cong(\Lie\cG_F)\widehat\otimes_F E,\]
and similarly for its prounipotent radical.
The second comparison identifies an algebraic relative completion over
$\Ql$ with a continuous relative completion.
\begin{proposition}[{\cite[Thm.~6.3]{hain_rational}}]
\label{prop:continuous-completion-profinite}
Let $\Gamma$ be a finitely generated discrete group, let $R$ be reductive
over $\Ql$, and let $\rho:\Gamma\to R(\Ql)$ be Zariski dense.
Suppose that $\rho$ extends to a continuous homomorphism
$\widehat\rho:\widehat\Gamma\to R(\Ql)$ from the profinite completion
of $\Gamma$. Then the canonical lift
$\widetilde\rho:\Gamma\to\cG(\Ql)$ to its relative completion extends
continuously to $\widehat\Gamma$, and this extension makes $\cG$ the
continuous relative completion of $\widehat\Gamma$ with respect to
$\widehat\rho$.
\end{proposition}
Combining the two comparisons gives the form needed below.
\begin{corollary}\label{cor:rational-to-continuous-completion}
Let $R$ be reductive over $\Q$, let $\Gamma$ be finitely generated, and
let $\rho:\Gamma\to R(\Q)$ be Zariski dense. Suppose that the induced
representation in $R(\Ql)$ extends continuously to $\widehat\Gamma$.
If $\cG$ is the relative completion of $(\Gamma,\rho)$ over $\Q$, then
$\cG\times_\Q\Ql$ is the continuous relative completion of
$\widehat\Gamma$ with respect to the extended representation.
Its prounipotent radical and prounipotent Lie algebra are
$\U\times_\Q\Ql$ and $\u\widehat\otimes_\Q\Ql$, respectively.
\end{corollary}
\begin{proof}
Extension of scalars first identifies $\cG\times_\Q\Ql$ with the relative
completion of the discrete group $\Gamma$ over $\Ql$.
Proposition~\ref{prop:continuous-completion-profinite} then identifies
it with the continuous relative completion of $\widehat\Gamma$.
\end{proof}

\subsection{Monodromy of the relative Jacobian}
\label{subsec:jacobian-monodromy-completion}
Let $\pi^c:\cC^c_{g,n}\to\M^c_{g,n}$ be the universal curve
of compact type, and let
\[
\pi_J^c:\J^c_{g,n}:=\Pic^0_{\cC^c_{g,n}/\M^c_{g,n}}
\longrightarrow\M^c_{g,n}
\]
be its relative Jacobian, as in Section~\ref{sec:background}.
Write $j:\M_{g,n}\hookrightarrow\M^c_{g,n}$ for the inclusion
and $\pi:\cC_{g,n}\to\M_{g,n}$ for the universal smooth curve.
% Let $\pi^c:\cC^c_{g,n}\to\M^c_{g,n}$ be the universal curve of compact
% type, and let
% \[
% \pi_J^c:\J^c_{g,n}:=\Pic^0_{\cC^c_{g,n}/\M^c_{g,n}}\longrightarrow\M^c_{g,n}
% \]
% be its relative Jacobian. This is a principally polarized abelian scheme
% of relative dimension $g$, whose formation commutes with base change.
% Its restriction along $j:\M_{g,n}\hookrightarrow\M^c_{g,n}$ is the
% relative Jacobian of the universal smooth curve
% $\pi:\cC_{g,n}\to\M_{g,n}$.
Over $\C$, set
\[
\mathbb H^c_\Z:=R^1(\pi_J^{c,\an})_*\Z(1),
\qquad
\mathbb H^c_\Q:=\mathbb H^c_\Z\otimes_\Z\Q.
\]
Then $\mathbb H^c_\Q$ is a polarized variation of Hodge structure
of weight $-1$.
Over an algebraically closed field $k$, with
$\ell\neq\operatorname{char}(k)$, set
\[
\mathbb H^c_{\Zl}:=R^1(\pi_J^c)_*\Zl(1),
\qquad\mathbb H^c_{\Ql}:=\mathbb H^c_{\Zl}\otimes_{\Zl}\Ql.
\]
The principal polarization induces perfect alternating pairings
\[
\mathbb H^c_\Z\otimes_\Z\mathbb H^c_\Z\to\Z(1),
\qquad
\mathbb H^c_{\Zl}\otimes_{\Zl}\mathbb H^c_{\Zl}\to\Zl(1).
\]
The identification of the first cohomology of a curve with that of
its Jacobian is independent of the Abel--Jacobi base point and gives
canonical polarization-preserving isomorphisms
\[
\begin{aligned}
\mathbb H_\Z&:=j^*\mathbb H^c_\Z
\cong R^1(\pi^{\an})_*\Z(1),\\
\mathbb H_{\Zl}&:=j^*\mathbb H^c_{\Zl}
\cong R^1\pi_*\Zl(1).
\end{aligned}
\]
Write $\mathbb H_\Q:=\mathbb H_\Z\otimes_\Z\Q$.
Fix a point $\xi\in\M_{g,n}(\C)$ corresponding to a smooth curve
$C_\xi$, and choose a marking $C_\xi^{\an}\cong\Sigma_g$.
For an algebraically closed field $k$, fix a geometric point
$\bar\xi\to\M_{g,n/k}$ corresponding to a smooth curve $C_{\bar\xi}$.
Set
\[
H_\Q:=(\mathbb H^c_\Q)_\xi,
\quad H_{\Zl}:=(\mathbb H^c_{\Zl})_{\bar\xi},
\quad H_{\Ql}:=H_{\Zl}\otimes_{\Zl}\Ql.
\]
The preceding identifications and Poincar\'e duality give
\[
H_\Q\cong H^1(C_\xi^{\an},\Q(1))\cong H_1(\Sigma_g,\Q),
\qquad H_{\Zl}\cong H^1_{\et}(C_{\bar\xi},\Zl(1)).
\]
We write $H=H_\Q$ in the Betti setting and
$H=H_{\Ql}$ in the $\ell$-adic setting.
The relative Jacobian defines monodromy representations
\[
\rho^c_{g,n}:\pi_1^{\orb}(\M^{c,\an}_{g,n},\xi)\to\Sp(H_\Q),
\quad
\rho^c_{g,n/k}:\pi_1(\M^c_{g,n/k},\bar\xi)\to\Sp(H_{\Zl}).
\]
Since $k$ is algebraically closed, its geometric fundamental group acts
trivially on $\Zl(1)$, so the second representation preserves the
alternating pairing rather than merely its similitude class. We use
the same notation after extending coefficients to $\Ql$.
Denote the pullbacks to the smooth locus by $\rho_{g,n}$ and
$\rho_{g,n/k}$; thus
$\rho_{g,n}=\rho^c_{g,n}\circ j_*$ and
$\rho_{g,n/k}=\rho^c_{g,n/k}\circ j_*$.
Under the chosen marking, $\rho_{g,n}$ is the standard symplectic
representation of $\G_{g,n}$ on $H_1(\Sigma_g,\Q)$
\cite[\S6.1]{hain_rational}.
Under the identification
$\pi_1^{\orb}(\M^{c,\an}_{g,n},\xi)\cong\G_{g,n}/K_{g,n}$,
the homomorphism $j_*$ is the quotient map $q_{g,n}$. Consequently,
$\rho_{g,n}=\rho^c_{g,n}\circ q_{g,n}$.
Thus the monodromy of the relative Jacobian is the standard symplectic
representation factored through $\G_{g,n}/K_{g,n}$.
Put $H_\Z:=H_1(\Sigma_g,\Z)$. The standard integral representation of
$\G_{g,n}$ has image $\Sp(H_\Z)$ \cite[\S6.1]{hain_rational}.
Since $q_{g,n}$ is surjective,
\[
\im\rho^c_{g,n}=\im\rho_{g,n}=\Sp(H_\Z)\subseteq\Sp(H_\Q).
\]
In particular, $\rho^c_{g,n}$ has Zariski-dense image in $\Sp(H_\Q)$.
The open immersion $j$ also induces a surjection on \'etale fundamental
groups. Consequently, $\im\rho^c_{g,n/k}=\im\rho_{g,n/k}$.
When $\operatorname{char}(k)=0$, invariance under extension of
algebraically closed fields and Betti--\'etale comparison imply that
this common image is Zariski dense in $\Sp(H_{\Ql})$.
The positive-characteristic case is established in
Subsection~\ref{subsec:compact-type-continuous-specialization}.

\subsection{Relative completions over \texorpdfstring{$\Q$}{Q}}
\label{subsec:rational-moduli-completions}
We write $\cG_{g,n}$ for the relative completion over $\Q$ of
$\rho_{g,n}:\pi_1^{\orb}(\M_{g,n}^{\an})\cong\G_{g,n}\to\Sp(H_\Q)$,
and $\cG^c_{g,n}$ for the relative completion over $\Q$ of
$\rho^c_{g,n}:\pi_1^{\orb}(\M^{c,\an}_{g,n})
\cong\G_{g,n}/K_{g,n}\to\Sp(H_\Q)$.
Denote the corresponding canonical lifts by
\[
\widetilde\rho_{g,n}:\G_{g,n}\longrightarrow\cG_{g,n}(\Q),
\qquad
\widetilde\rho^c_{g,n}:\G_{g,n}/K_{g,n}
\longrightarrow\cG^c_{g,n}(\Q).
\]
Denote their prounipotent radicals by $\U_{g,n}$ and $\U^c_{g,n}$,
respectively, and set
\[
\begin{aligned}
\g_{g,n}&:=\Lie\cG_{g,n},&\u_{g,n}&:=\Lie\U_{g,n},\\
\g^c_{g,n}&:=\Lie\cG^c_{g,n},&\u^c_{g,n}&:=\Lie\U^c_{g,n}.
\end{aligned}
\]
The open immersion $j$ induces a commutative diagram with exact rows
\[
\begin{tikzcd}
1\arrow[r]&\U_{g,n}\arrow[r]\arrow[d,two heads]&
\cG_{g,n}\arrow[r]\arrow[d,two heads,"\widetilde q_{g,n}"]&
\Sp(H_\Q)\arrow[r]\arrow[d,equal]&1\\
1\arrow[r]&\U^c_{g,n}\arrow[r]&\cG^c_{g,n}\arrow[r]&\Sp(H_\Q)\arrow[r]&1.
\end{tikzcd}
\]
The surjectivity follows from the quotient property of relative
completion. The kernel of $\widetilde q_{g,n}$ is studied in
Section~\ref{sec:compact-type-relative-completion}.
All forgetful morphisms obtained by forgetting marked points are
compatible with these representations: stabilizing a compact-type
curve after forgetting markings does not change its Jacobian.
They therefore induce the corresponding morphisms of relative
completions and their Lie algebras.
The polarized variations $\mathbb H_\Q$ and $\mathbb H^c_\Q$
equip the relative completions $\cG_{g,n}$ and $\cG^c_{g,n}$
over $\Q$, respectively, with Hain's Hodge-theoretic structures;
see \cite[\S13]{hain_hodge_rel} and
\cite[Thm.~4.2 and Cor.~4.3]{hain_g3}.
For $g\geq3$, the weight filtration on $\u_{g,n}$ is its lower central
series:
\[
W_{-m}\u_{g,n}=L^m\u_{g,n},\quad m\geq1,
\qquad L^1\u_{g,n}=\u_{g,n},\quad
L^{m+1}\u_{g,n}=\overline{[\u_{g,n},L^m\u_{g,n}]}.
\]
This is \cite[Cor.~4.8]{hain_infpretor}.
We use the quotient weight filtration on $\u^c_{g,n}$, as discussed
in Section~\ref{sec:compact-type-relative-completion}.

\subsection{Continuous relative completions of geometric fundamental groups}
\label{subsec:continuous-moduli-completions}
Let $k$ be an algebraically closed field and let
$\ell\neq\operatorname{char}(k)$. Using the Jacobian monodromy
representations, define $\cG_{g,n/k}$ and $\cG^c_{g,n/k}$ to be the
continuous relative completions over $\Ql$ of
\[
\rho_{g,n/k}:\pi_1(\M_{g,n/k},\bar\xi)\to\Sp(H_{\Ql}),
\qquad
\rho^c_{g,n/k}:\pi_1(\M^c_{g,n/k},\bar\xi)\to\Sp(H_{\Ql}),
\]
respectively. Their prounipotent radicals are denoted by $\U_{g,n/k}$
and $\U^c_{g,n/k}$. We use the notation
\[
\begin{aligned}
\g_{g,n/k}&:=\Lie\cG_{g,n/k},&\u_{g,n/k}&:=\Lie\U_{g,n/k},\\
\g^c_{g,n/k}&:=\Lie\cG^c_{g,n/k},&\u^c_{g,n/k}&:=\Lie\U^c_{g,n/k}.
\end{aligned}
\]
Thus there are exact sequences
$1\to\U_{g,n/k}\to\cG_{g,n/k}\to\Sp(H_{\Ql})\to1$ and
$1\to\U^c_{g,n/k}\to\cG^c_{g,n/k}\to\Sp(H_{\Ql})\to1$.
The field $k$ in the subscript specifies the geometric fundamental
group being completed. The coefficient field of each completion is $\Ql$.
Functoriality gives surjective morphisms
$\cG_{g,n/k}\to\cG^c_{g,n/k}$ and $\u_{g,n/k}\to\u^c_{g,n/k}$,
compatible with the identity on $\Sp(H_{\Ql})$.

Suppose now that $\operatorname{char}(k)=0$.
The characteristic-zero comparisons recalled above identify
\[
\pi_1(\M_{g,n/k})\cong\widehat\G_{g,n},
\qquad\pi_1(\M^c_{g,n/k})\cong\widehat{\G_{g,n}/K_{g,n}},
\]
compatibly with the Jacobian monodromy representations. For a general
algebraically closed field $k$ of characteristic zero, these
identifications may be made by choosing embeddings of $\overline\Q$
into $k$ and $\C$ and using invariance under extension of algebraically
closed fields.
Since the monodromy representations preserve the integral lattice
$H_\Z$, their $\ell$-adic realizations take values in $\Sp(H_{\Zl})$
and therefore extend uniquely and continuously to the profinite
completions of $\G_{g,n}$ and $\G_{g,n}/K_{g,n}$. Extending scalars
gives the corresponding $\Ql$-valued representations.
Corollary~\ref{cor:rational-to-continuous-completion} therefore gives
a commutative diagram
\[
\begin{tikzcd}
\cG_{g,n}\times_\Q\Ql\arrow[r]\arrow[d,"\sim"']&
\cG^c_{g,n}\times_\Q\Ql\arrow[d,"\sim"]\\
\cG_{g,n/k}\arrow[r]&\cG^c_{g,n/k}.
\end{tikzcd}
\]
All arrows are compatible with the identified reductive quotient
$\Sp(H_{\Ql})$. In particular,
\[
\u_{g,n/k}\cong\u_{g,n}\widehat\otimes_\Q\Ql,
\qquad\u^c_{g,n/k}\cong\u^c_{g,n}\widehat\otimes_\Q\Ql.
\]
We equip these Lie algebras with the weight filtrations obtained by
scalar extension of their rational counterparts. For example,
$\Gr^W_m\u_{g,n/k}\cong\Gr^W_m\u_{g,n}\otimes_\Q\Ql$.
These identifications are compatible with the morphisms induced by the
open immersion and by forgetting marked points.
The comparison with positive characteristic requires the specialization
argument of Section~\ref{sec:specialization-characteristic-p}.
It is distinct from extension of the coefficient field $\Q\subset\Ql$.

\subsection{Unipotent completions of fibers and configuration spaces}
\label{subsec:fiber-configuration-completions}
Let $C=C_\xi$ be the smooth complex reference curve fixed above.
For $n\geq1$, let
\[
F_n(C):=\{(z_1,\ldots,z_n)\in C^n:z_i\neq z_j\text{ whenever }i\neq j\}
\]
be its ordered configuration space. Choose $\mathbf x\in F_n(C)$ and set
\[
\cP_{g,n}:=(\pi_1^{\top}(F_n(C)^{\an},\mathbf x))^{\un}_{/\Q},
\qquad\p_{g,n}:=\Lie\cP_{g,n}.
\]
Thus $\cP_{g,n}$ is a prounipotent $\Q$-group and $\p_{g,n}$ is a
pronilpotent $\Q$-Lie algebra. For $n=1$, we abbreviate
$\cP:=\cP_{g,1}$ and $\p:=\p_{g,1}$.
We equip $\p$ with its natural weight filtration, which agrees with
its lower central series:
\[
W_{-m}\p=L^m\p,\qquad m\geq1.
\]
We write $\Der\p$ for the Lie algebra of continuous derivations of
the pronilpotent Lie algebra $\p$. It is equipped with the induced
weight filtration
\[
W_m\Der\p
:=
\left\{
D\in\Der\p:
D(W_r\p)\subseteq W_{m+r}\p
\text{ for all }r
\right\}.
\]
Thus an element of
\[
\Gr^W_m\Der\p
=W_m\Der\p/W_{m-1}\Der\p
\]
induces a derivation of degree $m$ on $\Gr^W_\bullet\p$.

In particular, $H_1(\p)\cong H_1(C^{\an},\Q)\cong H_\Q$.
The forgetful morphism $\M_{g,n}\to\M_g$ has fiber $F_n(C)$.
For $g\geq2$, right exactness of relative completion, together with the
extension of the conjugation action \cite[Lem.~3.8]{hain_infpretor}
and the center-freeness of $\cP_{g,n}$, gives exact sequences
\[
1\to\cP_{g,n}\to\cG_{g,n}\to\cG_g\to1,
\qquad0\to\p_{g,n}\to\u_{g,n}\to\u_g\to0.
\]
Let $\bar{\mathbf x}$ denote the geometric base point induced by
$\mathbf x$. Corollary~\ref{cor:rational-to-continuous-completion},
applied with trivial reductive quotient, gives a canonical
isomorphism
\[
\cP_{g,n}\times_\Q\Ql
\cong
(\pi_1^{\et}(F_n(C),\bar{\mathbf x}))^{\un}_{/\Ql}.
\]
Taking Lie algebras gives the corresponding comparison for
$\p_{g,n}\widehat\otimes_\Q\Ql$.
For $n=1$, we write $\cP_{\Ql}:=\cP\times_\Q\Ql$.
% For comparison with continuous unipotent completions over $\Ql$, put
% \[
% \cP_{g,n,\Ql}:=\cP_{g,n}\times_\Q\Ql,
% \qquad\p_{g,n,\Ql}:=\Lie\cP_{g,n,\Ql}
% =\p_{g,n}\widehat\otimes_\Q\Ql,
% \]
% and, for $n=1$, write $\cP_{\Ql}:=\cP_{g,1,\Ql}$ and
% $\p_{\Ql}:=\p_{g,1,\Ql}$. Let $\bar{\mathbf x}$ denote the geometric
% base point induced by $\mathbf x$.
% Corollary~\ref{cor:rational-to-continuous-completion}, applied with
% trivial reductive quotient, gives canonical isomorphisms
% \[
% \cP_{g,n,\Ql}\cong(\pi_1^{\et}(F_n(C),\bar{\mathbf x}))^{\un}_{/\Ql},
% \qquad
% \p_{g,n,\Ql}\cong\Lie((\pi_1^{\et}(F_n(C),\bar{\mathbf x}))^{\un}_{/\Ql}).
% \]
% For $n=1$, these specialize to
% \[
% \cP_{\Ql}\cong(\pi_1^{\et}(C,\bar x))^{\un}_{/\Ql},
% \qquad\p_{\Ql}\cong\Lie((\pi_1^{\et}(C,\bar x))^{\un}_{/\Ql}).
% \]

\subsection{Passage through relative pro-\texorpdfstring{$\ell$}{ell} completion}
\label{subsec:continuous-relative-pro-l}
We finally record the comparison that permits the use of relative
pro-$\ell$ specialization in the continuous proalgebraic setting.
We give the argument from the universal properties.
\begin{lemma}\label{lem:compact-subgroups-prounipotent}
Every compact subgroup of the $\Ql$-points of a prounipotent group
is a pro-$\ell$ group.
\end{lemma}
\begin{proof}
First let $U$ be a unipotent algebraic group over $\Ql$.
It has a finite normal series whose successive quotients are vector
groups. A compact subgroup of a vector group over $\Ql$ is contained
in an additive $\Zl$-lattice and is therefore pro-$\ell$.
Intersecting a compact subgroup of $U(\Ql)$ with the normal series and
arguing by induction proves the assertion for $U$.
For a prounipotent group, project to its unipotent algebraic quotients.
A compact subgroup is the inverse limit of its compact images in these
quotients, and is consequently pro-$\ell$.
\end{proof}
\begin{proposition}\label{prop:continuous-completion-relative-pro-l}
Let $R$ be a reductive group over $\Ql$ with a $\Zl$-model, and let
$\rho:\Gamma\to R(\Zl)$ be a continuous homomorphism from a profinite
group whose image is Zariski dense in $R$. Let $\Gamma^{\rel,\rho}$
be its relative pro-$\ell$ completion. Then the canonical homomorphism
$\Gamma\to\Gamma^{\rel,\rho}$ induces an isomorphism on continuous
relative completions over $\Ql$.
\end{proposition}
\begin{proof}
Put $P:=\rho(\Gamma)$, which is compact, and write
$\Gamma^{\mathrm{rel}}:=\Gamma^{\rel,\rho}$.
Let $\cG$ be the continuous relative completion of $\Gamma$ with respect
to $\rho$, and let $\U$ be its prounipotent radical.
The image $B:=\widetilde\rho(\Gamma)\subseteq\cG(\Ql)$ is compact.
Its projection to $R(\Ql)$ has image $P$, and there is an exact sequence
\[
1\to B\cap\U(\Ql)\to B\to P\to1.
\]
The kernel is pro-$\ell$ by Lemma~\ref{lem:compact-subgroups-prounipotent}.
The universal property of relative pro-$\ell$ completion therefore
factors the canonical lift as
$\Gamma\to\Gamma^{\mathrm{rel}}\to B\hookrightarrow\cG(\Ql)$.
Let $\cG^{\mathrm{rel}}$ denote the continuous relative completion of
$\Gamma^{\mathrm{rel}}$ with respect to its induced representation in
$R(\Ql)$. The canonical map $\Gamma\to\Gamma^{\mathrm{rel}}$ induces
$\cG\to\cG^{\mathrm{rel}}$, whereas the factorization just constructed
induces $\cG^{\mathrm{rel}}\to\cG$.
Both composites preserve the canonical lifts. Uniqueness in the
universal properties shows that they are the identity morphisms.
\end{proof}
Consequently, the continuous relative completions of
$\pi_1(\M_{g,n/k})$ and $\pi_1(\M^c_{g,n/k})$ may be computed from their
relative pro-$\ell$ completions with respect to Jacobian monodromy.
In particular, an isomorphism between such relative pro-$\ell$
completions that intertwines the monodromy representations induces an
isomorphism of their continuous relative completions, their prounipotent
radicals, and their Lie algebras. This is the passage used in
Section~\ref{sec:specialization-characteristic-p}.

\section{The Johnson Subgroup and Separating Twists}
\label{sec:johnson-subgroup}
In this section we determine the position of the Johnson subgroup
$K_{g,n}$ in the weight filtration on $\u_{g,n}$. The main result is that the logarithm of every element of
$K_{g,n}$ lies in $W_{-2}\u_{g,n}$. This allows us to associate a
weight $-2$ initial form to every element of $K_{g,n}$.
\subsection{The Johnson subgroup}\label{subsec:johnson-subgroup}
Recall from \eqref{eq:K} that the Johnson subgroup
$K_{g,n}\subset T_{g,n}$ is generated by Dehn twists about separating
simple closed curves. By Theorem~\ref{thm:main},
\[
\pi_1^{\orb}(\M^{c,\an}_{g,n})
\cong\G_{g,n}/K_{g,n}.
\]
The corresponding statement for relative completions is considered in
Section~\ref{sec:compact-type-relative-completion}.
Let $\Pi_{g,n}:=\ker\{\G_{g,n}\to\G_g\}$ be the pure surface braid
subgroup obtained by forgetting all marked points.
% \subsection{The Johnson subgroup}\label{subsec:johnson-subgroup}
% Recall from \eqref{eq:K} that the Johnson subgroup
% $K_{g,n}\subset T_{g,n}$ is generated by Dehn twists about separating
% simple closed curves. By Theorem~\ref{thm:main},
% \[
% \pi_1^{\orb}(\M^{c,\an}_{g,n})
% \cong\G_{g,n}/K_{g,n}.
% \]
% The corresponding statement for relative completions is considered in
% Section~\ref{sec:compact-type-relative-completion}.
% Let $\Pi_{g,n}:=\ker\{\G_{g,n}\to\G_g\}$ be the pure surface braid
% subgroup obtained by forgetting all marked points.
% The forgetful homomorphism restricts to a surjection $K_{g,n}\to K_g$.
% Its kernel is $\Pi_{g,n}\cap K_{g,n}$, so there is an exact sequence
% \[
% 1\to\Pi_{g,n}\cap K_{g,n}\to K_{g,n}\to K_g\to1.
% \]
% This separates the contribution already present in the unpointed mapping
% class group from the additional contributions arising from the marked
% points.

\subsection{Separating twists and the weight filtration}
\label{subsec:separating-twists-weight}
Restricting the canonical homomorphism
$\widetilde\rho_{g,n}:\G_{g,n}\to\cG_{g,n}(\Q)$ to the Torelli group
gives a homomorphism $T_{g,n}\to\U_{g,n}(\Q)$, since
$T_{g,n}=\ker\rho_{g,n}$. Composing with the logarithm and then
projecting to the abelianization gives the canonical homomorphism
\[
T_{g,n}\longrightarrow H_1(\u_{g,n}),\qquad
D\longmapsto\log\widetilde\rho_{g,n}(D)
\pmod{\overline{[\u_{g,n},\u_{g,n}]}}.
\]
It is equivariant for the $\G_{g,n}$-action, whose action on the target
factors through $\Sp(H)$.
For $g\geq3$, the weight filtration on $\u_{g,n}$ agrees with its lower
central series \cite[Cor.~4.8]{hain_infpretor}. In particular,
$H_1(\u_{g,n})\cong\Gr^W_{-1}\u_{g,n}$.
Hain identifies the first homology as an $\Sp(H)$-module
\cite[Prop.~6.2(ii)--(iii)]{hain_g3}:
\[
H_1(\u_{g,n})\cong V(\lambda_3)\oplus V(\lambda_1)^{\oplus n}
\cong\Lambda^3_0H\oplus H^{\oplus n}.
\]
This is also the description given in \cite[Thm.~9.11]{hain_rational}.
Under these identifications, the homomorphism
$T_{g,n}\to H_1(\u_{g,n})$ is the rational Johnson homomorphism together
with its marked-point components; see \cite[\S\S11--13]{hain_infpretor}.
\begin{proposition}\label{prop:Kgn-weight-minus-two}
Assume $g\geq3$. Then
$\log\widetilde\rho_{g,n}(K_{g,n})\subset W_{-2}\u_{g,n}$.
In particular, if $\alpha$ is a separating simple closed curve, then
$\log\widetilde\rho_{g,n}(T_\alpha)\in W_{-2}\u_{g,n}$.
\end{proposition}
\begin{proof}
For $m=0,1$, Hain's adjoint construction realizes the rational
Johnson homomorphisms as the composites
\[
\begin{aligned}
T_{g,1}&\longrightarrow\Gr^W_{-1}\u_{g,1}
\xrightarrow{\Gr^W_{-1}\ad}\Gr^W_{-1}\Der\p,\\
T_g&\longrightarrow\Gr^W_{-1}\u_g
\longrightarrow\Gr^W_{-1}\OutDer\p,
\end{aligned}
\]
where the first arrows send $D$ to
$\log\widetilde\rho_{g,m}(D)\pmod{W_{-2}\u_{g,m}}$ and
$\OutDer\p:=\Der\p/\ad(\p)$.
The second arrows are isomorphisms: the unpointed one follows
from the pointed one by quotienting out the point-pushing and
inner-derivation summands; see \cite[\S11]{hain_infpretor}
and \cite[Remark~8.5 and Thm.~9.7]{hain_rational}.
Boggi proves that the corresponding Johnson kernels are
$K_{g,m}$ \cite[Thm.~3.15 and the discussion following
Def.~3.17]{boggi_compact_type}.
Injectivity of the second arrows therefore gives
$\log\widetilde\rho_{g,m}(K_{g,m})\subset W_{-2}\u_{g,m}$
for $m=0,1$.
% For $m=0,1$, Hain's adjoint construction identifies the homomorphism
% \[
% T_{g,m}\longrightarrow\Gr^W_{-1}\u_{g,m},\quad
% D\longmapsto\log\widetilde\rho_{g,m}(D)\pmod{W_{-2}\u_{g,m}},
% \]
% with the rational Johnson homomorphism; see \cite[\S11]{hain_infpretor}
% and \cite[Remark~8.5]{hain_rational}. Boggi proves that the kernel of
% the corresponding Johnson homomorphism is $K_{g,m}$
% \cite[Thm.~3.15 and the discussion following Def.~3.17]{boggi_compact_type}.
% Therefore $\log\widetilde\rho_{g,m}(K_{g,m})\subset W_{-2}\u_{g,m}$
% for $m=0,1$.
Suppose now that $n\geq2$. For each $1\leq j\leq n$, forgetting all
marked points except $x_j$ induces homomorphisms
$F_j:T_{g,n}\to T_{g,1}$ and $f_j:\u_{g,n}\to\u_{g,1}$.
A separating twist becomes either a separating twist or the identity
after marked points are forgotten. Hence $F_j(K_{g,n})\subset K_{g,1}$.
Let $D\in K_{g,n}$. By the case $n=1$ and functoriality of relative
completion, the image of
$\log\widetilde\rho_{g,n}(D)\pmod{W_{-2}\u_{g,n}}$
under $\Gr^W_{-1}f_j$ is zero for every $j$.
The homomorphism
\[
\Gr^W_{-1}\u_{g,n}\longrightarrow
(\Gr^W_{-1}\u_{g,1})^{\oplus n}
\]
induced by the $f_j$ is injective \cite[Remark~8.5]{hain_rational}.
It follows that $\log\widetilde\rho_{g,n}(D)\equiv0\pmod{W_{-2}\u_{g,n}}$,
which proves the assertion.
\end{proof}
Proposition~\ref{prop:Kgn-weight-minus-two} allows us to associate a
weight $-2$ initial form to every element of $K_{g,n}$.
\begin{definition}
For $D\in K_{g,n}$, define
\[
\operatorname{in}_{-2}(D):=\log\widetilde\rho_{g,n}(D)
\pmod{W_{-3}\u_{g,n}}\in\Gr^W_{-2}\u_{g,n}.
\]
\end{definition}
The weight $-2$ initial forms of elements of $K_{g,n}$ will be studied
in Section~\ref{sec:image-Kgn}.

\section{The Compact-Type Relative Completion in Characteristic Zero}
\label{sec:compact-type-relative-completion}
We first identify the kernel of the natural morphism of relative
completions. We then use Boggi's description of $T_{g,n}/K_{g,n}$
to prove that the prounipotent radical of the compact-type relative
completion is commutative.

\subsection{The compact-type completion map}
\label{subsec:compact-type-completion-map}
Boggi's description of the orbifold fundamental group gives
\cite[Prop.~3.1]{boggi_compact_type}
\[
\pi_1^{\orb}(\M^{c,\an}_{g,n})\cong\G_{g,n}/K_{g,n}.
\]
As in Subsection~\ref{subsec:rational-moduli-completions},
the quotient $q_{g,n}:\G_{g,n}\to\G_{g,n}/K_{g,n}$ induces
the morphism over $\Sp(H)$
\[
\widetilde q_{g,n}:\cG_{g,n}\longrightarrow\cG^c_{g,n},
\]
with induced homomorphisms
$\U_{g,n}\to\U^c_{g,n}$ and $\u_{g,n}\to\u^c_{g,n}$.
% Since $K_{g,n}\subset\ker\rho_{g,n}$, the symplectic representation
% factors through the quotient:
% \[
% \begin{tikzcd}
% \G_{g,n}\arrow[r,"q_{g,n}"]\arrow[dr,"\rho_{g,n}"']&
% \G_{g,n}/K_{g,n}\arrow[d]\\
% &\Sp(H).
% \end{tikzcd}
% \]
% Functoriality of relative completion therefore gives a morphism over
% $\Sp(H)$
% \[
% \widetilde q_{g,n}:\cG_{g,n}\longrightarrow\cG^c_{g,n}.
% \]
% It restricts to a homomorphism of prounipotent radicals
% $\U_{g,n}\to\U^c_{g,n}$ and induces a homomorphism of their Lie
% algebras $\u_{g,n}\to\u^c_{g,n}$.

\begin{proposition}\label{prop:kernel-generated-by-K}
Let $\cN_{g,n}\subseteq\cG_{g,n}$ be the closed normal
proalgebraic subgroup generated by
$\widetilde\rho_{g,n}(K_{g,n})$, and set
\[
\n_{g,n}:=\Lie\cN_{g,n}.
\]
Then $\cN_{g,n}\subseteq\U_{g,n}$ and
\[
\ker\widetilde q_{g,n}=\cN_{g,n},
\qquad
\cG^c_{g,n}\cong\cG_{g,n}/\cN_{g,n}.
\]
Moreover, $\n_{g,n}$ is the closed Lie ideal of $\u_{g,n}$
generated by
$\log\widetilde\rho_{g,n}(K_{g,n})$, and
\[
\u^c_{g,n}\cong\u_{g,n}/\n_{g,n}.
\]
\end{proposition}
\begin{proof}
Apply Proposition~\ref{prop:relative-completion-quotient} to
$q_{g,n}:\G_{g,n}\to\G_{g,n}/K_{g,n}$.
It gives $\cN_{g,n}\subseteq\U_{g,n}$ and a canonical isomorphism
\[
\cG^c_{g,n}\cong\cG_{g,n}/\cN_{g,n}.
\]
Under this isomorphism, $\widetilde q_{g,n}$ is the quotient map, so
$\ker\widetilde q_{g,n}=\cN_{g,n}$.

Since $\cN_{g,n}$ is prounipotent, the exponential--logarithm
correspondence identifies
$\n_{g,n}=\Lie\cN_{g,n}$ with the closed Lie ideal of $\u_{g,n}$
generated by $\log\widetilde\rho_{g,n}(K_{g,n})$.
Taking Lie algebras gives
\(
\u^c_{g,n}\cong\u_{g,n}/\n_{g,n}.
\)
\end{proof}
We equip $\u^c_{g,n}$ with the quotient weight filtration.
The projection from $\u_{g,n}$ is strict for this filtration.
Since its kernel is closed and $W_\bullet\u_{g,n}$ is the complete
lower central series filtration, the quotient filtration is complete
and separated.

\subsection{The weight \texorpdfstring{$-1$}{-1} quotient}
\label{subsec:compact-type-weight-minus-one}
The computation of Section~\ref{sec:johnson-subgroup} shows that the
relations imposed by $K_{g,n}$ begin in weight $-2$.
\begin{corollary}\label{cor:weight-minus-one-unchanged}
Assume $g\geq3$. The natural homomorphism $\u_{g,n}\to\u^c_{g,n}$
induces an isomorphism
\[
\Gr^W_{-1}\u_{g,n}\xrightarrow{\sim}\Gr^W_{-1}\u^c_{g,n}.
\]
Consequently,
\[
\Gr^W_{-1}\u^c_{g,n}\cong V(\lambda_3)\oplus V(\lambda_1)^{\oplus n}
\cong\Lambda^3_0H\oplus H^{\oplus n}.
\]
\end{corollary}
\begin{proof}
Propositions~\ref{prop:Kgn-weight-minus-two}
and~\ref{prop:kernel-generated-by-K} give
$\n_{g,n}\subseteq W_{-2}\u_{g,n}$, since
$W_{-2}\u_{g,n}$ is a closed Lie ideal.
The quotient filtration on
$\u^c_{g,n}\cong\u_{g,n}/\n_{g,n}$ therefore gives the asserted
isomorphism in weight $-1$.
The module decomposition follows from the computation of
$H_1(\u_{g,n})$ recalled in
Subsection~\ref{subsec:separating-twists-weight}.
\end{proof}
% \begin{proof}
% By Proposition~\ref{prop:kernel-generated-by-K},
% $\n_{g,n}$ is the closed Lie ideal generated by
% $\log\widetilde\rho_{g,n}(K_{g,n})$. Proposition~\ref{prop:Kgn-weight-minus-two}
% gives
% \[
% \log\widetilde\rho_{g,n}(K_{g,n})
% \subseteq W_{-2}\u_{g,n}.
% \]
% Since $W_{-2}\u_{g,n}$ is a closed Lie ideal of $\u_{g,n}$, it follows
% that
% \[
% \n_{g,n}\subseteq W_{-2}\u_{g,n}.
% \]
% Thus the kernel of $\u_{g,n}\to\u^c_{g,n}$ has no weight $-1$
% quotient, and therefore
% \[
% \Gr^W_{-1}\u_{g,n}\xrightarrow{\sim}\Gr^W_{-1}\u^c_{g,n}.
% \]
% The final identification follows from the computation of
% $H_1(\u_{g,n})$ recalled in Section~\ref{sec:johnson-subgroup}.
% \end{proof}

\subsection{The compact-type relative completion}
\label{subsec:compact-type-relative-completion}
Boggi proves that $T_{g,n}/K_{g,n}$ is a free abelian group
\cite[Thm.~3.18]{boggi_compact_type}. We use this to prove the following.
\begin{theorem}[Compact-type relative completion]
\label{thm:compact-type-completion-char-zero}
Assume $g\geq3$. Then $\Gr^W_m\u^c_{g,n}=0$ for $m\leq-2$.
Consequently, $\u^c_{g,n}$ is abelian, and
\[
\u^c_{g,n}\cong\Gr^W_{-1}\u_{g,n}
\cong V(\lambda_3)\oplus V(\lambda_1)^{\oplus n}
\cong\Lambda^3_0H\oplus H^{\oplus n}.
\]
\end{theorem}
% REPLACED PROOF: the compact-type abelianness is proved from T/K.
\begin{proof}
By \cite[Thm.~6.1]{hain_g3}, the unipotent completion of
$T_{g,n}$ maps surjectively to $\U_{g,n}$.
Since $T_{g,n}$ has Zariski-dense image in its unipotent completion,
its images in $\U_{g,n}$ and the quotient $\U^c_{g,n}$ are
Zariski dense.
The latter image factors through the abelian group
$T_{g,n}/K_{g,n}$.
Let $U$ be any algebraic quotient of $\U^c_{g,n}$.
The image of $T_{g,n}$ in $U$ is abelian and Zariski dense.
The commutator morphism
\[
U\times U\longrightarrow U,\qquad
(u,v)\longmapsto uvu^{-1}v^{-1}
\]
has a closed fiber over $1$ containing the Zariski-dense image
of $T_{g,n}\times T_{g,n}$. This fiber is therefore $U\times U$,
so $U$ is commutative. Hence $\U^c_{g,n}$ is commutative.
% Denote the unipotent completion of $T_{g,n}$ by $\mathcal T_{g,n}^{\un}$.
% By \cite[Thm.~6.1]{hain_g3}, the homomorphism
% \[
% \mathcal T_{g,n}^{\un}\longrightarrow\U_{g,n}
% \]
% is surjective. The image of $T_{g,n}$ is Zariski dense in its
% unipotent completion, and hence in $\U_{g,n}$.
% Since $\U_{g,n}\to\U^c_{g,n}$ is surjective, the image of $T_{g,n}$
% in $\U^c_{g,n}$ is also Zariski dense.
% This image factors through the abelian group $T_{g,n}/K_{g,n}$.
% Let $U$ be any algebraic quotient of $\U^c_{g,n}$, and denote the
% image of $T_{g,n}$ in $U$ by $A$. Then $A$ is an abelian
% Zariski-dense subgroup of $U$. Consider the commutator morphism
% \[
% c_U:U\times U\longrightarrow U,\qquad
% (u,v)\longmapsto uvu^{-1}v^{-1}.
% \]
% Since $A$ is abelian, the closed subset $c_U^{-1}(1)$ contains
% $A\times A$, which is Zariski dense in $U\times U$.
% Thus $c_U$ is identically equal to the identity, and $U$ is
% commutative. Since $\U^c_{g,n}$ is the inverse limit of its
% algebraic quotients, it is also commutative.

It follows that the kernel of $\u_{g,n}\to\u^c_{g,n}$ contains
\[
\overline{[\u_{g,n},\u_{g,n}]}=W_{-2}\u_{g,n}.
\]
By the preceding proof, the reverse containment also holds.
Thus $\n_{g,n}=W_{-2}\u_{g,n}$, and the projection induces
\[
\u^c_{g,n}\cong\u_{g,n}/W_{-2}\u_{g,n}
=\Gr^W_{-1}\u_{g,n}.
\]
Corollary~\ref{cor:weight-minus-one-unchanged} gives the asserted
$\Sp(H)$-module. The quotient filtration has no weights at most $-2$.
\end{proof}

The comparison with characteristic zero applies over any algebraically
closed field of characteristic zero.
\begin{proposition}[Characteristic-zero compact-type comparison]
\label{prop:generic-fiber-compact-type-comparison}
Assume $g\geq3$, let $\ell$ be any prime, and let $F$ be an algebraically
closed field of characteristic zero. After compatible identifications
of the symplectic monodromy representations, there is an isomorphism
\[
\cG^c_{g,n/F}\cong\cG^c_{g,n}\times_\Q\Ql.
\]
Consequently, its prounipotent Lie algebra is abelian and isomorphic
to $\Lambda^3_0H_{\Ql}\oplus H_{\Ql}^{\oplus n}$.
\end{proposition}
\begin{proof}
Choose embeddings $\overline\Q\hookrightarrow F$ and
$\overline\Q\hookrightarrow\C$. Invariance under extension of
algebraically closed fields of characteristic zero
\cite[Lem.~4.1.16]{kedlaya_sheaves}, the comparison theorem
\cite[Exp.~XII, Cor.~5.2]{SGA1}, and Boggi's description
\cite[Prop.~3.1]{boggi_compact_type} give
\[
\pi_1(\M^c_{g,n/F})\cong\widehat{\G_{g,n}/K_{g,n}},
\]
compatibly with symplectic monodromy. For the stack comparison, apply
the scheme-theoretic results to a finite \'etale level cover and its
fiber product with itself, and descend finite \'etale covers along
this presentation. Such a level cover may be chosen using an auxiliary
odd prime, independently of the coefficient prime $\ell$.
Corollary~\ref{cor:rational-to-continuous-completion}, applied to
the finitely generated group $\G_{g,n}/K_{g,n}$, gives the asserted
isomorphism $\cG^c_{g,n/F}\cong\cG^c_{g,n}\times_\Q\Ql$.
% Apply Proposition~\ref{prop:continuous-completion-profinite} to the
% finitely generated group $\G_{g,n}/K_{g,n}$.
% Together with the coefficient-field comparison recalled in
% Section~\ref{sec:relative-completion}, it identifies its continuous
% relative completion with $\cG^c_{g,n}\times_\Q\Ql$.
The Lie-algebra assertion follows from
Theorem~\ref{thm:compact-type-completion-char-zero}.
\end{proof}
% Thus the passage from $\M_{g,n}$ to $\M^c_{g,n}$ preserves the weight
% $-1$ generators of the relative completion and kills all brackets
% among them. Equivalently, the prounipotent radical of the compact-type
% relative completion is commutative, with Lie algebra
% $\Lambda^3_0H\oplus H^{\oplus n}$.

% REORGANIZED SECTION: the detailed fixed-genus computations are placed
% here once. Moved computations remain black; revised setup is marked red.
\section{Separating Twists of Fixed Genus in Weight
\texorpdfstring{$-2$}{-2}}
\label{sec:image-Kgn}

Suppose that $g\geq3$, $n\geq0$, and $1\leq h\leq g-1$.
Throughout this section, put
\[
H_{\Z}:=H_1(\Sigma_g,\Z),
\qquad H:=H_{\Z}\otimes_{\Z}\Q.
\]
Let $N_h$ be the normal subgroup of $\G_{g,n}$ generated by
Dehn twists about separating curves that bound a genus-$h$
subsurface containing at most two marked points.
Then $N_h\subseteq K_{g,n}$.

We prove that the weight $-2$ initial forms of elements of $N_h$
span $\Gr^W_{-2}\u_{g,n}$. The corresponding assertion for
$K_{g,n}$ follows as a corollary.
For $D\in K_{g,n}$, we use the initial form
$\operatorname{in}_{-2}(D)\in\Gr^W_{-2}\u_{g,n}$
defined in Subsection~\ref{subsec:separating-twists-weight}.
We use the same notation in other unipotent completions for
the logarithm modulo $W_{-3}$, provided the logarithm lies
in $W_{-2}$.
% Recall from Section~\ref{sec:johnson-subgroup} that
% \[
% \log\widetilde\rho_{g,n}(K_{g,n})\subseteq W_{-2}\u_{g,n}.
% \]
% For $D\in K_{g,n}$, its initial form is
% \[
% \operatorname{in}_{-2}(D)=\log\widetilde\rho_{g,n}(D)
% \pmod{W_{-3}\u_{g,n}}\in\Gr^W_{-2}\u_{g,n}.
% \]
% We use only these initial forms in the computations below.

\subsection{The weight \texorpdfstring{$-2$}{-2} decomposition}
\label{subsec:weight-minus-two-decomposition}
For $n=0$, use the convention $\p_{g,0}=0$.
The forgetful homomorphism $\G_{g,n}\to\G_g$ induces an exact
sequence in weight $-2$:
\[
0\longrightarrow\Gr^W_{-2}\p_{g,n}
\longrightarrow\Gr^W_{-2}\u_{g,n}
\longrightarrow\Gr^W_{-2}\u_g
\longrightarrow0.
\]
For every $g\geq3$, Hain proves that
$\Gr^W_{-2}\u_g\cong V(2\lambda_2)$
\cite[Thm.~10.1]{hain_infpretor}.

Write
\[
\Lambda^2_0H:=\Lambda^2H/\langle\theta\rangle,
\]
where $\theta$ is the symplectic class. This is the irreducible
$\Sp(H)$-module $V(\lambda_2)$.
The coordinate projections of the configuration space, together
with the identification $\Gr^W_{-2}\p\cong\Lambda^2_0H$, give
an exact sequence
\[
0\longrightarrow\bigoplus_{i<j}\Q(1)_{ij}
\longrightarrow\Gr^W_{-2}\p_{g,n}
\longrightarrow(\Lambda^2_0H)^{\oplus n}
\longrightarrow0;
\]
see \cite[Thm.~12.6]{hain_infpretor} and, for the corresponding
statement over $\Ql$,
\cite[p.~732, preceding Lem.~8.7]{hain_rational}.
Here $\Q(1)_{ij}$ is the diagonal summand associated to the pair
$\{i,j\}$.

By semisimplicity of $\Sp(H)$-modules, these two exact sequences
give a decomposition
\begin{equation}\label{eq:wt2}
\Gr^W_{-2}\u_{g,n}
\cong V(2\lambda_2)
\oplus\bigoplus_{r=1}^n V(\lambda_2)^{(r)}
\oplus\bigoplus_{1\leq i<j\leq n}\Q(1)_{ij}.
\end{equation}
The coordinate-projection map above restricts to an isomorphism
from the $V(\lambda_2)$-isotypical component of
$\Gr^W_{-2}\p_{g,n}$ onto $(\Lambda^2_0H)^{\oplus n}$.
We define $V(\lambda_2)^{(r)}$ to be the inverse image of the
$r$th factor under this restriction. Thus the point-forgetting
maps act on these summands as coordinate projections.

For a normal subgroup
$N\triangleleft\G_{g,n}$ contained in $K_{g,n}$, set
\[
I_N:=\operatorname{span}_{\Q}
\{\operatorname{in}_{-2}(D):D\in N\}
\subseteq\Gr^W_{-2}\u_{g,n}.
\]
For the subgroup $N_h$ fixed above, write $I_h=I_{N_h}$.
\begin{lemma}\label{lem:image-Kgn-Sp-submodule}
The subspace $I_N$ is an $\Sp(H)$-submodule of
$\Gr^W_{-2}\u_{g,n}$.
\end{lemma}
\begin{proof}
Since $N$ is normal in $\G_{g,n}$, for every $\gamma\in\G_{g,n}$
and $D\in N$, one has $\gamma D\gamma^{-1}\in N$.
Conjugation equivariance of the logarithm gives
\[
\log\widetilde\rho_{g,n}(\gamma D\gamma^{-1})
=\Ad(\widetilde\rho_{g,n}(\gamma))
(\log\widetilde\rho_{g,n}(D)).
\]
The adjoint action of $\cG_{g,n}$ on $\Gr^W_{-2}\u_{g,n}$ factors
through $\Sp(H)$. Indeed,
\[
[W_{-1}\u_{g,n},W_{-2}\u_{g,n}]\subset W_{-3}\u_{g,n},
\]
so the adjoint action of $\U_{g,n}$ is trivial on
$\Gr^W_{-2}\u_{g,n}$. Consequently,
\[
\operatorname{in}_{-2}(\gamma D\gamma^{-1})
=\rho_{g,n}(\gamma)\cdot\operatorname{in}_{-2}(D).
\]
Thus $I_N$ is stable under $\rho_{g,n}(\G_{g,n})$.
The stabilizer
\[
\operatorname{Stab}_{\Sp(H)}(I_N)
=\{s\in\Sp(H):sI_N=I_N\}
\]
is a Zariski-closed subgroup of $\Sp(H)$.
% \begin{remark}[Closedness of the stabilizer]
% Put $B=\Gr^W_{-2}\u_{g,n}$, and choose a basis
% $e_1,\ldots,e_m$ of $B$ such that
% \[
% I_N=\operatorname{span}_{\Q}\{e_1,\ldots,e_r\}.
% \]
% The algebraic representation $R:\Sp(H)\to\GL(B)$ has, in this
% basis, a block matrix
% \[
% R(s)=
% \begin{pmatrix}
% A(s)&C(s)\\
% D(s)&E(s)
% \end{pmatrix}.
% \]
% The condition $sI_N\subseteq I_N$ is equivalent to $D(s)=0$:
% the lower-left block records the components of
% $s(e_1),\ldots,s(e_r)$ outside $I_N$.
% Since $R(s)$ is invertible, $\dim(sI_N)=\dim I_N$, so this
% containment is equivalent to $sI_N=I_N$. Consequently,
% \[
% \operatorname{Stab}_{\Sp(H)}(I_N)
% =\{s\in\Sp(H):D(s)=0\}.
% \]
% The entries of $D(s)$ are regular functions on $\Sp(H)$ because
% $R$ is algebraic. Their common zero locus is therefore Zariski
% closed. It is a subgroup since products and inverses of
% transformations preserving $I_N$ also preserve $I_N$.
% In particular, if this stabilizer contains the Zariski-dense
% subgroup $\rho_{g,n}(\G_{g,n})$, it equals $\Sp(H)$.
% Thus invariance under $\G_{g,n}$ implies that $I_N$ is an
% $\Sp(H)$-submodule of $B$.
% \end{remark}
Since this stabilizer contains
$\rho_{g,n}(\G_{g,n})=\Sp(H_{\Z})$, which is Zariski dense
in $\Sp(H)$, it equals $\Sp(H)$. Therefore $I_N$ is an $\Sp(H)$-submodule.
\end{proof}
The representations $V(2\lambda_2)$, $V(\lambda_2)$, and $\Q(1)$ are
pairwise nonisomorphic. Semisimplicity therefore gives
\[
\begin{split}
I_h={}&(I_h\cap V(2\lambda_2))
\oplus(I_h\cap V(\lambda_2)^{\oplus n})\\
&\oplus(I_h\cap\Q(1)^{\oplus\binom n2}).
\end{split}
\]
Thus each isotypical component of an element of $I_h$ also belongs
to $I_h$. This observation separates the three distinct isotypical
types. It does not, by itself, separate the multiple copies of
$V(\lambda_2)$ or $\Q(1)$. Those copies will be isolated below using
the appropriate forgetful homomorphisms.

Since the logarithms start in weight $-2$, the
Baker--Campbell--Hausdorff formula gives
\[
\operatorname{in}_{-2}(D_1D_2)
=\operatorname{in}_{-2}(D_1)+\operatorname{in}_{-2}(D_2)
\qquad(D_1,D_2\in K_{g,n}).
\]
Thus $I_h$ is the $\Sp(H)$-submodule generated by the initial forms
of the selected genus-$h$ twists.
Choose a genus-$h$ subsurface $A_S\subset\Sigma_{g,n}$ with one
boundary component, containing exactly the marked points
$x_i$ with $i\in S$. Set
\[
c_S:=\partial A_S,
\qquad d_S:=\operatorname{in}_{-2}(T_{c_S}).
\]
These choices will be specified in each computation below.
The notation $d_S$ suppresses its dependence on the chosen curve
$c_S$.
% \begin{remark}[Existence of the subsurface $A_S$]
% For every $1\leq h\leq g-1$ and every
% $S\subseteq\{1,\ldots,n\}$, there exists a genus-$h$ subsurface
% $A_S\subset\Sigma_{g,n}$ with one boundary component, containing
% exactly the marked points $x_i$ with $i\in S$.
% Indeed, take compact connected oriented surfaces $A'$ and $B'$
% of genera $h$ and $g-h$, respectively, each with one boundary
% component. Place distinct labeled points $y_i$ in the interior
% of $A'$ for $i\in S$ and in the interior of $B'$ for $i\notin S$.
% Glue the boundary circles by an orientation-reversing
% diffeomorphism. The resulting surface is a closed oriented
% surface of genus $g$ with $n$ ordered marked points.
% By the classification of marked surfaces, there is an
% orientation-preserving diffeomorphism
% \[
% f:A'\cup_{\partial}B'\longrightarrow\Sigma_{g,n}
% \]
% such that $f(y_i)=x_i$ for every $i$.
% Thus
% \[
% A_S:=f(A'),
% \qquad c_S:=\partial A_S
% \]
% have the required properties.
% The chosen side of $c_S$ is the subsurface $A_S$.
% This specifies the side even when $g=2h$, so that both sides
% have genus $h$. These choices are not canonical; representatives
% will be chosen to suit each subsequent calculation.
% \end{remark}

\subsection{The unpointed summand
\texorpdfstring{$V(2\lambda_2)$}{V(2 lambda 2)}}
\label{subsec:unpointed-summand}
\label{subsec:fixed-h-unpointed}
We first recall the second Johnson homomorphism and its compatibility
with relative completion. The Lie algebra $\p$ is the Malcev Lie
algebra of a closed genus-$g$ surface group. Fix $x\in \Sigma_g$, and put
$\Pi_x=\pi_1^{\top}(\Sigma_g,x)$.
Since $K_{g,1}$ acts trivially on $\Pi_x/L^3\Pi_x$, its action on
$\Pi_x/L^4\Pi_x$ defines
\[
\tau_2:K_{g,1}\longrightarrow\Hom(H_\Z,\Gr^L_3\Pi_x),\qquad
\tau_2(D)(u)=D(\widetilde u)\widetilde u^{-1},
\]
where $L^\bullet\Pi_x$ is the lower central series and
$\widetilde u$ lifts $u\in H_\Z$. After tensoring with $\Q$, we regard
this homomorphism as taking values in
$\Gr^W_{-2}\Der\p
\cong V(2\lambda_2)\oplus V(\lambda_2)$;
see \cite[\S6]{LuoWatanabeCollino}. We suppress Tate twists below.

For the computations we need only the following hyperelliptic setup.
Fix a hyperelliptic involution $\sigma$ on $\Sigma_g$; its $2g+2$
fixed points are called Weierstrass points. The \emph{hyperelliptic
mapping class group}
\[
\Delta_g\subset\G_g
\]
is the centralizer of $[\sigma]$ in $\G_g$. Its Torelli subgroup,
called the \emph{hyperelliptic Torelli group}, is
\[
T\Delta_g:=\Delta_g\cap T_g.
\]
The group $T\Delta_g$ fixes all Weierstrass points and is generated
by twists about separating curves invariant under $\sigma$; see
\cite[\S3.1 and Thm.~6.9]{LuoWatanabeCollino}. Marking a Weierstrass point $q$ and then identifying
$(\Sigma_g,q)$ with the fixed reference once-pointed surface gives
a homomorphism
\[
\iota_q:T\Delta_g\longrightarrow K_{g,1}.
\]
The hyperelliptic Johnson homomorphism $\tau_q^{\hyp}$ is defined by
the induced action on the lower central quotients of
$\pi_1^{\top}(\Sigma_g,q)$. Consequently,
\[
\tau_q^{\hyp}=\tau_2\circ\iota_q;
\]
see \cite[\S6]{LuoWatanabeCollino}. Here the based surface groups are
identified using homeomorphisms isotopic to the identity; the resulting
identifications of their associated graded Lie algebras agree, since
these Lie algebras are generated by $H$.
We record explicitly the compatibility with relative completion.
\begin{lemma}\label{lem:second-johnson-adjoint-compatibility}
For $1\leq j\leq n$, let
\[
F_j:\G_{g,n}\longrightarrow\G_{g,1}
\]
forget all marked points except $x_j$, and also write
\[
f_j:\Gr^W_{-2}\u_{g,n}\longrightarrow
\Gr^W_{-2}\u_{g,1}
\]
for the induced map in weight $-2$. Then the diagram
\[
\begin{tikzcd}[column sep=large]
K_{g,n}\arrow[r,"F_j"]
\arrow[d,"\operatorname{in}_{-2}"']&
K_{g,1}\arrow[d,"\tau_2"]\\
\Gr^W_{-2}\u_{g,n}
\arrow[r,"(\Gr^W_{-2}\ad)\circ f_j"']&
\Gr^W_{-2}\Der\p
\end{tikzcd}
\]
commutes. In particular, for $D\in T\Delta_g$,
\[
(\Gr^W_{-2}\ad)
\bigl(\operatorname{in}_{-2}(\iota_qD)\bigr)
=\tau_q^{\hyp}(D).
\]
\end{lemma}
\begin{proof}
We first consider the case of $K_{g,1}$. The action of
$\G_{g,1}$ on the based surface group
$\Pi_x=\pi_1^{\top}(\Sigma_g,x)$ induces an action on its
unipotent completion $\cP$. This action extends to the relative
completion and, on Lie algebras, gives a weight-preserving
homomorphism
\[
\ad:\u_{g,1}\longrightarrow\Der\p;
\]
see \cite[Prop.~8.8]{hain_rational}.

Let $D\in K_{g,1}$ and put
$L_D:=\log\widetilde\rho_{g,1}(D)\in W_{-2}\u_{g,1}$.
Since the action of the relative completion extends the original
action of $\G_{g,1}$, we have
$\ad(L_D)=\log(D_*)$, where $D_*$ denotes the automorphism of
$\p$ induced by $D$.

We compare the weight-$-2$ part of this derivation with the second
Johnson homomorphism. Let $u\in H_\Z$ and choose a lift
$\widetilde u\in\Pi_x$. Since $D$ acts trivially on
$\Pi_x/L^3\Pi_x$, the element
$\delta_u:=D(\widetilde u)\widetilde u^{-1}$ lies in
$L^3\Pi_x$. By definition,
\[
\tau_2(D)(u)
=[\delta_u]
\in
\Gr^L_3\Pi_x\otimes\Q.
\]

Denote the images of $\widetilde u$ and $\delta_u$ under the
canonical homomorphism $\Pi_x\to\cP(\Q)$ by the same symbols,
and set
$X:=\log(\widetilde u)\in W_{-1}\p$ and
$Y:=\log(\delta_u)\in W_{-3}\p$.
Since $D(\widetilde u)=\delta_u\widetilde u$, functoriality of
the logarithm and the Baker--Campbell--Hausdorff formula give
\[
D_*(X)
=\log\bigl(D(\widetilde u)\bigr)
=\log\bigl(\exp(Y)\exp(X)\bigr)
\equiv X+Y
\pmod{W_{-4}\p}.
\]
Thus $(D_*-1)(X)\equiv Y\pmod{W_{-4}\p}$.

Since $D_*-1$ lowers weight by at least $2$, it is topologically
nilpotent with respect to the weight filtration, and hence
\[
\log(D_*)
=\sum_{r\geq1}\frac{(-1)^{r+1}}{r}(D_*-1)^r.
\]
All terms with $r\geq2$ send $W_{-1}\p$ into $W_{-5}\p$.
Therefore
$\log(D_*)(X)\equiv Y\pmod{W_{-4}\p}$.
Under the canonical identification
$\Gr^W_{-3}\p\cong\Gr^L_3\Pi_x\otimes\Q$, the class of $Y$
is $[\delta_u]$. For $u\in H_\Z$, the preceding computation gives
\[
\begin{aligned}
\bigl((\Gr^W_{-2}\ad)(\operatorname{in}_{-2}(D))\bigr)(u)
&=[\ad(L_D)(X)]\\
&=[(\log D_*)(X)]=[Y]=\tau_2(D)(u),
\end{aligned}
\]
where brackets denote classes in $\Gr^W_{-3}\p$.
Since $H_\Z$ spans $H$ over $\Q$ and $\Gr^W_\bullet\p$ is
generated by $H$, equality on these generators gives
\begin{equation}\label{eq:second-johnson-adjoint}
(\Gr^W_{-2}\ad)
\bigl(\operatorname{in}_{-2}(D)\bigr)
=\tau_2(D)
\qquad (D\in K_{g,1}).
\end{equation}
Now let $D\in K_{g,n}$. Compatibility of the forgetful homomorphism
$F_j$ with relative completion and the weight filtration gives
\[
f_j\bigl(\operatorname{in}_{-2}(D)\bigr)
=\operatorname{in}_{-2}\bigl(F_j(D)\bigr).
\]
Applying \eqref{eq:second-johnson-adjoint} to $F_j(D)\in K_{g,1}$
proves the commutativity of the diagram. The final assertion
follows from $\tau_q^{\hyp}=\tau_2\circ\iota_q$.
% Now let $D\in K_{g,n}$. The forgetful homomorphism
% $F_j:\G_{g,n}\to\G_{g,1}$ is compatible with relative completion
% and the weight filtration, so
% \[
% f_j\bigl(\operatorname{in}_{-2}(D)\bigr)
% =\operatorname{in}_{-2}\bigl(F_j(D)\bigr).
% \]
% Applying the identity above to $F_j(D)\in K_{g,1}$ gives
% \[
% (\Gr^W_{-2}\ad)
% \bigl(f_j(\operatorname{in}_{-2}(D))\bigr)
% =\tau_2(F_j(D)),
% \]
% which proves the commutativity of the diagram.

% Finally, for $D\in T\Delta_g$, the identity
% $\tau_q^{\hyp}=\tau_2\circ\iota_q$ gives
% \[
% (\Gr^W_{-2}\ad)
% \bigl(\operatorname{in}_{-2}(\iota_qD)\bigr)
% =\tau_q^{\hyp}(D).
% \]
\end{proof}
Choose the hyperelliptic involution $\sigma$ so that its Weierstrass
points are disjoint from the marked points. Let
\[
\pi_\sigma:\Sigma_g\longrightarrow
\Sigma_g/\langle\sigma\rangle\cong S^2
\]
be the quotient map. Choose a simple closed curve
$\bar c_{\varnothing}\subset S^2$, disjoint from the branch points
and from the images of the markings, which bounds a disk containing
exactly $2h+1$ of the $2g+2$ branch points and no images of the
marked points. Its inverse image
$c_{\varnothing}:=\pi_\sigma^{-1}(\bar c_{\varnothing})$ is a
connected $\sigma$-invariant separating curve. The inverse image
$A$ of the enclosed disk is a genus-$h$ subsurface with one boundary
component and contains no markings. 

Since $c_{\varnothing}$ is disjoint from all marked points, the Dehn
twist $T_{c_{\varnothing}}$ may be represented by a diffeomorphism
supported in an annular neighborhood of $c_{\varnothing}$ disjoint
from the marked points. Thus we may regard
$T_{c_{\varnothing}}$ as an element of $K_{g,n}$. After forgetting
the marked points, we use the same notation for the corresponding
element of $T\Delta_g\subseteq K_g$.

Choose a Weierstrass point $q$
outside $A$; this is possible since the complementary disk contains
$2(g-h)+1$ branch points. Set
\[
D_q:=\iota_q(T_{c_{\varnothing}})\in K_{g,1},
\]
where on the right $T_{c_{\varnothing}}$ denotes the unpointed
element of $T\Delta_g$.

Let $B$ be the closure of $\Sigma_g\setminus A$, and let
$H_A,H_B\subset H$ be the images of $H_1(A,\Q)$ and $H_1(B,\Q)$,
respectively. Then
\[
H=H_A\oplus H_B.
\]
Choose a symplectic basis
$a_1,b_1,\ldots,a_g,b_g$ adapted to this decomposition, so that
\[
H_A=\langle a_1,b_1,\ldots,a_h,b_h\rangle,
\qquad
H_B=\langle a_{h+1},b_{h+1},\ldots,a_g,b_g\rangle,
\]
and $\langle a_r,b_s\rangle=\delta_{rs}$. Set
\[
\theta_A=\sum_{r=1}^h a_r\wedge b_r,\qquad
\theta_B=\sum_{r=h+1}^g a_r\wedge b_r,\qquad
\theta=\theta_A+\theta_B.
\]
The surface Lie algebra has the presentation
\[
\Gr^W_\bullet\p\cong\L(H)/(\theta),\qquad
\theta=\sum_{r=1}^g[a_r,b_r],
\]
where $\L(H)$ is the free Lie algebra on $H$.
We identify exterior products in degree two with Lie brackets, and
suppress Tate twists in the following formulas of symplectic representations.

Recall from \cite[\S6.1]{watanabe_sections} the
$\Sp(H)$-equivariant map
\[
\phi:S^2(\Lambda^2H)\longrightarrow
\Gr^W_{-2}\Der\p
\]
defined by
\begin{align*}
\phi((u_1\wedge v_1)(u_2\wedge v_2))(x)
={}&\langle u_1,x\rangle[v_1,[u_2,v_2]]
-\langle v_1,x\rangle[u_1,[u_2,v_2]]\\
&+\langle u_2,x\rangle[v_2,[u_1,v_1]]
-\langle v_2,x\rangle[u_2,[u_1,v_1]].
\end{align*}
By \cite[Cor.~6.5 and 6.6]{watanabe_sections},
\[
\Gr^W_{-2}\Der\p
\cong
V(2\lambda_2)\oplus\ad(\Gr^W_{-2}\p),
\qquad
\ad(\Gr^W_{-2}\p)\cong\Lambda^2_0H.
\]
 By
Lemma~\ref{lem:second-johnson-adjoint-compatibility},
\[
(\Gr^W_{-2}\ad)
\bigl(\operatorname{in}_{-2}(D_q)\bigr)
=\tau_q^{\hyp}(T_{c_{\varnothing}}).
\]
Since $q$ lies in the component of
$\Sigma_g\setminus c_{\varnothing}$ complementary to $A$,
Proposition~6.2 of \cite{watanabe_sections} gives
\[
\tau_q^{\hyp}(T_{c_{\varnothing}})
=\frac12\phi(\theta_A^2).
\]
Therefore
\begin{equation}\label{eq:fixed-h-alt-adjoint}
(\Gr^W_{-2}\ad)
\bigl(\operatorname{in}_{-2}(D_q)\bigr)
=\frac12\phi(\theta_A^2).
\end{equation}

We now compute the $V(2\lambda_2)$-component of
$\phi(\theta_A^2)$. We use the maps from
\cite[\S6.3]{watanabe_sections}, denoted here by
\[
\pi_\Lambda:S^2(\Lambda^2H)\longrightarrow\Lambda^2H,
\qquad
s_\Lambda:\Lambda^2H\longrightarrow S^2(\Lambda^2H).
\]
The first is given by
\begin{align*}
\pi_\Lambda((u_1\wedge v_1)(u_2\wedge v_2))
={}&\langle u_1,v_1\rangle v_2\wedge u_2
+\langle v_2,u_2\rangle u_1\wedge v_1\\
&+\frac12\bigl(
\langle u_1,v_2\rangle v_1\wedge u_2
+\langle v_1,u_2\rangle u_1\wedge v_2\\
&\hspace{28pt}
+\langle u_1,u_2\rangle v_2\wedge v_1
+\langle v_2,v_1\rangle u_1\wedge u_2
\bigr).
\end{align*}
Writing $\contr(u\wedge v)=\langle u,v\rangle$, the formula in
\cite[\S6.3]{watanabe_sections} simplifies to
\[
s_\Lambda(\alpha)
=-\frac1{g+1}\alpha\theta
+\frac{\contr(\alpha)}
{(g+1)(2g+1)}\theta^2.
\]
By \cite[Lem.~6.10 and Cor.~6.11--6.12]{watanabe_sections},
\begin{equation}\label{eq:pi-splitting}
\pi_\Lambda s_\Lambda=\id,
\qquad
\ker\pi_\Lambda/\Lambda^4H\cong V(2\lambda_2).
\end{equation}
Here $\Lambda^4H$ denotes its natural image in
$S^2(\Lambda^2H)$, as in
\cite[Prop.~6.8]{watanabe_sections}. We also use that $\phi$
vanishes on $\Lambda^4H$ and on $\theta^2$.

\begin{lemma}\label{lem:sLambda-inner}
For every $\alpha\in\Lambda^2H$, the derivation
$\phi(s_\Lambda(\alpha))$ is inner. More precisely,
\[
\phi(s_\Lambda(\alpha))
=\frac1{g+1}\ad([\alpha]),
\]
where $[\alpha]\in\Gr^W_{-2}\p$ denotes the class of $\alpha$.
\end{lemma}
\begin{proof}
It suffices by linearity to take $\alpha=u\wedge v$. From the
defining formula for $\phi$,
\[
\begin{aligned}
\phi(\alpha\theta)(x)
={}&
\langle u,x\rangle[v,\theta]
-\langle v,x\rangle[u,\theta]\\
&+
\sum_{i=1}^g
\left(
\langle a_i,x\rangle[b_i,[u,v]]
-\langle b_i,x\rangle[a_i,[u,v]]
\right).
\end{aligned}
\]
The first two terms vanish in
$\Gr^W_\bullet\p\cong\L(H)/(\theta)$. Using the symplectic-basis
identity
\[
x=\sum_{i=1}^g
\bigl(
\langle a_i,x\rangle b_i
-\langle b_i,x\rangle a_i
\bigr),
\]
the remaining terms sum to $[x,[u,v]]$. Thus
\[
\phi(\alpha\theta)(x)=[x,\alpha].
\]
Since $\phi(\theta^2)=0$, the formula for $s_\Lambda$ now gives
\[
\phi(s_\Lambda(\alpha))(x)
=-\frac1{g+1}[x,\alpha]
=\frac1{g+1}\ad([\alpha])(x).
\]
\end{proof}

The preceding lemma separates the inner-derivation part from the
$V(2\lambda_2)$-part.
\begin{lemma}\label{lem:V2lambda2-projection}
For every $\beta\in S^2(\Lambda^2H)$,
\[
\operatorname{pr}_{V(2\lambda_2)}\phi(\beta)
=\phi\bigl(\beta-s_\Lambda\pi_\Lambda(\beta)\bigr).
\]
\end{lemma}
\begin{proof}
By \eqref{eq:pi-splitting},
\[
\beta-s_\Lambda\pi_\Lambda(\beta)\in\ker\pi_\Lambda.
\]
Modulo $\Lambda^4H$, this kernel is $V(2\lambda_2)$, and $\phi$
vanishes on $\Lambda^4H$. Hence
$\phi(\beta-s_\Lambda\pi_\Lambda(\beta))$ belongs to the
$V(2\lambda_2)$-summand of
\[
\Gr^W_{-2}\Der\p
\cong
V(2\lambda_2)\oplus\ad(\Gr^W_{-2}\p).
\]
On the other hand, Lemma~\ref{lem:sLambda-inner} shows that
$\phi(s_\Lambda\pi_\Lambda(\beta))$ is inner. The asserted
formula follows from
\[
\phi(\beta)
=\phi\bigl(\beta-s_\Lambda\pi_\Lambda(\beta)\bigr)
+
\phi\bigl(s_\Lambda\pi_\Lambda(\beta)\bigr).
\]
\end{proof}

We apply this projection formula to $\beta=\theta_A^2$.
\begin{lemma}\label{lem:thetaA-projection}
One has
\[
\pi_\Lambda(\theta_A^2)=-(2h+1)\theta_A.
\]
Consequently,
\begin{equation}\label{eq:fixed-h-alt-component}
\operatorname{pr}_{V(2\lambda_2)}\phi(\theta_A^2)
=\phi\left(
\theta_A^2-\frac{2h+1}{g+1}\theta_A\theta
\right).
\end{equation}
\end{lemma}
\begin{proof}
Set $\theta_r=a_r\wedge b_r$. Directly from the formula for
$\pi_\Lambda$,
\[
\pi_\Lambda(\theta_r^2)=-3\theta_r,
\qquad
\pi_\Lambda(\theta_r\theta_s)
=-(\theta_r+\theta_s)
\quad(r\neq s).
\]
Since $\theta_A=\sum_{r=1}^h\theta_r$, an immediate expansion gives
\[
\theta_A^2
=\sum_{r=1}^h\theta_r^2
+
2\sum_{1\le r<s\le h}\theta_r\theta_s.
\]
Using the preceding formulas for $\pi_\Lambda$, a direct computation
therefore yields
\begin{align*}
\pi_\Lambda(\theta_A^2)
&=
\sum_{r=1}^h\pi_\Lambda(\theta_r^2)
+
2\sum_{1\le r<s\le h}
\pi_\Lambda(\theta_r\theta_s)\\
&=
-3\theta_A
-
2\sum_{1\le r<s\le h}(\theta_r+\theta_s)\\
&=
-3\theta_A-2(h-1)\theta_A\\
&=
-(2h+1)\theta_A.
\end{align*}

Since $\contr(\theta_A)=h$,
\[
s_\Lambda(\theta_A)
=-\frac1{g+1}\theta_A\theta
+\frac{h}{(g+1)(2g+1)}\theta^2.
\]
Therefore
\[
\theta_A^2-s_\Lambda\pi_\Lambda(\theta_A^2)
=\theta_A^2
-\frac{2h+1}{g+1}\theta_A\theta
+\frac{h(2h+1)}{(g+1)(2g+1)}\theta^2.
\]
Lemma~\ref{lem:V2lambda2-projection}, together with
$\phi(\theta^2)=0$, now gives
\eqref{eq:fixed-h-alt-component}.
\end{proof}

Set
\[
\xi_h
:=
\operatorname{pr}_{V(2\lambda_2)}\phi(\theta_A^2).
\]
By Lemma~\ref{lem:thetaA-projection},
\[
\xi_h
=\phi\left(
\theta_A^2-\frac{2h+1}{g+1}\theta_A\theta
\right).
\]
For $h=1$, the coefficient $(2h+1)/(g+1)=3/(g+1)$ agrees with
the computation in the proof of
\cite[Thm.~6.14]{watanabe_sections}.

\begin{proposition}\label{prop:fixed-h-unpointed}
For every $1\leq h\leq g-1$, $\xi_h\neq0$ and
$V(2\lambda_2)\subseteq I_h$.
\end{proposition}
\begin{proof}
By \cite[Lem.~6.1]{watanabe_sections},
\[
\phi(\theta_A^2)(x)=
\begin{cases}2[x,\theta_A],&x\in H_A,\\0,&x\in H_B.
\end{cases}
\]
Together with $\phi(\theta_A\theta)(x)=[x,\theta_A]$, this gives
\[
\xi_h(x)=\begin{cases}
\dfrac{2(g-h)+1}{g+1}[x,\theta_A],&x\in H_A,\\[6pt]
-\dfrac{2h+1}{g+1}[x,\theta_A],&x\in H_B.
\end{cases}
\]
In particular, for $j=h+1$,
\begin{equation}\label{eq:fixed-h-alt-evaluation}
\xi_h(a_j)=-\frac{2h+1}{g+1}[a_j,\theta_A].
\end{equation}
To see that this bracket is nonzero modulo the surface relation,
define $F:\L(H)\to\L(X,Y)$ by
\[
F(a_1)=X,\quad F(b_1)=Y,\quad F(a_j)=X,\quad F(b_j)=-Y,
\]
and send every other basis vector to zero. Since
$F(\theta)=[X,Y]+[X,-Y]=0$, this map descends to
$\L(H)/(\theta)=\Gr^W_\bullet\p$. Since $1\leq h<j$, it satisfies
\[
F(\theta_A)=[X,Y],\qquad
F([a_j,\theta_A])=[X,[X,Y]].
\]
The latter is nonzero: its image in the free associative algebra is
$X^2Y-2XYX+YX^2$. Thus $[a_j,\theta_A]\neq0$, and
\eqref{eq:fixed-h-alt-evaluation} proves that $\xi_h\neq0$.

Recall that $c_{\varnothing}$ bounds the genus-$h$ subsurface $A$
and contains no marked points on its chosen side. Hence
$T_{c_{\varnothing}}\in N_h\subseteq K_{g,n}$, and we set
\[
d_{\varnothing}
:=
\operatorname{in}_{-2}(T_{c_{\varnothing}})
\in I_h
\subseteq
\Gr^W_{-2}\u_{g,n}.
\]

For $m\geq0$, let
\[
f_{\mathrm{all}}^{(m)}:
\Gr^W_{-2}\u_{g,m}\longrightarrow\Gr^W_{-2}\u_g
\]
be the map induced by forgetting all marked points. Under the forgetful maps,
$T_{c_{\varnothing}}\in K_{g,n}$ and
$D_q\in K_{g,1}$ both map to the same unpointed twist
$T_{c_{\varnothing}}\in K_g$. Therefore
\[
f_{\mathrm{all}}^{(n)}(d_{\varnothing})
=
f_{\mathrm{all}}^{(1)}
\bigl(\operatorname{in}_{-2}(D_q)\bigr).
\]

The action of the pointed mapping class group on the surface group
becomes an outer action after the marked point is forgotten. Thus
the map to derivations induces an $\Sp(H)$-equivariant homomorphism
\[
\beta_n:
\Gr^W_{-2}\u_{g,n}
\xrightarrow{\,f_{\mathrm{all}}^{(n)}\,}
\Gr^W_{-2}\u_g
\longrightarrow
\Gr^W_{-2}\OutDer\p
\cong V(2\lambda_2).
\]
By \eqref{eq:fixed-h-alt-adjoint} and the definition of $\xi_h$,
\[
\beta_n(d_{\varnothing})
=\frac12\,
\operatorname{pr}_{V(2\lambda_2)}
\bigl(\phi(\theta_A^2)\bigr)
=\frac12\,\xi_h
\neq0.
\]
Since $\beta_n$ is $\Sp(H)$-equivariant and its target is the
irreducible module $V(2\lambda_2)$, it vanishes on the
$V(\lambda_2)$- and $\Q(1)$-isotypical summands of $\Gr^W_{-2}\u_{g,n}$.
Hence the $V(2\lambda_2)$-component of $d_{\varnothing}$ is nonzero.
Because $I_h$ is an $\Sp(H)$-submodule and
$V(2\lambda_2)$ occurs in $\Gr^W_{-2}\u_{g,n}$ with multiplicity one, we conclude that
\[
V(2\lambda_2)\subseteq I_h.
\]
\end{proof}

%%%%%%%%%%%%%%%%%%%%%%%%%%%%%%%%%%%%%%
\subsection{The marked-point summands}
\label{subsec:marked-point-summands}
%%%%%%%%%%%%%%%%%%%%%%%%%%%%%%%%%%%%%%

Suppose $n\geq1$, and fix a marked point $x_i$.

Recall that $A$ is a connected genus-$h$ subsurface with one boundary
component $c_{\varnothing}$ and contains no marked points. Choose a
small closed disk $U_i\subset \Sigma_g\setminus A$ whose interior
contains $x_i$ and no other marked point. Since the complement of $A$
is connected, we can choose an embedded arc joining $\partial A$ to
$\partial U_i$, with interior disjoint from $A$, $U_i$, and all marked
points. Thicken this arc to a narrow rectangular band meeting $A$ and
$U_i$ along one boundary interval each and avoiding all other markings.

Let $A_i$ be a small closed regular neighborhood of the union of $A$,
the band, and $U_i$, chosen so that
$A\subset\operatorname{int}(A_i)$ and $A_i$ contains no marked point
other than $x_i$. The band joins the boundary circle of $A$ to
$\partial U_i$, leaving one boundary component; denote it by $c_i$.
Attaching the band decreases Euler characteristic by one, so
\[
\chi(A_i)
=\chi(A)+\chi(U_i)-1
=(1-2h)+1-1
=1-2h.
\]
As $A_i$ is connected and orientable with one boundary component, it
therefore has genus $h$.

The region $A_i\setminus\operatorname{int}(A)$ is connected by
construction: it consists of a collar of $\partial A$ enlarged by the
band and the disk $U_i$. Its boundary is
$c_{\varnothing}\sqcup c_i$. Since
$A_i=A\cup\bigl(A_i\setminus\operatorname{int}(A)\bigr)$ and their
intersection is $\partial A$, which has Euler characteristic zero,
additivity of Euler characteristic gives
\[
\chi\bigl(A_i\setminus\operatorname{int}(A)\bigr)
=\chi(A_i)-\chi(A)=0.
\]
A compact connected orientable surface with two boundary components and
Euler characteristic zero has genus zero and is therefore an annulus.
Thus $c_{\varnothing}$ and $c_i$ bound an annulus in the underlying
surface whose interior contains exactly the marked point $x_i$.

After forgetting $x_i$, this annulus contains no marked points, so an
isotopy supported in a neighborhood of it carries
$c_{\varnothing}$ to $c_i$ while fixing all remaining marked points.

The chosen sides of $c_{\varnothing}$ and $c_i$, represented by $A$
and $A_i$, respectively, both have genus $h$ and carry the marking sets
$\varnothing$ and $\{i\}$. Both twists therefore belong to $N_h$ by
its definition.
Using the additivity of initial forms from
Subsection~\ref{subsec:weight-minus-two-decomposition}, set
\[
e_i:=\operatorname{in}_{-2}
\bigl(T_{c_i}T_{c_{\varnothing}}^{-1}\bigr)
=d_{\{i\}}-d_{\varnothing}\in I_h.
\]
% Set
% \[
% E_i=\log\bigl(T_{c_i}T_{c_{\varnothing}}^{-1}\bigr)
% \in W_{-2}\mathfrak u_{g,n},
% \]
% where the logarithm is taken in the relative completion. Define $e_i$ to be the weight-$-2$ initial form of $E_i$:
% \[
% e_i:=\operatorname{in}_{-2}\bigl(T_{c_i}T_{c_{\varnothing}}^{-1}\bigr)
% =d_{\{i\}}-d_{\varnothing}\in I_h.
% \]
% Indeed, the Baker--Campbell--Hausdorff formula gives
% $E_i\equiv \log T_{c_i}-\log T_{c_{\varnothing}}\pmod{W_{-4}}$, since both logarithms lie in $W_{-2}\mathfrak u_{g,n}$. Passing to $\Gr^W_{-2}\mathfrak u_{g,n}$ therefore gives
% $e_i=d_{\{i\}}-d_{\varnothing}$.

For each $j$, let
\(
f_j:\Gr^W_{-2}\u_{g,n}\longrightarrow
\Gr^W_{-2}\u_{g,1}
\)
be the forgetful map of
Lemma~\ref{lem:second-johnson-adjoint-compatibility}. If $j\neq i$, then forgetting all
markings except $x_j$ in particular forgets $x_i$, and hence
$c_{\varnothing}$ and $c_i$ become isotopic. The same is true after
forgetting all markings. Consequently,
\begin{equation}\label{eq:fixed-h-other-markings}
f_j(e_i)=0\quad(j\neq i),\qquad
f_{\mathrm{all}}^{(n)}(e_i)=0
\quad\text{in }\Gr^W_{-2}\u_g.
\end{equation}

\begin{proposition}\label{prop:fixed-h-marked}
For every $i$, $V(\lambda_2)^{(i)}\subseteq I_h$.
\end{proposition}
\begin{proof}
Choose a hyperelliptic model in which the unpointed separating curve is
symmetric, and choose Weierstrass points $q_A$ and $q_B$ on its
genus-$h$ and genus-$(g-h)$ sides, respectively. As above, a circle
enclosing $2h+1$ branch points in the quotient sphere gives such a model
for every $h$. Denote the corresponding unpointed Dehn twist by $D\in T\Delta_g$.
Since $c_i$ and $c_{\varnothing}$ become isotopic after forgetting
$x_i$, both $T_{c_i}$ and $T_{c_{\varnothing}}$ map to $D$ after
forgetting all marked points.

The two one-pointed realizations correspond to $(c_i,x_i)$ and
$(c_{\varnothing},x_i)$, respectively. Choose the identifications
compatibly with the isotopy obtained after forgetting $x_i$, so that
both induce the chosen identification on $H$. Basepoint paths induce the
same identification on $\Gr^W_\bullet\p$: changing a basepoint path acts
by inner conjugation, which is trivial on the associated graded.

We study $f_i(e_i)$ using the natural action of the pointed
relative completion on $\p$.
% Via the natural action of the pointed relative completion on $\p$, we
% view $f_i(e_i)$ as an element of
% $\Gr^W_{-2}\Der\p$. 
Under the compatible identifications above, the two one-pointed
realizations of $D$ satisfy
\[
\iota_{q_A}(D)=F_i(T_{c_i}),
\qquad
\iota_{q_B}(D)=F_i(T_{c_{\varnothing}})
\quad\text{in }K_{g,1}.
\]
Here $F_i$ forgets all marked points except $x_i$, and the
compatible pointed identifications send $q_A$ and $q_B$,
respectively, to $x_i$.
By functoriality and additivity of initial forms,
\[
f_i(e_i)
=\operatorname{in}_{-2}\bigl(\iota_{q_A}(D)\bigr)
-\operatorname{in}_{-2}\bigl(\iota_{q_B}(D)\bigr)
\qquad\text{in }\Gr^W_{-2}\u_{g,1}.
\]
% Consequently, after forgetting all marked points except $x_i$,
% \[
% f_i(e_i)
% =\operatorname{in}_{-2}\bigl(
% \iota_{q_A}(D)\iota_{q_B}(D)^{-1}
% \bigr).
% \]
% Since both $\log\iota_{q_A}(D)$ and
% $\log\iota_{q_B}(D)$ lie in $W_{-2}\mathfrak u_{g,1}$, the
% Baker--Campbell--Hausdorff formula gives
% \[
% f_i(e_i)
% =\operatorname{in}_{-2}\bigl(\iota_{q_A}(D)\bigr)
% -
% \operatorname{in}_{-2}\bigl(\iota_{q_B}(D)\bigr)
% \qquad\text{in }\Gr^W_{-2}\mathfrak u_{g,1}.
% \]
Applying the natural action
\[
\Gr^W_{-2}\ad:
\Gr^W_{-2}\mathfrak u_{g,1}
\longrightarrow
\Gr^W_{-2}\Der\mathfrak p
\]
and using
Lemma~\ref{lem:second-johnson-adjoint-compatibility} for the two
realizations yields
\begin{equation}\label{eq:fixed-h-LW-compatibility}
(\Gr^W_{-2}\ad)f_i(e_i)
=\tau_{q_A}^{\hyp}(D)-\tau_{q_B}^{\hyp}(D)
\qquad\text{in }\Gr^W_{-2}\Der\mathfrak p.
\end{equation}
Let
\[
\widetilde\pi_{\Lambda^2H}:
\Gr^W_{-2}\Der\p
\longrightarrow\Lambda^2H/\Q\theta
\cong V(\lambda_2)
\]
be the $\Sp(H)$-equivariant projection of
\cite[\S6.5]{LuoWatanabeCollino}, and put
$\widetilde\tau_q^{\hyp}
=\widetilde\pi_{\Lambda^2H}\circ\tau_q^{\hyp}$.
By \cite[Lem.~9.3]{LuoWatanabeCollino}, the difference in
\eqref{eq:fixed-h-LW-compatibility} lies in the
$V(\lambda_2)$-summand of $\Gr^W_{-2}\Der\p$.
Applying $\widetilde\pi_{\Lambda^2H}$ gives
\[
\widetilde\pi_{\Lambda^2H}
\bigl((\Gr^W_{-2}\ad)f_i(e_i)\bigr)
=\widetilde\tau_{q_A}^{\hyp}(D)
-\widetilde\tau_{q_B}^{\hyp}(D).
\]
% Let
% \[
% \widetilde\pi_{\Lambda^2H}:
% \Gr^W_{-2}\Der\p\longrightarrow\Lambda^2H/\Q\theta
% \]
% be the map of \cite[\S6.5]{LuoWatanabeCollino}, and put
% $\widetilde\tau_q^{\hyp}
% =\widetilde\pi_{\Lambda^2H}\circ\tau_q^{\hyp}$.
% By \cite[Lem.~9.3]{LuoWatanabeCollino}, the difference
% \[
% \tau_{q_A}^{\hyp}(D)-\tau_{q_B}^{\hyp}(D)
% \]
% in \eqref{eq:fixed-h-LW-compatibility} lies in the
% $V(\lambda_2)$-summand of
% $\Gr^W_{-2}\Der\p$. Recall that
% \[
% \widetilde\pi_{\Lambda^2H}:
% \Gr^W_{-2}\Der\p
% \longrightarrow
% \Lambda^2H/\Q\theta
% \cong V(\lambda_2)
% \]
% is the $\Sp(H)$-equivariant projection onto this summand. Thus,
% applying $\widetilde\pi_{\Lambda^2H}$ to
% \eqref{eq:fixed-h-LW-compatibility} gives
% \[
% \widetilde\pi_{\Lambda^2H}
% \bigl((\Gr^W_{-2}\ad)f_i(e_i)\bigr)
% =\widetilde\tau_{q_A}^{\hyp}(D)
% -\widetilde\tau_{q_B}^{\hyp}(D).
% \]
The calculation in the proof of
\cite[Thm.~9.1]{LuoWatanabeCollino} gives
\[
\widetilde\tau_{q_B}^{\hyp}(D)
-\widetilde\tau_{q_A}^{\hyp}(D)
=2(2g+2)[\theta_B].
\]
Consequently,
\begin{equation}\label{eq:fixed-h-LW-nonvanishing}
\widetilde\pi_{\Lambda^2H}
\bigl((\Gr^W_{-2}\ad)f_i(e_i)\bigr)
=-4(g+1)[\theta_B]
=4(g+1)[\theta_A].
\end{equation}
We claim that $[\theta_A]\neq0$ in $\Lambda^2H/\Q\theta$.
Indeed, if $\theta_A=c\theta$, then, writing
$\theta=\theta_A+\theta_B$, comparison of the
$\Lambda^2H_A$-components gives $c=1$, whereas comparison of the
$\Lambda^2H_B$-components gives $c=0$, a contradiction.
% We claim that $[\theta_A]\neq0$ in $\Lambda^2H/\Q\theta$.
% Indeed, if $\theta_A=c\theta$, then, writing
% $\theta=\theta_A+\theta_B$, comparison of the
% $\Lambda^2H_A$-components gives $c=1$, whereas comparison of the
% $\Lambda^2H_B$-components gives $c=0$, a contradiction.
% Equivalently, under the primitive-subspace realization of
% $\Lambda^2H/\Q\theta$,
% \[
% [\theta_A]\longleftrightarrow
% \theta_A-\frac hg\theta
% =\frac{g-h}{g}\theta_A-\frac hg\theta_B\neq0.
% \]
It follows from \eqref{eq:fixed-h-LW-nonvanishing} that the
$V(\lambda_2)$-component of
$(\Gr^W_{-2}\ad)f_i(e_i)$, and hence the corresponding
$V(\lambda_2)$-component of $f_i(e_i)$, is nonzero.
The coefficient is determined by the normalization of
$\widetilde\pi_{\Lambda^2H}$.

Together, \eqref{eq:fixed-h-other-markings} and
\eqref{eq:fixed-h-LW-nonvanishing} show that the
$V(\lambda_2)$-component of $e_i$ vanishes in every marked-point
summand except the $i$th, where it is nonzero. Moreover, $f_{\mathrm{all}}^{(n)}(e_i)=0$ shows that $e_i$
has no $V(2\lambda_2)$-component.
Thus any remaining components of $e_i$ are Tate classes.
By semisimplicity, the $V(\lambda_2)$-isotypical component separates
from the Tate summands. Since $V(\lambda_2)$ is irreducible and the
$i$th component is nonzero, it follows that
\[
V(\lambda_2)^{(i)}\subseteq I_h.
\]
\end{proof}

Thus all marked-point summands are obtained from the two topological
types $(h,\varnothing)$ and $(h,\{i\})$.

%%%%%%%%%%%%%%%%%%%%%%%%%%%%%%%%%%%%%%
\subsection{The diagonal summands and the lantern relation}
\label{subsec:diagonal-summands}
%%%%%%%%%%%%%%%%%%%%%%%%%%%%%%%%%%%%%%

Suppose that $n\geq2$, and fix $1\leq i<j\leq n$. Let $\gamma_{ij}$ be the boundary of a small embedded disk whose interior contains exactly the marked points $x_i$ and $x_j$. Then $\gamma_{ij}$ is an essential separating curve in the marked surface, and therefore $T_{\gamma_{ij}}\in K_{g,n}$.

After the marked points are forgotten, $\gamma_{ij}$ becomes null-homotopic, so $T_{\gamma_{ij}}$ maps to the identity under the forgetful homomorphism $\G_{g,n}\to\G_g$. Hence $T_{\gamma_{ij}}\in\Pi_{g,n}\cap K_{g,n}$.

The Birman exact sequence identifies $\Pi_{g,n}\cong\pi_1^{\top}(F_n(C)^{\an},\mathbf x)$, where $F_n(C)=C^n-\bigcup_{r<s}\Delta_{rs}$ and $\mathbf x=(x_1,\ldots,x_n)$. Let $\mu_{ij}\in\pi_1^{\top}(F_n(C)^{\an},\mathbf x)$ be a based meridian about the diagonal $\Delta_{ij}$. It may be represented by the pure braid in which $x_i$ travels once around $x_j$, while all other marked points remain fixed. Under the point-pushing identification, $\mu_{ij}\mapsto T_{\gamma_{ij}}^{\pm1}$, where the sign depends on the orientation convention.

Near a point of $\Delta_{ij}$ lying on no other diagonal, a normal coordinate is $w=z_i-z_j$. A positive meridian is represented by $w(t)=\epsilon e^{2\pi\mathrm i t}$, $0\leq t\leq1$, with the other local coordinates fixed. We choose the orientations of twists and meridians compatibly. We will show below that the logarithm of $\mu_{ij}$, taken in the Malcev completion of $\pi_1^{\top}(F_n(C)^{\an},\mathbf x)$ with Lie algebra $\p_{g,n}$, lies in $W_{-2}\p_{g,n}$.

Choose a symplectic basis $a_1,\ldots,a_g,b_1,\ldots,b_g$ of $H$. For $X\in H$, write $X^{(r)}$ for $X$ in the $r$th copy of $\Gr^W_{-1}\p_{g,n}\cong H^{\oplus n}$. For $r<s$, set $q_{rs}=\sum_{k=1}^g[a_k^{(r)},b_k^{(s)}]$. Hain's presentation gives $\Q(1)_{rs}=\Q q_{rs}$ and an exact sequence
\[
0\longrightarrow\bigoplus_{r<s}\Q(1)_{rs}
\longrightarrow\Gr^W_{-2}\p_{g,n}
\longrightarrow\bigoplus_{r=1}^n\Gr^W_{-2}\p
\longrightarrow0;
\]
see \cite[Thm.~12.6]{hain_infpretor}.

We now determine the initial form of $\mu_{ij}$. The argument has two steps. First, coordinate- and pair-forgetting maps show that its initial form is supported only on $\Q(1)_{ij}$. We then use the normal bundle of the diagonal to prove that its component on this line is nonzero.

\begin{lemma}\label{lem:diagonal-meridian-support}
For each $i<j$, there is a scalar $\kappa_{ij}\in\Q$ such that
\[
\operatorname{in}_{-2}(\mu_{ij})=\kappa_{ij}q_{ij}.
\]
\end{lemma}
\begin{proof}
For each $r$, consider the coordinate projection $\mathrm{pr}_r:F_n(C)\to C$.
Represent $\mu_{ij}$ by the pure braid in which only the $i$th point
moves around $x_j$. If $r\neq i$, the chosen representative of $\mu_{ij}$ projects
under $\mathrm{pr}_r$ to a constant loop.
If $r=i$, it is a small circle about $x_j$, which is contractible in
$C$. Notice that the target is $C$, not $C-\{x_j\}$. Thus
\[
(\mathrm{pr}_r)_*(\mu_{ij})=1
\]
for every $r$.

By \cite[Prop.~2.1]{hain_infpretor}, the coordinate projections jointly
induce an isomorphism
\[
H_1(F_n(C),\Q)\xrightarrow{\sim}\bigoplus_{r=1}^nH_1(C,\Q).
\]
It follows that $[\mu_{ij}]=0$ in $H_1(F_n(C),\Q)$. Since
$\Gr^W_{-1}\p_{g,n}\cong H_1(F_n(C),\Q)$, the image of
$\log\mu_{ij}$ in $\Gr^W_{-1}\p_{g,n}$ vanishes. Therefore
\[
\log\mu_{ij}\in W_{-2}\p_{g,n}.
\]
Set
\[
t_{ij}:=\operatorname{in}_{-2}(\mu_{ij})
=[\log\mu_{ij}]
\in\Gr^W_{-2}\p_{g,n}
=W_{-2}\p_{g,n}/W_{-3}\p_{g,n}.
\]
Since $(\mathrm{pr}_r)_*(\mu_{ij})=1$ for every $r$, the class $t_{ij}$ is
killed by each induced map $\Gr^W_{-2}\p_{g,n}\to\Gr^W_{-2}\p$.
It therefore lies in the diagonal kernel
\[
\ker\left\{\Gr^W_{-2}\p_{g,n}\longrightarrow
\bigoplus_{r=1}^n\Gr^W_{-2}\p\right\}
=\bigoplus_{r<s}\Q(1)_{rs}=\bigoplus_{r<s}\Q q_{rs}.
\]
Consequently, we may write $t_{ij}=\sum_{u<v}\kappa_{uv}q_{uv}$.

Fix $r<s$, and let $f_{rs}:F_n(C)\to F_2(C)$ retain the $r$th and
$s$th coordinates. On the diagonal summands, the induced map satisfies
\[
\Gr^W_{-2}(f_{rs})_*(q_{uv})=
\begin{cases}
q_{12},&(u,v)=(r,s),\\
0,&(u,v)\neq(r,s).
\end{cases}
\]
Consequently,
\[
\Gr^W_{-2}(f_{rs})_*(t_{ij})=\kappa_{rs}q_{12}.
\]
If $(r,s)\neq(i,j)$, then $(f_{rs})_*(\mu_{ij})=1$.
Indeed, if the moving $i$th point is forgotten, the resulting loop is
constant. If the $i$th point is retained but the $j$th point is
forgotten, its small circle about $x_j$ contracts without meeting the
other retained point. Hence
\[
0=\Gr^W_{-2}(f_{rs})_*(t_{ij})=\kappa_{rs}q_{12}.
\]
Since $q_{12}\neq0$, it follows that $\kappa_{rs}=0$. Thus every coefficient
except possibly $\kappa_{ij}$ vanishes, and therefore
$\operatorname{in}_{-2}(\mu_{ij})=\kappa_{ij}q_{ij}$.
\end{proof}

Lemma~\ref{lem:diagonal-meridian-support} determines the only diagonal summand in which the initial form of $\mu_{ij}$ can lie. It remains to prove that its coefficient in that summand is nonzero. For this, we reduce to two points and identify the diagonal meridian with the tangent-fiber loop in a punctured tubular neighborhood of the diagonal.

\begin{lemma}\label{lem:diagonal-meridian-nonzero}
For every $i<j$,
\[
\operatorname{in}_{-2}(\mu_{ij})=\pm\frac1gq_{ij}.
\]
In particular, this initial form is nonzero.
\end{lemma}
\begin{proof}
In view of Lemma~\ref{lem:diagonal-meridian-support} and the pair-forgetting maps, it suffices to compute the coefficient in the case $n=2$.

Choose a Riemannian metric on $C$, and let $T^1C=\{(z,v)\in TC:|v|=1\}$ be the unit tangent bundle. Then $T^1C$ is a deformation retract of $T^\times C:=TC-\{0\}$. For sufficiently small $\epsilon>0$, define $\Phi_\epsilon:T^1C\to F_2(C)$ by $\Phi_\epsilon(z,v)=(\exp_z(\epsilon v),z)$, where $\exp_z:T_zC\to C$ is the Riemannian exponential map; thus $\exp_z(\epsilon v)$ is obtained by moving distance $\epsilon$ from $z$ in the direction $v$.

This identifies $T^1C$ with a circle bundle in a punctured tubular neighborhood of the diagonal. Fix a base point $(z_0,v_0)\in T^1C$. The fiber $T^1_{z_0}C=\{v\in T_{z_0}C:|v|=1\}$ is an oriented circle. Let $b_0$ denote the loop based at $v_0$ that traverses this circle once in the positive direction. Then, up to the orientation convention, $(\Phi_\epsilon)_*(b_0)=\mu_{12}^{\pm1}$.

Let $\pi_{g,\vec 1}:=\pi_1^{\top}(T^\times C)$
and let $\p_{g,\vec 1}$ denote the Lie algebra of its Malcev completion. Thus $\p_{g,\vec 1}$ is the Lie algebra denoted by $\p^0_{g,1}$ in \cite[\S\S2--3]{hain_infpretor}; the arrow over $1$ indicates that the decoration is a nonzero tangent vector rather than an ordinary marked point. Set
\[
w_0=\sum_{k=1}^g[a_k,b_k]
\in\Gr^W_{-2}\p_{g,\vec 1}.
\]

The circle bundle $T^1C\to C$ has Euler number $2-2g$; see \cite[proof of Prop.~2.1]{hain_infpretor}. Choose lifts $\widetilde a_k,\widetilde b_k\in \pi_{g,\vec 1}$ of the corresponding standard generators of $\pi_1^\top(C)$, so that their images in
$\operatorname{Gr}^W_{-1}\mathfrak p_{g,\vec 1}\cong H$
are $a_k,b_k$, respectively. They satisfy
\[
\prod_{k=1}^g[\widetilde a_k,\widetilde b_k]
=b_0^{\pm(2g-2)}.
\]
Passing to the Malcev Lie algebra and reducing modulo $W_{-3}\p_{g,\vec 1}$ gives $w_0=\pm(2g-2)\operatorname{in}_{-2}(b_0)$. Since $g\geq2$, it follows that $\operatorname{in}_{-2}(b_0)=\pm w_0/(2g-2)$.

Let $\mathrm{pr}_1,\mathrm{pr}_2:F_2(C)\to C$ denote the coordinate projections. For $r=1,2$, the composite $\mathrm{pr}_r\circ\Phi_\epsilon:T^1C\to C$ is homotopic to the bundle projection $T^1C\to C$. Consequently, the homomorphism
\[
\Gr^W_{-1}(\Phi_\epsilon)_*:
\Gr^W_{-1}\p_{g,\vec 1}
\longrightarrow
\Gr^W_{-1}\p_{g,2}
\]
identifies, under
\[
\Gr^W_{-1}\p_{g,\vec 1}\cong H,
\qquad
\Gr^W_{-1}\p_{g,2}\cong H^{(1)}\oplus H^{(2)},
\]
with the diagonal map
\[
H\longrightarrow H^{(1)}\oplus H^{(2)},
\qquad
X\longmapsto X^{(1)}+X^{(2)}.
\]
It follows that
\[
\Gr^W_{-2}(\Phi_\epsilon)_*(w_0)
=s_1+s_2+2q_{12},
\qquad
s_r=\sum_{k=1}^g[a_k^{(r)},b_k^{(r)}].
\]
The relations in Hain's presentation give $s_1+q_{12}/g=0$ and $s_2+q_{12}/g=0$; see \cite[Thm.~12.6]{hain_infpretor}. Therefore
\[
\Gr^W_{-2}(\Phi_\epsilon)_*(w_0)
=\frac{2(g-1)}{g}q_{12},
\]
which is nonzero since $g\geq2$. Since
$\operatorname{in}_{-2}(b_0)=\pm w_0/(2g-2)$, we obtain
\[
\Gr^W_{-2}(\Phi_\epsilon)_*
\bigl(\operatorname{in}_{-2}(b_0)\bigr)
=\pm\frac1g q_{12}.
\]
By functoriality of initial forms and the identity
$(\Phi_\epsilon)_*(b_0)=\mu_{12}^{\pm1}$, we obtain
\[
\operatorname{in}_{-2}(\mu_{12})
=\pm\frac1g q_{12}.
\]
For general $i<j$, the pair-forgetting map $f_{ij}$ sends
$\mu_{ij}$ to $\mu_{12}^{\pm1}$ and $q_{ij}$ to $q_{12}$.
Lemma~\ref{lem:diagonal-meridian-support} then gives
$\operatorname{in}_{-2}(\mu_{ij})=\pm q_{ij}/g$.
\end{proof}
% The homomorphism $(\Phi_\epsilon)_*$ on fundamental groups induces a
% homomorphism of Malcev Lie algebras, and hence
% \[
% (\Phi_\epsilon)_*(\log b_0)
% =\log\bigl((\Phi_\epsilon)_*(b_0)\bigr).
% \]
% Passing to weight $-2$, and using
% $(\Phi_\epsilon)_*(b_0)=\mu_{12}^{\pm1}$, gives
% \[
% \Gr^W_{-2}(\Phi_\epsilon)_*
% \bigl(\operatorname{in}_{-2}(b_0)\bigr)
% =\operatorname{in}_{-2}(\mu_{12}^{\pm1}).
% \]
% Since inversion changes the initial form by a sign, we conclude that
% \[
% \operatorname{in}_{-2}(\mu_{12})
% =\pm\frac1g q_{12}.
% \]
% For general $i<j$, the pair-forgetting map $f_{ij}$ takes $\mu_{ij}$ to $\mu_{12}^{\pm1}$ and $q_{ij}$ to $q_{12}$. Together with Lemma~\ref{lem:diagonal-meridian-support}, this gives $\operatorname{in}_{-2}(\mu_{ij})=\pm q_{ij}/g$.
% \end{proof}

We now return from the configuration-space meridian to the corresponding separating twist in the mapping class group. The point-pushing identification transfers the preceding computation to $T_{\gamma_{ij}}$.

\begin{proposition}\label{prop:diagonal-summands-in-image}
For every $1\leq i<j\leq n$,
\[
\Q\operatorname{in}_{-2}(T_{\gamma_{ij}})=\Q(1)_{ij}.
\]
\end{proposition}
\begin{proof}
Under the point-pushing identification, $\mu_{ij}\mapsto T_{\gamma_{ij}}^{\pm1}$. Lemma~\ref{lem:diagonal-meridian-nonzero} therefore gives $\Q\operatorname{in}_{-2}(T_{\gamma_{ij}})=\Q q_{ij}=\Q(1)_{ij}$. The inclusion $\p_{g,n}\to\u_{g,n}$ is injective on weight-graded quotients, so this class remains nonzero in $\Gr^W_{-2}\u_{g,n}$.
\end{proof}

\subsubsection{The lantern relation}
We now use the lantern relation to isolate the individual diagonal summands
$\Q(1)_{ij}$. Fix $i<j$.

\begin{proposition}\label{prop:fixed-h-diagonal}
For every $i<j$, the line $\Q(1)_{ij}$ is contained in $I_h$.
\end{proposition}
\begin{proof}
Choose the representatives below for the fixed pair $i<j$; they need not coincide with representatives chosen earlier.

Take disjoint small closed disks $U_i,U_j$ around $x_i,x_j$, respectively, lying outside $A$ and containing no other markings, and attach them to $A$ by disjoint bands avoiding all markings. Let $A_{ij}$ be a small closed neighborhood of this union, with boundary corners smoothed, chosen so that $A,U_i,U_j\subset\operatorname{int}(A_{ij})$. Set $c_{ij}=\partial A_{ij}$.

Each band joins two boundary components and decreases Euler
characteristic by one. Thus $A_{ij}$ has one boundary component.
Since $U_i,U_j$ are disks with their marked points retained,
\[
\chi(A_{ij})
=\chi(A)+\chi(U_i)+\chi(U_j)-2
=(1-2h)+1+1-2
=1-2h.
\]
Hence $A_{ij}$ has genus $h$ and contains exactly the markings
$x_i,x_j$.
% Each band joins two boundary components, so the two bands reduce the three boundary components of $A,U_i,U_j$ to one. Since a band is a rectangle of Euler characteristic $1$ attached along two intervals, attaching it decreases the Euler characteristic by $1$. Thickening the union and smoothing corners do not change the Euler characteristic. Since $U_i$ and $U_j$ are disks with their marked points retained,
% \[
% \chi(A_{ij})
% =\chi(A)+\chi(U_i)+\chi(U_j)-2
% =(1-2h)+1+1-2
% =1-2h.
% \]
% Thus $A_{ij}$ has genus $h$ and contains exactly the markings $x_i,x_j$.

Removing the interiors of $A,U_i,U_j$ from $A_{ij}$ gives a four-holed sphere with boundary curves $c_{\varnothing}$, $\delta_i=\partial U_i$, $\delta_j=\partial U_j$, and $c_{ij}$. Choose its interior curves $c_i,c_{ij}^{(0)},c_j$ in the standard lantern configuration of \cite[Fig.~5.1 and Prop.~5.1]{FarbMargalit}, using an orientation-preserving identification and labeling them so that their twists occur in the order $T_{c_i}T_{c_{ij}^{(0)}}T_{c_j}$ in the lantern relation. Thus $c_i$ encloses $c_{\varnothing},\delta_i$, $c_{ij}^{(0)}$ encloses $\delta_i,\delta_j$, and $c_j$ encloses $c_{\varnothing},\delta_j$. We use the twist and composition conventions of that reference.

The curve $c_{ij}^{(0)}$ bounds, after filling $\delta_i,\delta_j$ with $U_i,U_j$, a disk containing exactly $x_i$ and $x_j$. Like $\gamma_{ij}$ defined above, it therefore has the two-point genus-zero topological type, but it is chosen specifically for this lantern configuration. The superscript $(0)$ distinguishes it from $c_{ij}=\partial A_{ij}$, whose genus-$h$ side contains the same two markings.

After restoring $A,U_i,U_j$, the component bounded by $c_i$ that contains $A$ is a genus-$h$ subsurface containing exactly $x_i$; similarly, the corresponding component bounded by $c_j$ contains exactly $x_j$. Thus $c_{\varnothing},c_i,c_j,c_{ij}$ have separating types
\[
(h,\varnothing),\qquad
(h,\{i\}),\qquad
(h,\{j\}),\qquad
(h,\{i,j\}),
\]
respectively. Hence their Dehn twists correspond to divisors in $\mathcal D_{h,\leq2}$ and belong to $N_h$.

With these representatives, the lantern relation is
\begin{equation}\label{eq:fixed-h-lantern-before-capping}
T_{c_i}T_{c_{ij}^{(0)}}T_{c_j}
=T_{c_{\varnothing}}T_{\delta_i}T_{\delta_j}T_{c_{ij}}.
\end{equation}
Filling $\delta_i,\delta_j$ with their once-marked disks makes $T_{\delta_i}$ and $T_{\delta_j}$ trivial in $\G_{g,n}$. Hence
\[
T_{c_i}T_{c_{ij}^{(0)}}T_{c_j}
=T_{c_{\varnothing}}T_{c_{ij}}.
\]
Multiplying on the left by $T_{c_i}^{-1}$ and on the right by $T_{c_j}^{-1}$ gives
\begin{equation}\label{eq:fixed-h-lantern-group}
T_{c_{ij}^{(0)}}
=T_{c_i}^{-1}
T_{c_{\varnothing}}T_{c_{ij}}
T_{c_j}^{-1}.
\end{equation}
In particular, $T_{c_{ij}^{(0)}}\in N_h$.

For $S\in\{\varnothing,\{i\},\{j\},\{i,j\}\}$, set
$d_S:=\operatorname{in}_{-2}(T_{c_S})$, where
$c_{\{i\}}=c_i$, $c_{\{j\}}=c_j$, and $c_{\{i,j\}}=c_{ij}$.
By the additivity of initial forms from
Subsection~\ref{subsec:weight-minus-two-decomposition},
\eqref{eq:fixed-h-lantern-group} gives
\begin{equation}\label{eq:fixed-h-lantern-weight-two}
\operatorname{in}_{-2}(T_{c_{ij}^{(0)}})
=d_{\varnothing}+d_{\{i,j\}}-d_{\{i\}}-d_{\{j\}}
\in I_h.
\end{equation}

Since $c_{ij}^{(0)}$ and $\gamma_{ij}$ each bound a disk containing
exactly the marked points $x_i,x_j$, there is $f\in\G_{g,n}$ such that
$T_{c_{ij}^{(0)}}=fT_{\gamma_{ij}}f^{-1}$.
By Proposition~\ref{prop:diagonal-summands-in-image},
$\operatorname{in}_{-2}(T_{\gamma_{ij}})$ spans the trivial
$\Sp(H)$-module $\Q(1)_{ij}$. Conjugation equivariance therefore gives
\[
\begin{aligned}
\operatorname{in}_{-2}(T_{c_{ij}^{(0)}})
&=\rho_{g,n}(f)\cdot\operatorname{in}_{-2}(T_{\gamma_{ij}})\\
&=\operatorname{in}_{-2}(T_{\gamma_{ij}})\ne0.
\end{aligned}
\]
Thus \eqref{eq:fixed-h-lantern-weight-two} implies
$\Q(1)_{ij}\subseteq I_h$.
\end{proof}
\subsection{The fixed-genus spanning result}
\label{subsec:fixed-genus-spanning}
%%%%%%%%%%%%%%%%%%%%%%%%%%%%%%%%%%%%%%
We now combine the preceding computations.
% We now combine the preceding computations. The unpointed calculation
% gives the $V(2\lambda_2)$ summand, the pointed calculations give each
% $V(\lambda_2)^{(i)}$, and the lantern relation gives each diagonal
% summand $\Q(1)_{ij}$. Together they exhaust the decomposition
% \eqref{eq:wt2}.

\begin{theorem}[Fixed-genus spanning]
\label{thm:fixed-genus-spanning}
For every $1\leq h\leq g-1$,
\[
I_h=\Gr^W_{-2}\u_{g,n}.
\]
\end{theorem}
\begin{proof}
By Propositions~\ref{prop:fixed-h-unpointed},
\ref{prop:fixed-h-marked}, and~\ref{prop:fixed-h-diagonal},
the subspace $I_h$ contains every summand in the decomposition
\eqref{eq:wt2} of $\Gr^W_{-2}\u_{g,n}$. Hence
$I_h=\Gr^W_{-2}\u_{g,n}$.
\end{proof}

%%%%%%%%%%%%%%%%%%%%%%%%%%%%%%%%%%%%%%%%%%%%%%%%%%%%%%%%%%%%%%%%%%%%%%%%%%%%%%%%
% REORGANIZED SECTION: geometric opens and closed-ideal consequences.
% All individual representation computations are in the preceding section.
\section{Partial Compactifications and Relative Completion}
\label{sec:fixed-genus}
%%%%%%%%%%%%%%%%%%%%%%%%%%%%%%%%%%%%%%%%%%%%%%%%%%%%%%%%%%%%%%%%%%%%%%%%%%%%%%%%
We apply the computation of Section~\ref{sec:image-Kgn} to partial
compactifications of $\M_{g,n}$. Throughout this section,
$g\geq3$ and $n\geq0$ are fixed. We use the notation $\M_{\mathcal D}^{\circ}$ for the open substack of
$\Mbar_{g,n}$ obtained by adjoining to $\M_{g,n}$ the one-node loci
of the selected separating boundary divisors, suppressing the fixed
pair $(g,n)$.

\subsection{The selected boundary opens}
All boundary divisors and boundary strata in this subsection are
understood as substacks of $\Mbar_{g,n}$. For a stable separating type $(h,S)$, let $\Delta_{h,S}$ be the
closed boundary divisor whose genus-$h$ component carries precisely
the markings in $S\subseteq\{1,\ldots,n\}$. The types $(h,S)$ and
$(g-h,S^c)$ denote the same divisor. Stability requires
\[
2h-1+|S|>0,\qquad 2(g-h)-1+n-|S|>0.
\]
Both conditions hold for all $S$ when $1\leq h\leq g-1$.
For a rational tail, $h=0$, one must have $|S|\geq2$.
For a separating boundary divisor $D$, denote by $D^\circ$ its
one-node stratum:
\[
D^\circ=\{[C;x_1,\ldots,x_n]\in D:C\text{ has exactly one node}\}.
\]
Each curve in $D^\circ$ has two smooth irreducible components
meeting at its unique node.
% Since $D$ is separating, such a curve has two smooth irreducible
% components meeting at that node. Thus the superscript $\circ$
% specifies the one-node stratum of the boundary divisor.
Let $\mathcal D$ be any collection of separating boundary types,
regarded as a collection of their closed divisors. Define
\begin{equation}\label{eq:partial-open-definitions}
\begin{aligned}
\M_{\mathcal D}^{\circ}
&=\M_{g,n}\cup
\mathop{\cup}\limits_{D\in\mathcal D} D^\circ,\\
\M_{\mathcal D}
&=\Mbar_{g,n}\setminus\left(
\Delta_{\mathrm{irr}}\cup
\mathop{\cup}\limits_{\substack{D\notin\mathcal D\\
D\text{ separating}}} D
\right).
\end{aligned}
\end{equation}
Here $\Delta_{\mathrm{irr}}$ is the closed nonseparating boundary
divisor. The open $\M_{\mathcal D}$ allows several nodes, provided
each node corresponds to a divisor in $\mathcal D$.
The open $\M_{\mathcal D}^{\circ}$ consists of the curves in
$\M_{\mathcal D}$ with at most one node, including smooth curves.
% Here $\Delta_{\mathrm{irr}}$ is the closed nonseparating boundary
% divisor. The open $\M_{\mathcal D}$ may contain curves with several
% nodes, provided that the boundary divisor corresponding to each node
% belongs to $\mathcal D$. The superscript $\circ$ in
% $\M_{\mathcal D}^{\circ}$ restricts to curves with at most one node,
% including the smooth locus. Thus $\M_{\mathcal D}^{\circ}$ is the
% one-node open inside $\M_{\mathcal D}$. 
We have open immersions
\[
\M_{\mathcal D}^{\circ}\subseteq\M_{\mathcal D}\subseteq\M^c_{g,n}.
\]
The subscript $/T$ will indicate base change to a base scheme $T$, as in
$\M_{\mathcal D/T}^{\circ}$ and $\M_{\mathcal D/T}$.

The complement of $\M_{\mathcal D}^{\circ}\subset\M_{\mathcal D}$ has
codimension at least two. Indeed, on an \'etale chart, a curve with
$r$ nodes has independent smoothing parameters $q_1,\ldots,q_r$, and
the locus where these nodes persist has codimension $r$. Every point
of $\M_{\mathcal D}\setminus\M_{\mathcal D}^{\circ}$ has at least two
nodes. Zariski--Nagata purity
\cite[Exp.~X, Thm.~3.1]{SGA1}, applied on regular \'etale charts,
therefore gives an equivalence of finite \'etale cover categories for
$\M_{\mathcal D}^{\circ}$ and $\M_{\mathcal D}$. Their orbifold
fundamental groups over $\C$ are likewise isomorphic.

For $1\leq h\leq g-1$, put
\[
\begin{aligned}
\mathcal D_h&=\{\Delta_{h,S}:S\subseteq\{1,\ldots,n\}\},\\
\mathcal D_{h,\leq2}&=\{\Delta_{h,S}:S\subseteq\{1,\ldots,n\},\ |S|\leq2\}.
\end{aligned}
\]
These are collections of distinct divisors, so complementary labels
are identified when $g=2h$. We refer to $\M_{\mathcal D_h}^{\circ}$ and $\M_{\mathcal D_{h,\leq2}}^{\circ}$ as the fixed-genus one-node opens; their corresponding larger opens are $\M_{\mathcal D_h}$ and $\M_{\mathcal D_{h,\leq2}}$.

\subsection{Boundary subgroups and the closed-ideal criterion}
We next identify the orbifold fundamental groups of the selected boundary opens.

\begin{proposition}\label{prop:boundary-group}
Let $\mathcal D$ be a collection of separating boundary types, and
let $N_{\mathcal D}\unlhd\G_{g,n}$ be the normal subgroup generated
by one positive Dehn twist of each type in $\mathcal D$. Then
\begin{equation}\label{eq:boundary-group}
\pi_1^{\orb}((\M_{\mathcal D}^{\circ})^{\an})
\cong \G_{g,n}/N_{\mathcal D}.
\end{equation}
The same quotient describes
$\pi_1^{\orb}((\M_{\mathcal D})^{\an})$.
\end{proposition}
\begin{proof}
As in the proof of \cite[Prop.~3.1]{boggi_compact_type}, the inclusion
$\M_{g,n}\hookrightarrow\M_{\mathcal D}^{\circ}$ induces a surjection on
orbifold fundamental groups whose kernel is normally generated by
small meridians about the boundary components $D^\circ$, for
$D\in\mathcal D$. Under the identification
$\pi_1^{\orb}(\M_{g,n}^{\an})\cong\G_{g,n}$, such a meridian is the
positive Dehn twist about a separating curve of the corresponding
marked topological type. Since $\G_{g,n}$ acts transitively on curves
of each such type, their normal closure is $N_{\mathcal D}$. This
proves \eqref{eq:boundary-group}. Finally,
$\M_{\mathcal D}\setminus\M_{\mathcal D}^{\circ}$ has complex codimension
at least two, so the inclusion
$(\M_{\mathcal D}^{\circ})^{\an}\hookrightarrow
(\M_{\mathcal D})^{\an}$ induces an isomorphism on orbifold
fundamental groups.
\end{proof}

For $\mathcal D=\mathcal D_{h,\leq2}$, the group
$N_{\mathcal D}$ is exactly the subgroup $N_h$ used in
Section~\ref{sec:image-Kgn}. Let $\mathcal N_{\mathcal D}$ be the closed normal proalgebraic
subgroup of $\cG_{g,n}$ generated by the image of $N_{\mathcal D}$.
Since $N_{\mathcal D}$ is normal in $\G_{g,n}$ and the image of
$\G_{g,n}$ is Zariski dense in $\cG_{g,n}$, $\mathcal N_{\mathcal D}$
is simply the Zariski closure of the image of $N_{\mathcal D}$. Set
$\mathfrak n_{\mathcal D}:=\Lie\mathcal N_{\mathcal D}$.
It follows from Proposition~\ref{prop:relative-completion-quotient} that
$\u_{\M_{\mathcal D}^{\circ}}\cong\u_{g,n}/\mathfrak n_{\mathcal D}$.
Also set
\[
I_{\mathcal D}:=
\operatorname{span}_{\Q}
\{\operatorname{in}_{-2}(D):D\in N_{\mathcal D}\}
\subseteq\Gr^W_{-2}\u_{g,n}.
\]

\begin{lemma}\label{lem:boundary-ideal-weight-two}
One has
\[
\mathfrak n_{\mathcal D}\subseteq W_{-2}\u_{g,n},
\]
and the image of $\mathfrak n_{\mathcal D}$ in
$\Gr^W_{-2}\u_{g,n}$ is $I_{\mathcal D}$.
\end{lemma}
\begin{proof}
For $m\geq1$, write
$W_{-m}\U_{g,n}:=\exp(W_{-m}\u_{g,n})$.
Since $N_{\mathcal D}\subseteq K_{g,n}$,
Proposition~\ref{prop:Kgn-weight-minus-two} implies that its image
in $\U_{g,n}$ is contained in $W_{-2}\U_{g,n}$. Since
$W_{-2}\U_{g,n}$ is closed and normal in $\cG_{g,n}$, it contains
$\mathcal N_{\mathcal D}$. Hence
$\mathfrak n_{\mathcal D}\subseteq W_{-2}\u_{g,n}$.

Consider the quotient
$W_{-2}\U_{g,n}/W_{-3}\U_{g,n}$.
It is a vector group, and the logarithm induces an isomorphism of
vector groups
\[
W_{-2}\U_{g,n}/W_{-3}\U_{g,n}
\xrightarrow{\sim}
\Gr^W_{-2}\u_{g,n}.
\]
Via this isomorphism, we regard the linear subspace
$I_{\mathcal D}\subseteq\Gr^W_{-2}\u_{g,n}$ as a closed vector
subgroup of the quotient above. The image of each
$D\in N_{\mathcal D}$ corresponds to
$\operatorname{in}_{-2}(D)$, and therefore lies in
$I_{\mathcal D}$. Since $\mathcal N_{\mathcal D}$ is the Zariski
closure of the image of $N_{\mathcal D}$, its image in the quotient
is also contained in $I_{\mathcal D}$.

Taking Lie algebras, the image of $\mathfrak n_{\mathcal D}$ in
$\Gr^W_{-2}\u_{g,n}$ is therefore contained in $I_{\mathcal D}$.
Conversely, for every $D\in N_{\mathcal D}$,
$\log\widetilde\rho_{g,n}(D)\in \mathfrak n_{\mathcal D}$, so
$\operatorname{in}_{-2}(D)$ lies in the image of $\mathfrak n_{\mathcal D}$.
Taking their span gives the reverse inclusion.
\end{proof}

\begin{proposition}\label{prop:general-criterion}
If $I_{\mathcal D}=\Gr^W_{-2}\u_{g,n}$, then $\mathfrak n_{\mathcal D}=W_{-2}\u_{g,n}$.
Consequently, the inclusion
$\M_{\mathcal D}^{\circ}\hookrightarrow\M^c_{g,n}$ induces an
isomorphism on rational relative completions.
\end{proposition}
\begin{proof}
The hypothesis gives
\[
\mathfrak n_{\mathcal D}+W_{-3}\u_{g,n}=W_{-2}\u_{g,n}.
\]
Put $\mathfrak q=\u_{g,n}/\mathfrak n_{\mathcal D}$ and give it the quotient
filtration. Since $W_\bullet\u_{g,n}$ is the closed lower central
series, the quotient filtration is the closed lower central series of
$\mathfrak q$. Thus
\[
W_{-2}\mathfrak q=W_{-3}\mathfrak q,\qquad
W_{-r-1}\mathfrak q=
\overline{[\mathfrak q,W_{-r}\mathfrak q]}.
\]
It follows inductively that
$W_{-2}\mathfrak q=W_{-r}\mathfrak q$ for every $r\geq2$.
Since $\mathfrak n_{\mathcal D}$ is closed and $W_\bullet\u_{g,n}$ is the complete lower central series filtration, the induced filtration on $\mathfrak q$ is separated, and hence $W_{-2}\mathfrak q=0$. Therefore
$W_{-2}\u_{g,n}\subseteq \mathfrak n_{\mathcal D}$, and the reverse inclusion
was established above. Thus
$\mathfrak n_{\mathcal D}=W_{-2}\u_{g,n}$.
By Theorem~\ref{thm:compact-type-completion-char-zero},
\[
\Lie\mathcal N_{g,n}=W_{-2}\u_{g,n}.
\]
Since $N_{\mathcal D}\subseteq K_{g,n}$, one has
$\mathcal N_{\mathcal D}\subseteq\mathcal N_{g,n}$, and the equality
$\mathfrak n_{\mathcal D}=W_{-2}\u_{g,n}$ shows that these two closed
prounipotent subgroups have the same Lie algebra. Over a field of
characteristic zero, the exponential and logarithm give inverse
equivalences between prounipotent groups and pronilpotent Lie
algebras. Hence
$\mathcal N_{\mathcal D}=\mathcal N_{g,n}$.
Therefore $\cG_{\M_{\mathcal D}^{\circ}}$ and $\cG^c_{g,n}$ are the same
quotient of $\cG_{g,n}$.
\end{proof}

\subsection{The fixed-genus partial completions}
We now apply the fixed-genus spanning theorem to the selected boundary
opens. Put
\[
\mathfrak n_h:=\mathfrak n_{\mathcal D_{h,\leq2}}.
\]
\begin{theorem}[Fixed-genus partial completions]
\label{thm:fixed-genus-betti}
For every $1\leq h\leq g-1$,
\[
\mathfrak n_h=W_{-2}\u_{g,n}.
\]
Moreover, the inclusions induce isomorphisms
\[
\cG_{\M_{\mathcal D_{h,\leq2}}^{\circ}}
\xrightarrow{\sim}
\cG_{\M_{\mathcal D_h}^{\circ}}
\xrightarrow{\sim}
\cG^c_{g,n}.
\]
The same assertions hold for the corresponding larger opens
$\M_{\mathcal D_{h,\leq2}}$ and $\M_{\mathcal D_h}$. In
particular, their prounipotent radicals are commutative, with Lie
algebra
\[
\Lambda^3_0H\oplus H^{\oplus n}.
\]
\end{theorem}
\begin{proof}
By Theorem~\ref{thm:fixed-genus-spanning},
\(
I_{\mathcal D_{h,\leq2}}
=I_h
=\Gr^W_{-2}\u_{g,n}.
\)
Proposition~\ref{prop:general-criterion} therefore gives
\(
\mathfrak n_h=W_{-2}\u_{g,n}
\)
and an isomorphism
\[
\cG_{\M_{\mathcal D_{h,\leq2}}^{\circ}}
\xrightarrow{\sim}\cG^c_{g,n}.
\]
This morphism factors through
$\cG_{\M_{\mathcal D_h}^{\circ}}$. Since the inclusion-induced morphisms of
relative completions are surjective, both morphisms in
\[
\cG_{\M_{\mathcal D_{h,\leq2}}^{\circ}}
\longrightarrow
\cG_{\M_{\mathcal D_h}^{\circ}}
\longrightarrow
\cG^c_{g,n}
\]
are isomorphisms. The description of the prounipotent radicals follows
from Theorem~\ref{thm:compact-type-completion-char-zero}. The assertion
for the larger opens follows from the codimension-two comparison above.
\end{proof}

For any algebraically closed field $F$ of characteristic zero, the
same conclusions hold for continuous relative completion over $\Ql$,
for every prime $\ell$, by
Corollary~\ref{cor:rational-to-continuous-completion} and invariance
under algebraically closed base change. This proves the
characteristic-zero part of Theorem~\ref{thm:fixed-genus-main}.

\begin{remark}
The equality $\mathfrak n_h=W_{-2}\u_{g,n}$ concerns the closed normal image
of $N_h$ in $\cG_{g,n}$. It does not identify the intrinsic
unipotent completion of $N_h$ or $K_{g,n}$ with that image.
\end{remark}

% Replacement draft for Section 9. This is a section fragment and uses the
% notation, theorem environments, and bibliography of the manuscript.
% Existing labels are retained where their statements have been generalized.
%
% Additional bibliography entry, if it is not already in references.bib:
% @incollection{pikaart_dejong,
%   author    = {Pikaart, M. and de Jong, A. J.},
%   title     = {Moduli of curves with non-abelian level structure},
%   booktitle = {The Moduli Space of Curves},
%   editor    = {Dijkgraaf, R. and Faber, C. and van der Geer, G.},
%   series    = {Progress in Mathematics},
%   volume    = {129},
%   publisher = {Birkh\"auser},
%   year      = {1995},
%   pages     = {483--510}
% }

%%%%%%%%%%%%%%%%%%%%%%%%%%%%%%%%%%%%%%%%%%%%%%%%%%%%%%%%%%%%%%%%%%%%%%%%%%%%
\section{Specialization and Comparison in Characteristic
\texorpdfstring{$p$}{p}}
\label{sec:specialization-characteristic-p}
%%%%%%%%%%%%%%%%%%%%%%%%%%%%%%%%%%%%%%%%%%%%%%%%%%%%%%%%%%%%%%%%%%%%%%%%%%%%

Throughout this section, $k$ is an algebraically closed field of
characteristic $p>0$, and $\ell\neq p$ is an odd prime.
Let $\mathcal D$ be any collection of separating boundary types.
We compare the continuous relative completion of
$\pi_1(\M_{\mathcal D/k})$ with the relative completion computed in
characteristic zero, compatibly with enlarging $\mathcal D$.
Taking $\mathcal D$ to be the collection of all separating boundary
types recovers $\M^c_{g,n}$.
After choosing geometric base points below, we write $H_{\bar\eta,\Ql}$
and $H_{\bar\xi,\Ql}$ for the corresponding symplectic representations
on the generic and special fibers.
Purity will give the same comparison for the one-node opens
$\M_{\mathcal D}^{\circ}$.

\subsection{A smooth level cover}
\label{subsec:compact-type-level-cover}

Choose an auxiliary base point
$x_0\in\Sigma_g\setminus\{x_1,\ldots,x_n\}$ and put
\[
\Pi^{\mathrm{pun}}_{g,n}
:=\pi_1^{\top}\bigl(
\Sigma_g\setminus\{x_1,\ldots,x_n\},x_0\bigr).
\]
This is the punctured surface group, with the convention that it is
the closed surface group when $n=0$. It is distinct from the pure
surface braid group $\Pi_{g,n}$ and its unipotent completion
$\cP_{g,n}$.
For $n>0$, let $W^\bullet\Pi^{\mathrm{pun}}_{g,n}$ be the weight
filtration of Boggi--Pikaart; when $n=0$, it is the lower central
series. Consider the geometric level
\[
\Gamma^\lambda
:=\ker\left(
\G_{g,n}\longrightarrow
\Out\left(
\Pi^{\mathrm{pun}}_{g,n}/
\bigl(W^3\Pi^{\mathrm{pun}}_{g,n}\cdot
(\Pi^{\mathrm{pun}}_{g,n})^\ell\bigr)
\right)\right).
\]
Here $(\Pi^{\mathrm{pun}}_{g,n})^\ell$ denotes the subgroup normally
generated by the $\ell$th powers.
When $n=0$, this is the level of Pikaart--de Jong with $k=2$ and
level modulus $\ell$; see
\cite[Prop.~2.3.6 and Thms.~3.1.1, 3.1.3]{pikaart_dejong}.
When $n>0$, it is the level $\Gamma_{g,n}^{4,\ell}$ of
Boggi--Pikaart. In their refined filtration, one has
$V^4\Pi^{\mathrm{pun}}_{g,n}=W^3\Pi^{\mathrm{pun}}_{g,n}$;
the term $V^4$ is not $W^4$.
The arithmetic construction in
\cite[Prop.~2.6, case $k=4$]{boggi_pikaart} supplies this pointed
level, and its local monodromy is given by their Proposition~2.8.

Choose the corresponding finite \'etale Galois level cover
\[
q:M^\lambda\longrightarrow\M_{g,n/\Z[1/\ell]},
\]
and let
\[
\overline q:\overline M^\lambda
\longrightarrow\Mbar_{g,n/\Z[1/\ell]}
\]
be the normalization of the Deligne--Mumford compactification with
respect to $M^\lambda$.
For the choices just made, $\overline M^\lambda$ is a proper smooth
scheme over $\Z[1/\ell]$, and its boundary
$E^\lambda:=\overline M^\lambda-M^\lambda$ is a relative
normal-crossings divisor. These properties follow from the cited
constructions and their local-monodromy descriptions.

\subsection{The level covers of the selected opens}
\label{subsec:selected-level-opens}
\label{subsec:compact-type-level-fundamental-group}

Let $\Delta_{\mathrm{irr}}$ denote the reduced complement of compact
type in $\Mbar_{g,n}$. Define
\[
M_{\mathcal D}^\lambda
:=\overline M^\lambda\times_{\Mbar_{g,n}}\M_{\mathcal D},
\qquad
D_{\mathcal D}^\lambda
:=\overline M^\lambda-M_{\mathcal D}^\lambda,
\]
where the complement is taken as the union of the corresponding
components of $E^\lambda$, each taken once.
Since $\M_{\mathcal D}\hookrightarrow\Mbar_{g,n}$ is an open
immersion, its base change identifies $M_{\mathcal D}^\lambda$ with
an open subscheme of $\overline M^\lambda$.
The complement of $\M_{\mathcal D}$ is the union of
$\Delta_{\mathrm{irr}}$ and the separating boundary divisors whose
types are not in $\mathcal D$.
Thus $D_{\mathcal D}^\lambda$ is the union of the components of
$E^\lambda$ lying over the omitted boundary divisors.
In local normal-crossing coordinates it has an equation
\[
s_{e_1}\cdots s_{e_t}=0,
\]
where only a subset of the boundary parameters is used. Therefore
$(\overline M^\lambda,D_{\mathcal D}^\lambda)$ is a smooth proper
pair with relative normal-crossings boundary, and
$M_{\mathcal D}^\lambda=\overline M^\lambda-D_{\mathcal D}^\lambda$
is smooth over $\Z[1/\ell]$.

For $\mathcal D=\varnothing$, we have
$\M_{\mathcal D}=\M_{g,n}$ and $M_{\mathcal D}^\lambda=M^\lambda$.
When $\mathcal D$ consists of all separating boundary types, we have
$\M_{\mathcal D}=\M^c_{g,n}$ and write
$M^{c,\lambda}=M_{\mathcal D}^\lambda$.
Thus the compact-type level cover is included in this construction.

\begin{proposition}[\'Etaleness over the selected opens]
\label{prop:selected-level-etale}
\label{prop:compact-type-level-etale}
For every collection $\mathcal D$ of separating boundary types, the
induced morphism
\[
q_{\mathcal D}:M_{\mathcal D}^\lambda\longrightarrow\M_{\mathcal D}
\]
is finite \'etale. When $\mathcal D$ contains all separating types,
we write $q^c:M^{c,\lambda}\to\M^c_{g,n}$ for this cover.
\end{proposition}

\begin{proof}
The morphism is finite because it is the base change of $\overline q$.
We first prove \'etaleness over $\C$. Let
$[C]\in \M_{\mathcal D/\C}\setminus\M_{g,n/\C}$, and let $\Gamma$ be its
dual graph. Write $E(\Gamma)$ for the set of edges of $\Gamma$, one
for each node of $C$, and let $B(\Gamma)\subseteq E(\Gamma)$ be the
set of separating edges, that is, the edges whose removal disconnects
$\Gamma$. The local monodromy lattice is
\[
L_\Gamma=\bigoplus_{e\in E(\Gamma)}\Z\varepsilon_e,
\]
where $\varepsilon_e$ is the homotopy class of a positively oriented
meridian loop around the boundary branch corresponding to $e$.
The local monodromy homomorphism
\[
\mu_\Gamma:L_\Gamma\longrightarrow\G_{g,n}
\]
sends $\varepsilon_e$ to the Dehn twist about the vanishing cycle
corresponding to the node $e$. Define
\[
L_\Gamma^\lambda
:=\mu_\Gamma^{-1}(\Gamma^\lambda)
=\ker\left(
L_\Gamma\longrightarrow\G_{g,n}/\Gamma^\lambda
\right).
\]
Thus $L_\Gamma^\lambda$ consists of the local loops that lift to
closed loops in the level cover.
For $n=0$, Theorem~3.1.3(2) of \cite{pikaart_dejong}, with $k=2$
and odd modulus $\ell$, gives $L_\Gamma^\lambda=\ell F^0+F^2$,
where $F^0=L_\Gamma$ and $F^2$ is generated by the separating edges.
For $n>0$, we use \cite[Prop.~2.8]{boggi_pikaart}.
In either case,
\[
L_\Gamma^\lambda
=
\bigoplus_{e\notin B(\Gamma)}\ell\Z\varepsilon_e
\oplus
\bigoplus_{e\in B(\Gamma)}\Z\varepsilon_e.
\]
Consequently, the ramification index along the boundary branch
corresponding to $e$ is
\[
[\Z\varepsilon_e:
L_\Gamma^\lambda\cap\Z\varepsilon_e]
=
\begin{cases}
\ell,&e\notin B(\Gamma),\\
1,&e\in B(\Gamma).
\end{cases}
\]
Since $\M_{\mathcal D}\subseteq\M^c_{g,n}$, the curve $C$ is of compact type.
Thus $\Gamma$ is a tree, so every edge of $\Gamma$ is separating.
Hence $B(\Gamma)=E(\Gamma)$ and
$L_\Gamma^\lambda=L_\Gamma$. In the local polydisc description of
the level cover, all ramification indices are therefore equal to $1$.
Thus the normalization is \'etale over $[C]$. Since the cover is
already \'etale over the smooth locus, $q_{\mathcal D,\C}$ is finite \'etale.
Since $\C/\Q$ is faithfully flat and \'etaleness descends under
faithfully flat base change, $q_{\mathcal D,\Q}$ is finite \'etale.

It remains to exclude ramification in the special fibers.
The source is normal and $\M_{\mathcal D}$ is regular. Purity of the branch locus,
applied after passing to an \'etale atlas of $\M_{\mathcal D}$,
\cite[Exp.~X, Thm.~3.1]{SGA1}, implies that any nonempty branch locus
has a codimension-one component. Such a component cannot dominate
$\Spec\Z[1/\ell]$, since the generic fiber is \'etale.
A component supported over $p\neq\ell$ would be an irreducible
component of $\M_{\mathcal D/\F_p}$. Its generic point lies outside the
relative boundary divisor, hence in $\M_{g,n/\F_p}$, where
$q_{\mathcal D}$ agrees with the
finite \'etale cover $q$. Thus the branch locus is empty.
\end{proof}

Thus $M_{\mathcal D}^\lambda$ is a smooth finite \'etale cover of $\M_{\mathcal D}$. The morphism gives,
by pullback, a family of stable compact-type curves over $M_{\mathcal D}^\lambda$.
We use only this pullback family and the fact that $M_{\mathcal D}^\lambda$ is a scheme;
no separate modular interpretation of the level structure at the nodes
is required.

Define
\[
M_{\mathcal D}^{\circ,\lambda}
:=M_{\mathcal D}^\lambda
\times_{\M_{\mathcal D}}\M_{\mathcal D}^{\circ}.
\]
The complement of $M_{\mathcal D}^{\circ,\lambda}$ in $M_{\mathcal D}^\lambda$ has codimension at least two,
as does the corresponding complement downstairs. The schemes are
regular. Purity of the branch locus gives an equivalence of finite
\'etale cover categories
\[
\mathrm{FEt}(M_{\mathcal D}^\lambda)
\xrightarrow{\sim}\mathrm{FEt}(M_{\mathcal D}^{\circ,\lambda}).
\]
The same conclusion holds on the stacks by descent on regular \'etale
charts. It holds over $W(k)$ and over each geometric fiber.
Thus these inclusions induce isomorphisms of full fundamental groups.
It suffices to construct comparison for the larger opens $\M_{\mathcal D}$ and
their level covers $M_{\mathcal D}^\lambda$.

Choose a connected component of $M^\lambda_\C$, together with a base
point, and the corresponding component of $M_{\mathcal D/\C}^\lambda$.
By an abuse of notation, denote these components again by
$M^\lambda_\C$ and $M_{\mathcal D/\C}^\lambda$.
The cover identifies $\pi_1^{\top}(M^\lambda_\C)$ with the
finite-index subgroup $\Gamma^\lambda\subseteq\G_{g,n}$.
Recall that $N_{\mathcal D}\unlhd\G_{g,n}$ is the normal subgroup
generated by one positive Dehn twist of each type in $\mathcal D$.

\begin{proposition}
\label{prop:compact-type-level-fundamental-group}
For compatible connected components and base points, there is a
natural isomorphism
$\pi_1^{\top}(M_{\mathcal D/\C}^\lambda)
\cong\Gamma^\lambda/N_{\mathcal D}$.
\end{proposition}

\begin{proof}
Since $M^\lambda_\C$ is the pullback of the connected finite
\'etale cover $q_{\mathcal D,\C}$ to $\M_{g,n/\C}$,
the subgroup correspondence and
Proposition~\ref{prop:boundary-group} give
$N_{\mathcal D}\subseteq\Gamma^\lambda$ and the commutative diagram
% By construction,
% \[
% M^\lambda_\C
% =\M_{g,n/\C}\times_{\M_{\mathcal D/\C}}
% M_{\mathcal D/\C}^\lambda.
% \]
% Since $q_{\mathcal D,\C}$ is a connected finite \'etale cover, the
% subgroup correspondence, together with
% Proposition~\ref{prop:boundary-group}, shows that
% $N_{\mathcal D}\subseteq\Gamma^\lambda$ and gives the commutative
% diagram
\[
\begin{tikzcd}[column sep=small]
1\arrow[r]&
N_{\mathcal D}\arrow[r]\arrow[d,equal]&
\Gamma^\lambda\arrow[r]\arrow[d,hook]&
\pi_1^{\top}(M_{\mathcal D/\C}^\lambda)
\arrow[r]\arrow[d,hook]&1\\
1\arrow[r]&
N_{\mathcal D}\arrow[r]&
\G_{g,n}\arrow[r]&
\pi_1^{\orb}((\M_{\mathcal D})^{\an})\arrow[r]&1,
\end{tikzcd}
\]
with exact rows.
The top row gives the assertion.
\end{proof}

\subsection{Specialization of the level covers}
\label{subsec:selected-level-specialization}

Let $W=W(k)$ be the ring of Witt vectors of $k$, let $L$ be its
fraction field, and fix an algebraic closure $\bar L$.
Since $\ell\neq p$, there is a morphism
$\Spec W\to\Spec\Z[1/\ell]$.
After base change to $W$, fix a connected component of
$\overline M_W^\lambda$ and use the same notation for it and for
the corresponding opens and geometric fibers.
Since $W$ is strictly henselian and $\overline M_W^\lambda$ is
proper and smooth, its geometric fibers are connected.
They are smooth, hence irreducible, so their nonempty open subsets
$M_A^\lambda$ and $M_{\mathcal D/A}^\lambda$ are connected for $A=\bar L,k$.
The same holds for the opens over $W$.
Let $\Delta^\lambda$ be the stabilizer of the chosen component in
the deck group. It is the common deck group of the covers
$q_{\mathcal D,W}$ and their geometric fibers, for all $\mathcal D$.
Characteristic-zero comparison identifies it with
$\G_{g,n}/\Gamma^\lambda$, after compatible choices of base points.

By construction, $\overline M_W^\lambda$ is proper and smooth over
$W$, $D_{\mathcal D/W}^\lambda$ is a relative normal-crossings divisor, and
\[
M_{\mathcal D/W}^\lambda=\overline M_W^\lambda-D_{\mathcal D/W}^\lambda.
\]
Choose geometric base points $\bar\eta\in M^\lambda_{\bar L}$ and
$\bar\xi\in M^\lambda_k$. Regard both as geometric points of
$M_W^\lambda$ and fix an \'etale path $\beta$ from $\bar\eta$ to
$\bar\xi$ in this common smooth level locus.
Use the induced paths in all the $M_{\mathcal D/W}^\lambda$.
Use the same symbols for the images of the base points downstairs.
Unless otherwise indicated, fundamental groups are taken at these
chosen base points.
The geometric-fiber maps and $\beta$ determine
\begin{equation}
\label{eq:specialization-level-cover}
\operatorname{sp}_{\mathcal D}^{(\ell)}:
\pi_1(M_{\mathcal D/\bar L}^\lambda)^{(\ell)}
\longrightarrow\pi_1(M_{\mathcal D/k}^\lambda)^{(\ell)}.
\end{equation}
The tame specialization theorem makes it an isomorphism;
see \cite[Exp.~XIII, Cor.~2.9 and \S2.10]{SGA1}.
More precisely, with the two total-space base points identified
using $\beta$, the geometric-fiber maps give
\begin{equation}
\label{eq:selected-level-zigzag}
\pi_1(M_{\mathcal D/\bar L}^\lambda)^{(\ell)}
\xrightarrow{\sim}\pi_1(M_{\mathcal D/W}^\lambda)^{(\ell)}
\xleftarrow{\sim}\pi_1(M_{\mathcal D/k}^\lambda)^{(\ell)}.
\end{equation}
All finite $\ell$-group torsors involved are tame along the boundary.
The maps are functorial for enlarging $\mathcal D$, and hence for
the inclusions $M_{\mathcal D}^\lambda\hookrightarrow M^{c,\lambda}$.

Let $C_{\bar\eta},C_{\bar\xi}$ be the smooth curves at the base points,
and put $H_{t,\Zl}=H^1_{\et}(C_t,\Zl(1))$ for
$t=\bar\eta,\bar\xi$.
Write
\[
\Sp(H_{t,\Zl})[\ell]
=\ker\bigl\{
\Sp(H_{t,\Zl})\to\Sp(H_{t,\Zl}/\ell H_{t,\Zl})
\bigr\}.
\]
Since the level cover dominates full level $\ell$, the images of
$\rho_{\bar\eta}:\pi_1(M_{\mathcal D/\bar L}^\lambda)\to\Sp(H_{\bar\eta,\Zl})$
and
$\rho_{\bar\xi}:\pi_1(M_{\mathcal D/k}^\lambda)\to\Sp(H_{\bar\xi,\Zl})$
lie in these congruence subgroups. They are pro-$\ell$, so both
representations factor through the maximal pro-$\ell$ quotients.

Let $J\to M_{\mathcal D/W}^\lambda$ be the relative Jacobian, and set
$T_\ell J:=\varprojlim_r J[\ell^r]$.
This is a lisse $\Zl$-sheaf with
\[
(T_\ell J)_t\cong H^1_{\et}(C_t,\Zl(1))=H_{t,\Zl}.
\]
It is the pullback of the lisse sheaf $\mathbb H^c_{\Zl}$ on
$\M^c_{g,n/W}$ introduced in
Subsection~\ref{subsec:jacobian-monodromy-completion}.
Transport along $\beta$, together with the polarization pairing,
gives a symplectic isomorphism
$\alpha_{\Zl}:H_{\bar\eta,\Zl}\xrightarrow{\sim}H_{\bar\xi,\Zl}$,
and hence an induced conjugation isomorphism
$\alpha_{\Zl,*}:\Sp(H_{\bar\eta,\Zl})
\xrightarrow{\sim}\Sp(H_{\bar\xi,\Zl})$.
These identifications are the same for every $\mathcal D$.

\begin{lemma}[Compatibility with symplectic monodromy]
\label{lem:specialization-monodromy}
For every $\mathcal D$, the diagram
\[
\begin{tikzcd}
\pi_1(M_{\mathcal D/\bar L}^\lambda)^{(\ell)}
\arrow[r,"\operatorname{sp}_{\mathcal D}^{(\ell)}","\sim"']
\arrow[d,"\rho_{\bar\eta}^{(\ell)}"']&
\pi_1(M_{\mathcal D/k}^\lambda)^{(\ell)}
\arrow[d,"\rho_{\bar\xi}^{(\ell)}"]\\
\Sp(H_{\bar\eta,\Zl})
\arrow[r,"\alpha_{\Zl,*}","\sim"']&
\Sp(H_{\bar\xi,\Zl})
\end{tikzcd}
\]
commutes.
\end{lemma}

\begin{proof}
The two representations are monodromies of the two fibers of the same
lisse sheaf $T_\ell J$. Functoriality of specialization, with
base-point comparison determined by $\beta$, gives
$\alpha_{\Zl,*}\circ\rho_{\bar\eta}^{(\ell)}
=\rho_{\bar\xi}^{(\ell)}\circ\operatorname{sp}_{\mathcal D}^{(\ell)}$.
\end{proof}

For $A\in\{\bar L,W,k\}$, consider $\pi_1(\M_{\mathcal D/A})$.
Over $\bar L$ and $k$, use the images of $\bar\eta$ and $\bar\xi$,
respectively, as base points. Over $W$, use either of these geometric
points, indicating the base point when necessary.
Write $\pi_1(\M_{\mathcal D/A})^\rel$ for its relative pro-$\ell$
completion with respect to the integral Jacobian monodromy
representation, regarded as a continuous surjection onto its
closed image.
Since the kernel of reduction modulo $\ell$ in this image is
pro-$\ell$, Lemma~\ref{lem:relative-pro-l-change-of-target} allows
us to compute the same completion relative to the mod-$\ell$
monodromy image.

\begin{lemma}\label{lem:level-mod-l-kernel}
The mod-$\ell$ Jacobian monodromy on $\M_{\mathcal D/A}$, for
$A\in\{\bar L,W,k\}$ and with the base points chosen above,
factors through $\Delta^\lambda$.
Under the chosen identifications, the kernel of the resulting
representation of $\Delta^\lambda$ is
\[
\G_{g,n}[\ell]/\Gamma^\lambda,
\]
and is a finite $\ell$-group. Here $\G_{g,n}[\ell]$ denotes the
full symplectic level-$\ell$ subgroup.
These identifications are compatible with the geometric-fiber maps,
change of base point along $\beta$, and enlarging $\mathcal D$.
\end{lemma}

\begin{proof}
Over $\bar L$, characteristic-zero comparison identifies the deck
group with $\G_{g,n}/\Gamma^\lambda$.
Since $\Gamma^\lambda\subseteq\G_{g,n}[\ell]$, the mod-$\ell$
monodromy factors through this quotient, and its kernel is
$\G_{g,n}[\ell]/\Gamma^\lambda$.
We first show that this quotient is a finite $\ell$-group.
Put
\[
\Pi^{\mathrm{pun}}_\lambda
:=\Pi^{\mathrm{pun}}_{g,n}/
\bigl(W^3\Pi^{\mathrm{pun}}_{g,n}\cdot
(\Pi^{\mathrm{pun}}_{g,n})^\ell\bigr),
\]
and let $B_\ell$ be the image of $W^2\Pi^{\mathrm{pun}}_{g,n}$.
Write $H_{\F_\ell}=H_1(\Sigma_g,\F_\ell)$.
Then
\[
1\longrightarrow B_\ell
\longrightarrow\Pi^{\mathrm{pun}}_\lambda
\longrightarrow H_{\F_\ell}\longrightarrow1
\]
is a central extension of finite groups of exponent $\ell$. Since
$[\Pi^{\mathrm{pun}}_{g,n},W^2\Pi^{\mathrm{pun}}_{g,n}]
\subseteq W^3\Pi^{\mathrm{pun}}_{g,n}$,
the subgroup $B_\ell$ is central in
$\Pi^{\mathrm{pun}}_\lambda$.
The group $B_\ell$ is generated by commutators and, when $n>0$,
puncture classes.
A geometric mapping class acting trivially on $H_{\F_\ell}$ fixes
commutators, because $B_\ell$ is central, and fixes puncture classes,
because the markings are ordered and these classes are central.
It therefore fixes $B_\ell$ pointwise.
For an automorphism $f$ inducing the identity on both $B_\ell$ and
$H_{\F_\ell}$, the assignment
\[
\bar x\longmapsto f(x)x^{-1}
\]
is a well-defined homomorphism $H_{\F_\ell}\to B_\ell$.
It embeds this automorphism group into
$\Hom_{\F_\ell}(H_{\F_\ell},B_\ell)$.
Passing to outer automorphisms takes a quotient of a finite
$\ell$-group. Since $\G_{g,n}[\ell]/\Gamma^\lambda$ is the image
of $\G_{g,n}[\ell]$ in this outer automorphism group, it is itself
a finite $\ell$-group.

The sheaf $\mathbb H^c_{\Zl}/\ell\mathbb H^c_{\Zl}$ becomes trivial
on $M_W^\lambda$, hence also on every $M_{\mathcal D/W}^\lambda$,
since the dense open immersion into this normal scheme induces a
surjection on fundamental groups. The action of
$\pi_1(\M_{\mathcal D/W},\bar\eta)$ on
$H_{\bar\eta,\Zl}/\ell H_{\bar\eta,\Zl}$ therefore factors through
$\Delta^\lambda$.
Base-point change along the image of $\beta$ downstairs induces
the identity on $\Delta^\lambda$. Transport along $\beta$ therefore
identifies the two $\Delta^\lambda$-representations. The
geometric-fiber maps likewise induce the identity on the deck group.
Under the fixed identification
$\Delta^\lambda\cong\G_{g,n}/\Gamma^\lambda$, the kernels of these
representations coincide with the subgroup
$\G_{g,n}[\ell]/\Gamma^\lambda$ of $\Delta^\lambda$.
This holds for $A=\bar L,W,k$, with either chosen base point
when $A=W$. Since all the sheaves and covers are restrictions
of the same ones over compact type, this identification is also
compatible with enlarging $\mathcal D$.
\end{proof}
The following form of finite-cover descent will be convenient.

\begin{lemma}\label{lem:finite-cover-descent}
Let $N$ be an open normal subgroup of a profinite group $\Pi$ and
suppose that a continuous surjection $\Pi\to S$ onto a finite group
factors through $\Delta=\Pi/N$, with
$\ker(\Delta\to S)$ a finite $\ell$-group.
Put $N_0=\ker(N\to N^{(\ell)})$.
Then the relative pro-$\ell$ completion $\Pi^\rel$ with respect
to $\Pi\to S$ is canonically isomorphic to $\Pi/N_0$, and fits into
an exact sequence
\[
1\longrightarrow N^{(\ell)}
\longrightarrow \Pi^\rel
\longrightarrow\Delta\longrightarrow1.
\]
\end{lemma}

\begin{proof}
Since $N_0$ is a closed characteristic subgroup of $N$, it is closed
and normal in $\Pi$. The canonical map $N\to N^{(\ell)}$ is
surjective, so $N/N_0\cong N^{(\ell)}$, and there is an exact sequence
\[
1\longrightarrow N^{(\ell)}\longrightarrow\Pi/N_0
\longrightarrow\Delta\longrightarrow1.
\]
The kernel of $\Pi/N_0\to S$ is an extension of
$\ker(\Delta\to S)$ by $N^{(\ell)}$, and is therefore pro-$\ell$.
Thus $\Pi\to\Pi/N_0$ is a relative pro-$\ell$ quotient over $S$.
The universal property gives a continuous homomorphism
$\varphi:\Pi^\rel\to\Pi/N_0$ through which the quotient map
$\Pi\to\Pi/N_0$ factors.

Conversely, let $j:\Pi\to\Pi^\rel$ be the canonical map.
The image of $N$ lies in the pro-$\ell$ group
$\ker(\Pi^\rel\to S)$, so $j|_N$ factors through $N^{(\ell)}$.
In particular, $j$ kills $N_0$ and induces a continuous homomorphism
$\psi:\Pi/N_0\to\Pi^\rel$.
By construction, $\varphi\circ\psi$ is the identity after composition
with the surjection $\Pi\to\Pi/N_0$, and $\psi\circ\varphi$ is the
identity on the image of $j$. This image is dense in $\Pi^\rel$,
so continuity gives $\psi\circ\varphi=\id$ on $\Pi^\rel$.
Thus $\varphi$ and $\psi$ are inverse isomorphisms.
The preceding exact sequence gives the assertion.
\end{proof}

\begin{lemma}\label{lem:compact-type-relative-pro-l-exact}
For $A\in\{\bar L,W,k\}$, there is a natural exact sequence
\[
1\longrightarrow\pi_1(M_{\mathcal D/A}^\lambda)^{(\ell)}
\longrightarrow\pi_1(\M_{\mathcal D/A})^\rel
\longrightarrow\Delta^\lambda\longrightarrow1.
\]

\end{lemma}

\begin{proof}
Apply Lemma~\ref{lem:finite-cover-descent} with
$\Pi=\pi_1(\M_{\mathcal D/A})$,
$N=\pi_1(M_{\mathcal D/A}^\lambda)$, and $S$ the mod-$\ell$
monodromy image.
Lemma~\ref{lem:level-mod-l-kernel} verifies the kernel condition.
The same argument applies over $W$ with either base point.
\end{proof}

To pass from specialization of the level cover to relative
pro-$\ell$ completion of $\M_{\mathcal D}$, we compare these exact sequences over
$\bar L$ and $k$ through the corresponding sequences over $W$.
The following lemma records their compatibility with the
geometric-fiber maps, change of base point along $\beta$, and
symplectic monodromy.

\begin{lemma}[Comparison over the Witt ring]
\label{lem:compact-type-level-specialization-compatibility}
Let $\iota_{\bar\eta}$ and $\iota_{\bar\xi}$ denote the maps induced
by the geometric-fiber morphisms. These maps and $\beta$ give
\[
\begin{tikzcd}[
column sep=1.2em,
row sep=1.8em,
cells={nodes={font=\small}}
]
1\arrow[r]&
\pi_1(M_{\mathcal D/\bar L}^\lambda)^{(\ell)}\arrow[r,hook]
\arrow[d,"\iota_{\bar\eta}^{(\ell)}"']&
\pi_1(\M_{\mathcal D/\bar L})^\rel\arrow[r]
\arrow[d,"\iota_{\bar\eta}^{\rel}"']&
\Delta^\lambda\arrow[r]\arrow[d,equal]&1\\
1\arrow[r]&
\pi_1(M_{\mathcal D/W}^\lambda,\bar\eta)^{(\ell)}\arrow[r,hook]
\arrow[d,"\beta_*^{(\ell)}"']&
\pi_1(\M_{\mathcal D/W},\bar\eta)^\rel\arrow[r]
\arrow[d,"\beta_*^\rel"']&
\Delta^\lambda\arrow[r]\arrow[d,equal]&1\\
1\arrow[r]&
\pi_1(M_{\mathcal D/W}^\lambda,\bar\xi)^{(\ell)}\arrow[r,hook]&
\pi_1(\M_{\mathcal D/W},\bar\xi)^\rel\arrow[r]&
\Delta^\lambda\arrow[r]&1\\
1\arrow[r]&
\pi_1(M_{\mathcal D/k}^\lambda)^{(\ell)}\arrow[r,hook]
\arrow[u,"\iota_{\bar\xi}^{(\ell)}"]&
\pi_1(\M_{\mathcal D/k})^\rel\arrow[r]
\arrow[u,"\iota_{\bar\xi}^{\rel}"]&
\Delta^\lambda\arrow[r]\arrow[u,equal]&1.
\end{tikzcd}
\]
All rows are exact and all arrows in the first nontrivial column are
isomorphisms. The diagram is compatible with symplectic monodromy.
For $t=\bar\eta,\bar\xi$, write
$\rho_{W,t}^{\rel}$ for the induced monodromy on
$\pi_1(\M_{\mathcal D/W},t)^\rel$ and
$\rho_{W,t}^{(\ell)}$ for the monodromy on the level-cover kernel.
More precisely, the change-of-base-point maps give
\[
\begin{tikzcd}[
row sep=1em,
column sep=1.5em,
cells={nodes={font=\small}}
]
\pi_1(M_{\mathcal D/W}^\lambda,\bar\eta)^{(\ell)}\arrow[rr,hook]
\arrow[dr,"\rho_{W,\bar\eta}^{(\ell)}"']
\arrow[dd,"\beta_*^{(\ell)}"']&&
\pi_1(\M_{\mathcal D/W},\bar\eta)^\rel
\arrow[dl,"\rho_{W,\bar\eta}^{\rel}"]
\arrow[dd,"\beta_*^\rel"]\\
&\Sp(H_{\bar\eta,\Zl})
\arrow[dd,"\alpha_{\Zl,*}"'{pos=0.25}]&\\
\pi_1(M_{\mathcal D/W}^\lambda,\bar\xi)^{(\ell)}\arrow[rr,hook]
\arrow[dr,"\rho_{W,\bar\xi}^{(\ell)}"']&&
\pi_1(\M_{\mathcal D/W},\bar\xi)^\rel
\arrow[dl,"\rho_{W,\bar\xi}^{\rel}"]\\
&\Sp(H_{\bar\xi,\Zl})&
\end{tikzcd}
\]
and the analogous diagrams for the two fiber-to-$W$ maps commute.
\end{lemma}

\begin{proof}
The chosen component of $q_{\mathcal D,W}$ is a connected Galois cover with
group $\Delta^\lambda$.
As established above, the mod-$\ell$ monodromy is trivial on
$M_{\mathcal D/W}^\lambda$, and the induced representation of $\Delta^\lambda$ has
finite $\ell$-group kernel.
Lemma~\ref{lem:compact-type-relative-pro-l-exact} therefore
gives all four rows.
Functoriality of pro-$\ell$ and relative pro-$\ell$ completion
gives the comparison maps and their compatibility with the
horizontal maps. The geometric-fiber maps preserve the chosen
base points, and $\beta$ is a path in the level cover, so the
induced maps on the deck quotient $\Delta^\lambda$ are the identity.
Indeed, the image of $\beta$ downstairs lifts from the chosen point
upstairs to $\beta$ itself, so no deck transformation is introduced
at the endpoint. This holds simultaneously for every $\mathcal D$ and gives
the first commutative diagram.

By \eqref{eq:selected-level-zigzag},
$\iota_{\bar\eta}^{(\ell)}$ and $\iota_{\bar\xi}^{(\ell)}$
are isomorphisms. The map $\beta_*^{(\ell)}$ is an
isomorphism by change of base point along $\beta$.
% The homomorphisms $\iota_{\bar\eta}^{(\ell)}$ and
% $\iota_{\bar\xi}^{(\ell)}$, induced respectively by
% \[
% M_{\mathcal D/\bar L}^\lambda\longrightarrow M_{\mathcal D/W}^\lambda,
% \qquad
% M_{\mathcal D/k}^\lambda\longrightarrow M_{\mathcal D/W}^\lambda,
% \]
% are isomorphisms by tame specialization
% \cite[Exp.~XIII, Cor.~2.9 and \S2.10]{SGA1}.
% The homomorphism $\beta_*^{(\ell)}$ is the isomorphism obtained
% by changing the base point of $M_{\mathcal D/W}^\lambda$ from $\bar\eta$ to
% $\bar\xi$ along $\beta$.

Finally, all monodromy representations arise from the lisse
sheaf $\mathbb H^c_{\Zl}$ on $\M^c_{g,n/W}$, whose pullback
to $M_{\mathcal D/W}^\lambda$ identifies with $T_\ell J$.
Pullback gives compatibility with the geometric-fiber maps,
and transport along $\beta$ gives $\alpha_{\Zl}$.
The level-cover actions are restrictions of the actions on
$\pi_1(\M_{\mathcal D/W})$.
These compatibilities pass to the indicated completions by
functoriality, proving the remaining assertions.
\end{proof}

The comparison of exact sequences now yields the desired
specialization isomorphism for the relative pro-$\ell$ completions,
compatible with symplectic monodromy.

\begin{proposition}[Specialization of relative pro-$\ell$ completion]
\label{prop:compact-type-relative-pro-l-specialization}
\label{prop:selected-functorial-comparison}
For every collection $\mathcal D$, the specialization data determine an
isomorphism
\[
\operatorname{sp}_{\mathcal D}^\rel:
\pi_1(\M_{\mathcal D/\bar L})^\rel\xrightarrow{\sim}\pi_1(\M_{\mathcal D/k})^\rel
\]
compatible with symplectic monodromy through $\alpha_{\Zl,*}$.
These isomorphisms are functorial for enlarging $\mathcal D$.
The same statements hold for $\M_{\mathcal D}^{\circ}$.
\end{proposition}

\begin{proof}
The short five lemma applied to the four-row diagram shows that all
three comparison maps in the middle column are isomorphisms.
Define
\[
\operatorname{sp}_{\mathcal D}^\rel
=(\iota_{\bar\xi}^\rel)^{-1}
\circ\beta_*^\rel\circ\iota_{\bar\eta}^\rel.
\]
It restricts on the level-cover kernels to
$(\iota_{\bar\xi}^{(\ell)})^{-1}\circ\beta_*^{(\ell)}
\circ\iota_{\bar\eta}^{(\ell)}=\operatorname{sp}_{\mathcal D}^{(\ell)}$
and induces the identity on $\Delta^\lambda$.
The monodromy compatibilities of the preceding lemma give
\[
\begin{tikzcd}
\pi_1(\M_{\mathcal D/\bar L})^\rel
\arrow[r,"\operatorname{sp}_{\mathcal D}^\rel","\sim"']
\arrow[d,"\rho_{\bar\eta}^{\mathcal D,\rel}"']&
\pi_1(\M_{\mathcal D/k})^\rel
\arrow[d,"\rho_{\bar\xi}^{\mathcal D,\rel}"]\\
\Sp(H_{\bar\eta,\Zl})
\arrow[r,"\alpha_{\Zl,*}","\sim"']&
\Sp(H_{\bar\xi,\Zl}).
\end{tikzcd}
\]
The level-cover maps are induced by pullback of finite \'etale covers
on the common models, and the exact sequences are natural for these
maps. Since the same base points and path $\beta$ are used for all
$\mathcal D$, the comparisons are compatible with enlarging $\mathcal D$.
Purity gives the one-node version.
\end{proof}

\subsection{Continuous relative completion}
\label{subsec:compact-type-continuous-specialization}

The Betti symplectic monodromy of the smooth level cover has image
of finite index in the image of $\G_{g,n}$ in the integral symplectic
group, and is therefore Zariski dense.
By characteristic-zero comparison, the monodromy of
$\pi_1(M^\lambda_{\bar L})$ is Zariski dense over $\Ql$ as well.
The dense open immersion
$M^\lambda_{\bar L}\hookrightarrow M_{\mathcal D/\bar L}^\lambda$ induces a
surjection on fundamental groups, and the monodromy on the smooth
level cover is the restriction of that on $M_{\mathcal D/\bar L}^\lambda$.
Consequently, the monodromy of $\pi_1(M_{\mathcal D/\bar L}^\lambda)$ is Zariski dense
in $\Sp(H_{\bar\eta,\Ql})$.
Lemma~\ref{lem:specialization-monodromy} gives the same conclusion
for $\pi_1(M_{\mathcal D/k}^\lambda)$.
Since the level-cover groups map to $\pi_1(\M_{\mathcal D/A})$, the monodromy of
$\pi_1(\M_{\mathcal D/A})$ is also Zariski dense for $A=\bar L,k$.
The same holds over $W$, by compatibility with the generic-fiber
map. We may therefore pass to continuous relative completion.

For $A=\bar L,W,k$, denote by $\cG_{\M_{\mathcal D}/A}$ the
continuous relative completion over $\Ql$ with respect to symplectic
monodromy, and by $\u_{\M_{\mathcal D}/A}$ the Lie algebra of its prounipotent
radical. When $\mathcal D$ contains all separating types, these are the groups
$\cG^c_{g,n/A}$ and Lie algebras $\u^c_{g,n/A}$ considered earlier.

\begin{proposition}[Specialization of relative completion]
\label{prop:compact-type-relative-completion-specialization}
The specialization data induce isomorphisms
\[
\cG_{\M_{\mathcal D}/\bar L}\xrightarrow{\sim}\cG_{\M_{\mathcal D}/k},
\qquad
\u_{\M_{\mathcal D}/\bar L}\xrightarrow{\sim}\u_{\M_{\mathcal D}/k},
\]
compatible with symplectic monodromy under
$\alpha_{\Ql}=\alpha_{\Zl}\otimes_{\Zl}\Ql$.
They are functorial for enlarging $\mathcal D$.
The same statements hold for $\M_{\mathcal D}^{\circ}$.
\end{proposition}

\begin{proof}
The preceding paragraph verifies Zariski density for
Proposition~\ref{prop:continuous-completion-relative-pro-l}.
Thus the completion of $\pi_1(\M_{\mathcal D/A})$ is canonically
that of $\pi_1(\M_{\mathcal D/A})^\rel$.
The preceding relative pro-$\ell$ isomorphism is compatible with
the $\Ql$-valued Jacobian monodromy representations under
$\alpha_{\Ql}$.
By functoriality of continuous relative completion
(Subsection~\ref{subsec:relative-completion-properties}), this
isomorphism and its inverse induce mutually inverse maps on
continuous relative completions, proving the first assertion.

Writing $\mathcal U$ for the prounipotent radicals, compatibility
with monodromy gives the commutative diagram with exact rows
\[
\begin{tikzcd}[
column sep=1.2em,
row sep=1.5em,
cells={nodes={font=\small}}
]
1\arrow[r]&
\mathcal U_{\M_{\mathcal D}/\bar L}
\arrow[r]\arrow[d]&
\cG_{\M_{\mathcal D}/\bar L}
\arrow[r]\arrow[d,"\sim"]&
\Sp(H_{\bar\eta,\Ql})
\arrow[r]\arrow[d,"\alpha_{\Ql,*}","\sim"']&1\\
1\arrow[r]&
\mathcal U_{\M_{\mathcal D}/k}
\arrow[r]&
\cG_{\M_{\mathcal D}/k}
\arrow[r]&
\Sp(H_{\bar\xi,\Ql})
\arrow[r]&1.
\end{tikzcd}
\]
The middle and right vertical maps are isomorphisms, so the
left vertical map is also an isomorphism. Taking Lie algebras
gives the second assertion.
Functoriality and the assertion for $\M_{\mathcal D}^{\circ}$
follow from Proposition~\ref{prop:selected-functorial-comparison}
and purity.
\end{proof}

In particular, the inclusions into compact type give the commutative
diagram
\[
\begin{tikzcd}[
column sep=1.3em,
cells={nodes={font=\small}}
]
\cG_{\M_{\mathcal D}^{\circ}/\bar L}
\arrow[r,"\sim"]\arrow[d]&
\cG_{\M_{\mathcal D}^{\circ}/W}\arrow[d]&
\cG_{\M_{\mathcal D}^{\circ}/k}
\arrow[l,"\sim"']\arrow[d]\\
\cG^c_{g,n/\bar L}\arrow[r,"\sim"]&
\cG^c_{g,n/W}&
\cG^c_{g,n/k}\arrow[l,"\sim"'].
\end{tikzcd}
\]
Here the two total-space base points are identified using $\beta$.
The same diagram holds with $\M_{\mathcal D}$ in place of
$\M_{\mathcal D}^{\circ}$.
For an arbitrary collection $\mathcal D$, this compares characteristics;
it does not assert that the vertical maps are isomorphisms.

\begin{proof}[Proof of Theorem~\ref{thm:compact-type-completion}]
If $\operatorname{char}(k)=0$, the theorem follows from
Proposition~\ref{prop:generic-fiber-compact-type-comparison}
with $F=k$.
Now suppose that $\operatorname{char}(k)=p>0$.
Since $\bar L$ has characteristic zero,
Proposition~\ref{prop:generic-fiber-compact-type-comparison}
with $F=\bar L$ shows that $\u^c_{g,n/\bar L}$ is abelian and
\[
\u^c_{g,n/\bar L}
\cong\Lambda^3_0H_{\bar\eta,\Ql}
\oplus H_{\bar\eta,\Ql}^{\oplus n}.
\]
Apply Proposition~\ref{prop:compact-type-relative-completion-specialization}
with $\mathcal D$ equal to the collection of all separating boundary
types, so that $\M_{\mathcal D}=\M^c_{g,n}$.
It gives an isomorphism
$\u^c_{g,n/\bar L}\xrightarrow{\sim}\u^c_{g,n/k}$
compatible with symplectic monodromy under
$\alpha_{\Ql}:H_{\bar\eta,\Ql}\xrightarrow{\sim}H_{\bar\xi,\Ql}$.
Consequently, $\u^c_{g,n/k}$ is abelian, and this identification gives
\[
\u^c_{g,n/k}
\cong\Lambda^3_0H_{\bar\xi,\Ql}
\oplus H_{\bar\xi,\Ql}^{\oplus n}
\]
as $\Sp(H_{\bar\xi,\Ql})$-modules.
\end{proof}

\begin{proof}[Proof of Theorem~\ref{thm:fixed-genus-main}]
The characteristic-zero statement follows from
Theorem~\ref{thm:fixed-genus-betti} and characteristic-zero
comparison. Suppose now that $\operatorname{char}(k)=p>0$, and let
$\mathcal D$ be either $\mathcal D_h$ or $\mathcal D_{h,\leq2}$.
Since $\bar L$ has characteristic zero, the inclusion induces an
isomorphism
\[
\cG_{\M_{\mathcal D}^{\circ}/\bar L}
\xrightarrow{\sim}\cG^c_{g,n/\bar L}.
\]
Apply the preceding specialization diagram to
$\M_{\mathcal D}^{\circ}\hookrightarrow\M^c_{g,n}$.
The horizontal maps are isomorphisms, and the vertical map on the
generic fiber is the isomorphism just obtained. Commutativity
therefore implies that the vertical map over $k$ is also an
isomorphism.
This isomorphism is compatible with the projections to
$\Sp(H_{\bar\xi,\Ql})$, so it identifies their kernels,
the prounipotent radicals. Taking Lie algebras and applying
Theorem~\ref{thm:compact-type-completion} gives
\[
\u_{\M_{\mathcal D}^{\circ}/k}
\cong\u^c_{g,n/k}
\cong\Lambda^3_0H_{\bar\xi,\Ql}
\oplus H_{\bar\xi,\Ql}^{\oplus n},
\]
with trivial Lie bracket and the asserted symplectic action.
The argument was carried out over $W(k)$, so it applies to every
algebraically closed field $k$ of characteristic $p$.
\end{proof}

\section{Extension of Local Systems}\label{sec:local-systems}
%%%%%%%%%%%%%%%%%%%%%%%%%%%%%%%%%%%%%%%%%%%%%%%%%%%%%
%%%%%%%%%%%%%%%%%%%%%%%%%%%%%%%%%%%%%%%%%%%%%%%%%%%%%
Suppose that $g\geq3$ and $n\geq0$. We work either with rational
Betti local systems over $\C$, with coefficient field $F=\Q$, or
with lisse $\Ql$-sheaves over an algebraically closed field $k$,
with $F=\Ql$. In the latter setting, $\ell\neq\operatorname{char}(k)$,
and $\ell$ is assumed odd when $\operatorname{char}(k)>0$.
All stacks are taken over the corresponding base field.
Write $\mathbb H^c_F$ for the Jacobian local system of
Subsection~\ref{subsec:jacobian-monodromy-completion}, and let $H$
be its fiber at a fixed base point in $\M_{g,n}$.
The notation $Z$ will denote any of the opens $\M_{g,n}$,
$\M_{\mathcal D}^{\circ}$, $\M_{\mathcal D}$, or $\M^c_{g,n}$.

\subsection{The category associated to relative completion}
\begin{definition}
Let $\Loc_H(Z;F)$ be the full subcategory of finite-dimensional
$F$-local systems on $Z$ consisting of those admitting a finite
filtration by local subsystems
\[
0=\V_0\subseteq\V_1\subseteq\cdots\subseteq\V_m=\V
\]
whose successive quotients are associated to algebraic
representations of $\Sp(H)$ through the monodromy of
$\mathbb H^c_F|_Z$.
The filtration is required to exist but is not part of the data;
morphisms are all morphisms of the underlying local systems.
\end{definition}

\begin{proposition}\label{prop:tannakian-local-systems}
With the fiber functors at the chosen base point, there are
equivalences of tensor categories
\begin{equation}\label{eq:tannakian}
\begin{aligned}
\Loc_H(Z;\Q)&\simeq\Rep_{\Q}(\cG_Z),\\
\Loc_H(Z;\Ql)&\simeq\Rep_{\Ql}(\cG_{Z/k})
\end{aligned}
\end{equation}
in the Betti and \'etale settings, respectively. Here $\Rep_F$
denotes finite-dimensional algebraic representations over $F$.
\end{proposition}
\begin{proof}
This is the Tannakian description of relative completion
\cite[\S4]{HainMatsumoto2009}. We give the argument using the
universal properties recalled in
Subsection~\ref{subsec:relative-completion-definition}, which also
apply in the continuous setting \cite[\S6.2]{hain_rational}.
Write $\Gamma$ for $\pi_1^{\orb}(Z^{\an})$ in the Betti setting
and for $\pi_1(Z)$ in the \'etale setting, and denote its relative
completion by $\cG$.

Let $V$ be the fiber of an object of $\Loc_H(Z;F)$, and choose a
filtration as in the definition. Let $P_V\subseteq\GL(V)$ be its
stabilizer. A vector-space splitting of the filtration shows that
\[
P_V\longrightarrow\prod_{i=1}^m\GL(V_i/V_{i-1})
\]
is surjective, with kernel consisting of block upper triangular
matrices with identity diagonal blocks, hence unipotent.
% Let $V$ be the fiber of an object of $\Loc_H(Z;F)$, and choose a
% filtration as in the definition. Let $P_V\subseteq\GL(V)$ be its
% stabilizer. Let $P_V\subseteq\GL(V)$ be the stabilizer of the filtration.
% The homomorphism
% $P_V\to\prod_{i=1}^m\GL(V_i/V_{i-1})$ is surjective.
% Indeed, choose vector-space complements $E_i$ such that
% $V_i=V_{i-1}\oplus E_i$. Then $V=\bigoplus_{i=1}^m E_i$,
% and the quotient maps identify $E_i$ with $V_i/V_{i-1}$.
% Any tuple of automorphisms of the graded quotients therefore
% lifts to the direct sum of the corresponding automorphisms of
% the $E_i$. This lift preserves every $V_i$, so it belongs to $P_V$.
% With respect to this decomposition, the kernel consists of block
% upper triangular matrices with identity diagonal blocks, and is
% therefore unipotent. 
The algebraic symplectic actions on the graded
quotients therefore give an algebraic group
\[
G_V=\Sp(H)\times_{\prod_a\GL(V_a/V_{a-1})}P_V
\]
that is an extension of $\Sp(H)$ by a unipotent group.
The pair consisting of symplectic monodromy and the monodromy on
$V$ gives a lift $\Gamma\to G_V(F)$, continuous in the \'etale
setting. The universal property of $\cG$ extends this lift and
hence gives an algebraic representation of $\cG$ on $V$.

Conversely, let $V$ be a finite-dimensional algebraic
representation of $\cG$, and let $\U$ be its prounipotent radical.
The action of $\U$ on $V$ factors through a unipotent algebraic
group. Set $V_0=0$ and define $V_{i+1}$ to be the inverse image
of $(V/V_i)^{\U}$ under the quotient map $V\to V/V_i$. Thus
\[
V_{i+1}/V_i=(V/V_i)^{\U}.
\]
Normality of $\U$ makes this filtration $\cG$-stable.
Since every nonzero representation of a unipotent group has
nonzero invariants, $V_i\subsetneq V_{i+1}$ whenever $V_i\neq V$.
Finite dimensionality therefore gives $V_m=V$ for some $m$.
% Since $\U$ is normal in $\cG$, the $\U$-invariant subspace of
% every $\cG$-representation is $\cG$-invariant. Inductively, each
% $V_i$ is therefore $\cG$-invariant, so the construction is
% well-defined. Whenever $V_i\neq V$, the nonzero quotient $V/V_i$
% has a nonzero $\U$-invariant vector, since $\U$ acts through a
% unipotent group. Hence $V_i\subsetneq V_{i+1}$, and finite
% dimensionality implies that $V_m=V$ for some $m$.
By construction, $\U$ acts trivially on every $V_{i+1}/V_i$.
The graded actions consequently factor through
$\cG/\U\cong\Sp(H)$, so restriction to $\Gamma$ gives an object
of $\Loc_H(Z;F)$.

Finally, the canonical image of $\Gamma$ is Zariski dense in
$\cG$. Thus an extension of the representation is unique, and a
linear map intertwines the $\Gamma$-actions if and only if it
intertwines the $\cG$-actions. Restriction is therefore a fully
faithful, essentially surjective tensor functor, as required.
\end{proof}

\subsection{Extension from a partial compactification}
\begin{theorem}\label{thm:local-extension}
Suppose that $1\leq h\leq g-1$, and put
$X=\M_{\mathcal D_{h,\leq2}}^{\circ}$ and $Y=\M^c_{g,n}$.
Restriction along $j:X\hookrightarrow Y$ induces an equivalence
of tensor categories
\[
j^*:\Loc_H(Y;F)\xrightarrow{\sim}\Loc_H(X;F).
\]
The same assertion holds for $X=\M_{\mathcal D_h}^{\circ}$ and
for the corresponding larger opens
$\M_{\mathcal D_{h,\leq2}}$ and $\M_{\mathcal D_h}$.
\end{theorem}
\begin{proof}
Under Proposition~\ref{prop:tannakian-local-systems}, restriction
corresponds to pullback of representations along
$\cG_X\to\cG_Y$ in the Betti setting and along
$\cG_{X/k}\to\cG_{Y/k}$ in the \'etale setting.
These morphisms are isomorphisms by
Theorem~\ref{thm:fixed-genus-main} for the one-node opens.
The fundamental-group comparisons of
Section~\ref{sec:fixed-genus} give the same conclusion for the
larger opens. Pullback therefore gives the asserted equivalences.
\end{proof}

Full faithfulness shows that an extension, together with a
specified identification of its restriction, is unique up to a
unique isomorphism respecting that identification. Since restriction
is exact, a chosen filtration by local subsystems and its inclusion
maps extend as well.

\subsection{A boundary-monodromy criterion}
In the Betti setting, boundary monodromy is the action of a
meridian around a boundary divisor. In the \'etale setting, it
is the action of the full inertia group at a geometric generic
point of the divisor on a regular \'etale chart.

\begin{lemma}\label{lem:local-inertia}
Let $U\subseteq Z$ be a dense open immersion between smooth
connected stacks considered here, whose complement consists of
boundary divisors and possibly strata of codimension at least two.
A finite-dimensional local system on $U$ extends to $Z$ if and
only if its boundary monodromy is trivial along every
codimension-one component of $Z\setminus U$.
The extension is unique up to a unique isomorphism compatible
with its specified restriction.
\end{lemma}
\begin{proof}
In the Betti setting, the inclusion induces a surjection on
orbifold fundamental groups whose kernel is normally generated
by boundary meridians, by transversality. The assertion follows
by factoring the monodromy representation through this quotient.

In the \'etale setting, choose a monodromy-invariant
$\Zl$-lattice. Trivial boundary inertia implies that each finite
representation obtained by reduction modulo $\ell^m$ extends
across the boundary, by purity of the branch locus
\cite[Exp.~X, Thm.~3.1]{SGA1}. Passing to the inverse limit
and tensoring with $\Ql$ gives the required extension.
For the stacks considered here, apply this argument on regular
\'etale charts and descend.

In both settings, surjectivity on fundamental groups makes
restriction fully faithful, which gives the stated uniqueness.
\end{proof}
% \begin{proof}
% In the Betti setting, the homomorphism
% $\pi_1^{\orb}(U^{\an})\to\pi_1^{\orb}(Z^{\an})$ is surjective,
% with kernel normally generated by the boundary meridians.
% This follows by transversality on smooth charts; strata of complex
% codimension at least two do not affect the fundamental group.
% A monodromy representation therefore factors through
% $\pi_1^{\orb}(Z^{\an})$ exactly when these meridians act trivially.

% In the \'etale setting, choose a monodromy-invariant $\Zl$-lattice;
% one exists because the image of $\pi_1(U)$ is compact.
% Each reduction modulo $\ell^m$ has finite image.
% If full boundary inertia acts trivially, the normalization of
% each corresponding finite \'etale cover over a regular chart
% of $Z$ is unramified at every codimension-one point.
% These normalizations are finite because the charts are excellent.
% Purity of the branch locus
% \cite[Exp.~X, Thm.~3.1]{SGA1} makes them \'etale everywhere.
% Uniqueness of extension over normal charts extends the actions
% and descent data. Thus each finite representation factors through
% $\pi_1(Z)$. These factorizations are compatible, so their inverse
% limit, tensored with $\Ql$, gives the required lisse sheaf.
% Conversely, a sheaf on $Z$ has trivial boundary inertia on $U$.

% In both settings, the homomorphism on the corresponding
% fundamental groups is surjective. Restriction is consequently
% fully faithful, which proves uniqueness.
% \end{proof}

\begin{corollary}\label{cor:boundary-test}
Suppose that $1\leq h\leq g-1$, and let
$\V\in\Loc_H(\M_{g,n};F)$. The following are equivalent:
\begin{enumerate}[(i)]
\item $\V$ extends to a local system on $\M^c_{g,n}$;
\item $\V$ extends to a local system on
$\M_{\mathcal D_{h,\leq2}}^{\circ}$;
\item the boundary monodromy of $\V$ is trivial along every
$\Delta_{h,S}$ with $|S|\leq2$.
\end{enumerate}
The extensions belong to the corresponding categories $\Loc_H$
and are unique up to a unique isomorphism compatible with their
specified restrictions. The same assertion holds with all subsets
$S$ in \textup{(ii)} and \textup{(iii)}.
\end{corollary}
\begin{proof}
Put $X=\M_{\mathcal D_{h,\leq2}}^{\circ}$.
Restriction gives \textup{(i)}$\Rightarrow$\textup{(ii)}, and
Lemma~\ref{lem:local-inertia} gives
\textup{(ii)}$\Leftrightarrow$\textup{(iii)}.

For \textup{(ii)}$\Rightarrow$\textup{(i)}, put $U=\M_{g,n}$
and choose an extension of $\V$ to $X$. Write
$0=V_0\subseteq\cdots\subseteq V_m=V$ for the fiber filtration
of $\V$ supplied by the definition of $\Loc_H$.
Let $q:\Gamma_U\twoheadrightarrow\Gamma_X$ be the surjection
induced by inclusion, where $\Gamma_U$ and $\Gamma_X$ denote
the corresponding orbifold or \'etale fundamental groups.
Surjectivity of $q$ implies that each $V_i$ is invariant under
the extended monodromy.

Denote the Jacobian monodromy representations by $\rho_U$
and $\rho_X$, so that $\rho_U=\rho_X\circ q$.
For each $i$, let
$r_i:\Sp(H)\to\GL(V_i/V_{i-1})$
be the algebraic representation defining the original graded
monodromy, and let
$\sigma_i:\Gamma_X\to\GL(V_i/V_{i-1})(F)$
be the graded monodromy of the extension.
Since the extension restricts to $\V$,
\[
\sigma_i\circ q
=r_i\circ\rho_U
=r_i\circ\rho_X\circ q.
\]
Surjectivity of $q$ gives $\sigma_i=r_i\circ\rho_X$.
Thus the extended filtration has algebraic symplectic graded
quotients, and the extension belongs to $\Loc_H(X;F)$.
Theorem~\ref{thm:local-extension} then extends it to an object of
$\Loc_H(\M^c_{g,n};F)$.
Uniqueness follows from Lemma~\ref{lem:local-inertia}.
Replacing $\mathcal D_{h,\leq2}$ by $\mathcal D_h$ and applying
the same lemma and theorem proves the assertion for all subsets $S$.
\end{proof}

Theorem~\ref{thm:local-extension} and
Corollary~\ref{cor:boundary-test} prove
Corollary~\ref{thm:local-extension-main}.
For $n=0$, extension to compact type is therefore equivalent to
trivial boundary monodromy along the single divisor
$\Delta_{h,\varnothing}$.
In the Betti setting, it is enough to test one genus-$h$
separating twist, since all twists of that type are conjugate in
$\G_g$.

% NEW SECTION: the divisor-count consequence in the project notes.
\section{The Number of Boundary Divisors}\label{sec:divisor-count}
For $h\neq g-h$, the collection $\mathcal D_h$ consists of $2^n$
distinct separating boundary divisors. If $g=2h$ and $n\geq1$, the identification of complementary
subsets leaves $2^{n-1}$ types. For $n=0$, there is one type in
either case. The collection $|S|\leq2$ has at most
\[
1+n+\binom n2
\]
distinct divisors. Complementary subsets can reduce this number
when $g=2h$.

\begin{corollary}\label{cor:divisor-lower-bound}
Suppose that $\mathcal D$ is a collection of $r$ distinct separating
boundary divisors and that
\[
\M_{\mathcal D}^{\circ}\hookrightarrow\M^c_{g,n}
\]
induces an isomorphism on the relative completions of
Theorem~\ref{thm:fixed-genus-main}. Then $r\geq\binom n2$.
\end{corollary}
\begin{proof}
First consider rational Betti relative completion.
Write $\mathcal D=\{D_1,\ldots,D_r\}$ and choose a full
meridian $T_s$ around each $D_s$, regarded as a separating
Dehn twist in $\G_{g,n}$. Any two positive full meridians around $D_s$ are conjugate
in $\G_{g,n}$, since their associated separating curves have
the same topological type.

Let $N_{\mathcal D}$ be the normal subgroup generated by the
$T_s$, and let $\mathfrak n_{\mathcal D}\subseteq\u_{g,n}$ be the Lie
algebra of its closed normal image in $\cG_{g,n}$.
By the boundary-meridian description of the fundamental group
and Proposition~\ref{prop:relative-completion-quotient},
the prounipotent Lie algebra of the relative completion of
$\M_{\mathcal D}^{\circ}$ is $\u_{g,n}/\mathfrak n_{\mathcal D}$.
Proposition~\ref{prop:Kgn-weight-minus-two} gives
$\mathfrak n_{\mathcal D}\subseteq W_{-2}\u_{g,n}$.

Set $d_s=\operatorname{in}_{-2}(T_s)\in\Gr^W_{-2}\u_{g,n}$.
By \eqref{eq:wt2}, the canonical equivariant projection onto
the trivial $\Sp(H)$-isotypical component is
\[
\operatorname{pr}_{\mathrm{triv}}:
\Gr^W_{-2}\u_{g,n}\longrightarrow
\bigl(\Gr^W_{-2}\u_{g,n}\bigr)_{\mathrm{triv}}
=\bigoplus_{i<j}\Q(1)_{ij}.
\]
Its target has dimension $\binom n2$.

Conjugation equivariance, established in the proof of
Lemma~\ref{lem:image-Kgn-Sp-submodule}, gives
\[
\begin{aligned}
\operatorname{pr}_{\mathrm{triv}}
\bigl(\operatorname{in}_{-2}(\gamma T_s\gamma^{-1})\bigr)
&=\operatorname{pr}_{\mathrm{triv}}
\bigl(\rho_{g,n}(\gamma)d_s\bigr)\\
&=\operatorname{pr}_{\mathrm{triv}}(d_s)
\end{aligned}
\qquad(\gamma\in\G_{g,n}),
\]
since $\Sp(H)$ acts trivially on the target of
$\operatorname{pr}_{\mathrm{triv}}$.
% Since
% $[\u_{g,n},W_{-2}\u_{g,n}]\subseteq W_{-3}\u_{g,n}$,
% the adjoint action on $\Gr^W_{-2}\u_{g,n}$ factors through
% $\Sp(H)$. Consequently, for every $\gamma\in\G_{g,n}$,
% \[
% \operatorname{in}_{-2}(\gamma T_s\gamma^{-1})
% =\rho_{g,n}(\gamma)d_s.
% \]
% The action on the trivial isotypical component is trivial, so
% \[
% \operatorname{pr}_{\mathrm{triv}}
% \bigl(\operatorname{in}_{-2}(\gamma T_s\gamma^{-1})\bigr)
% =\operatorname{pr}_{\mathrm{triv}}(d_s).
% \]

The weight-$-2$ initial form is additive on $K_{g,n}$.
Since $N_{\mathcal D}$ is generated by these conjugates and
their inverses, its projected initial forms lie in
\[
\operatorname{span}_{\Q}
\{\operatorname{pr}_{\mathrm{triv}}(d_1),\ldots,
  \operatorname{pr}_{\mathrm{triv}}(d_r)\}.
\]
By Lemma~\ref{lem:boundary-ideal-weight-two}, the image of
$\mathfrak n_{\mathcal D}$ in $\Gr^W_{-2}\u_{g,n}$ is
$I_{\mathcal D}$, the span of the initial forms of elements of
$N_{\mathcal D}$. Its projection onto the trivial isotypical
component therefore has dimension at most $r$.
% Passing to the closed Lie ideal $\mathfrak n_{\mathcal D}$ adds no
% further projected directions: brackets with $\u_{g,n}$ vanish
% in weight $-2$, and the displayed subspace is closed.
% Thus the projection of the image of $\mathfrak n_{\mathcal D}$ in
% $\Gr^W_{-2}\u_{g,n}$ has dimension at most $r$.

By hypothesis, $\u_{g,n}/\mathfrak n_{\mathcal D}$ is isomorphic to the
compact-type prounipotent Lie algebra, which is abelian.
Hence
\[
W_{-2}\u_{g,n}
=\overline{[\u_{g,n},\u_{g,n}]}
\subseteq \mathfrak n_{\mathcal D}.
\]
Together with the reverse inclusion established above, this
gives $\mathfrak n_{\mathcal D}=W_{-2}\u_{g,n}$.
Its image in $\Gr^W_{-2}\u_{g,n}$ is therefore the entire
space, so its projection onto the trivial isotypical component
is surjective. Consequently, $\binom n2\leq r$.

For continuous relative completion in characteristic zero,
the same argument applies after extension of scalars to $\Ql$
and profinite comparison.
In positive characteristic, apply
Proposition~\ref{prop:compact-type-relative-completion-specialization}
to $\M_{\mathcal D}^{\circ}$ and $\M^c_{g,n}$.
Functoriality identifies the inclusion-induced maps under the
specialization isomorphisms. Thus an isomorphism in positive
characteristic gives an isomorphism in characteristic zero,
where the bound has already been proved.
\end{proof}

This proves Corollary~\ref{cor:divisor-lower-bound-main}.
The bound also applies if restriction is assumed to give an
equivalence of the categories $\Loc_H$ with their fiber functors,
since those categories recover the relative completions.
Thus the lower bound and the fixed-genus upper bound have the
same quadratic leading term. No assertion that the lower bound
is attained is needed here.

\begin{remark}
For $g=3$ and $n=0$, the types $h=1$ and $h=2$ label the
same divisor, which is the only separating boundary divisor.
Thus the complement of $\M_{\mathcal D_1}^{\circ}$ in
$\M^c_3$ has codimension at least two.
The inclusion therefore induces an isomorphism on \'etale
fundamental groups by purity
\cite[Exp.~X, Thm.~3.1]{SGA1}, and hence on continuous
relative completions.
The rational Betti assertion follows from
Corollary~\ref{cor:rational-to-continuous-completion}
and faithful extension of scalars from $\Q$ to $\Ql$.
\end{remark}

%%%%%%%%%%%%%%%%%%%%%%%%%%%%%%%%%%%%%%%%%%%%%%%%%%%%%%%%%%%%%%%%%%%%%%%%%%%%%%%%%%%%%%%%%%%%%%%%%%%%%%
\section*{Notation and glossary}\label{sec:notation-glossary}

We use $\pi_1$, $\pi_1^{\top}$, and $\pi_1^{\orb}$ for \'etale,
topological, and orbifold fundamental groups, respectively.
The subscript $/k$ indicates base change; in geometric completion
statements, $k$ is algebraically closed.

\subsection*{Groups and completions}
\begin{description}
\item[$\G_{g,n},T_{g,n},K_{g,n}$]
The mapping class group fixing the markings pointwise, its Torelli
subgroup, and the subgroup generated by separating Dehn twists.

\item[$\M_{g,n},\M^c_{g,n},\Mbar_{g,n}$]
Moduli of smooth, compact-type, and stable genus-$g$ curves with
$n$ ordered markings, respectively.

\item[$\Pi_{g,n},\Pi^{\mathrm{pun}}_{g,n}$]
The pure surface braid group $\ker(\G_{g,n}\to\G_g)$ and the
fundamental group of $\Sigma_g\setminus\{x_1,\ldots,x_n\}$,
respectively.

\item[$\Gamma^{(\ell)},\widehat\Gamma$]
The ordinary pro-$\ell$ and profinite completions of $\Gamma$.

\item[$\Gamma^{\rel,\rho}$]
Relative pro-$\ell$ completion with respect to a dense profinite
representation $\rho$; we write $\Gamma^\rel$ when $\rho$ is understood.
This is a profinite group. For moduli groups, the representation is
integral Jacobian monodromy onto its closed image.

\item[$\cG_Z,\cG_{Z/k}$]
Relative completions of the Betti and geometric \'etale fundamental
groups of $Z$, over $\Q$ and $\Ql$, respectively, with respect to
Jacobian monodromy. These are proalgebraic groups.
For $Z=\M_{g,n}$ or $\M^c_{g,n}$, we write
$\cG_{g,n}$ or $\cG^c_{g,n}$, with $/k$ in the \'etale setting.

\item[$\U,\g,\u$]
The prounipotent radical of a relative completion $\cG$ and the
Lie algebras of $\cG$ and $\U$, respectively. Subscripts and
superscripts follow those of $\cG$.

\item[$\cP,\p$]
The rational unipotent completion of the fundamental group of a
smooth complex fiber and its Lie algebra.

\item[$F_n(C),\cP_{g,n},\p_{g,n}$]
The ordered configuration space of $n$ distinct points on $C$,
the rational unipotent completion of its fundamental group,
and its Lie algebra.

\item[$W_\bullet,L^\bullet$]
The weight filtration and lower central series; lower central
terms are closed in the pronilpotent topology.

\item[$\widehat\otimes$]
Scalar extension of pro-vector spaces, computed on
finite-dimensional quotients.
\end{description}

\subsection*{Symplectic notation}
\begin{description}
\item[$\mathbb H^c_\Z,\mathbb H^c_\Q,
\mathbb H^c_{\Zl},\mathbb H^c_{\Ql}$]
The Jacobian local systems on compact type, with the indicated
Betti or \'etale coefficients.

\item[$\mathbb H_\Z,\mathbb H_\Q,\mathbb H_{\Zl}$]
Their corresponding restrictions to the smooth locus.

\item[$H_\Q,H_{\Zl},H_{\Ql}$]
At the chosen smooth fibers,
$H_\Q=H^1(C^{\an},\Q(1))\cong H_1(C^{\an},\Q)$,
$H_{\Zl}=H^1_{\et}(C,\Zl(1))$, and
$H_{\Ql}=H_{\Zl}\otimes_{\Zl}\Ql$.
We write $H=H_\Q$ in the Betti setting and $H=H_{\Ql}$
in the $\ell$-adic setting.

\item[$\theta,V(\lambda)$]
The symplectic class and the irreducible $\Sp(H)$-module of
highest weight $\lambda$, with Tate twists as specified in the text.

\item[$\Lambda^2_0H,\Lambda^3_0H$]
The primitive exterior powers, normalized as
$\Lambda^2_0H=\Lambda^2H/\langle\theta\rangle$ and
$\Lambda^3_0H=(\Lambda^3H/(\theta\wedge H))(-1)$.

\item[$\rho_{g,n},\rho^c_{g,n}$]
Betti Jacobian monodromy on the smooth and compact-type loci,
respectively; $/k$ denotes the $\ell$-adic versions.

\item[$\widetilde\rho_{g,n},\widetilde\rho^c_{g,n}$]
Canonical lifts to the corresponding Betti relative completions.
\end{description}

% NEW NOTATION: appended in black, rather than recoloring the glossary.
\subsection*{Boundary notation and local systems}
\begin{description}
\item[$\Delta_{h,S},\Delta_{\mathrm{irr}}$]
The stable separating divisor whose genus-$h$ side carries $S$,
and the nonseparating divisor, respectively.
The labels $(h,S)$ and $(g-h,S^c)$ denote the same divisor;
$D^\circ$ denotes the one-node locus of $D$.

\item[$\M_{\mathcal D}^{\circ},\M_{\mathcal D}$]
The smooth locus together with the one-node strata of types in
$\mathcal D$, and the compact-type locus where every node has a
type in $\mathcal D$, respectively.

\item[$\mathcal D_h,\mathcal D_{h,\leq2}$]
The genus-$h$ collections using all subsets $S$ and only those
with $|S|\leq2$, respectively; complementary labels are counted once.

\item[$N_{\mathcal D},\mathcal N_{\mathcal D},
\mathfrak n_{\mathcal D}$]
The normal subgroup of $\G_{g,n}$ generated by twists of types
in $\mathcal D$, the Zariski closure of its image in $\cG_{g,n}$,
and the Lie algebra of that closure, respectively.
For all separating types, these are $K_{g,n}$,
$\cN_{g,n}$, and $\n_{g,n}$.

\item[$\operatorname{in}_{-2}(D),I_N$]
When the logarithm of $D$ lies in $W_{-2}$ of its indicated
completion, $\operatorname{in}_{-2}(D)$ is its class modulo $W_{-3}$.
For $N\unlhd\G_{g,n}$ contained in $K_{g,n}$, $I_N$ is the span
of these classes in $\Gr^W_{-2}\u_{g,n}$.

\item[$I_{\mathcal D}$]
The space $I_{N_{\mathcal D}}$.

\item[$N_h,\mathfrak n_h,I_h$]
The corresponding objects for $\mathcal D=\mathcal D_{h,\leq2}$.

\item[$\Loc_H(Z;F)$]
Finite-dimensional local systems admitting a finite filtration
with algebraic symplectic graded quotients.
Here $F=\Q$ in the Betti setting and $F=\Ql$ in the \'etale setting;
the filtration is not part of the data.
\end{description}

\subsection*{Level covers and specialization}
\begin{description}
\item[$\Gamma^\lambda,\Delta^\lambda$]
The geometric level subgroup of $\G_{g,n}$ and the deck group
of the chosen connected level cover.

\item[$M^\lambda,\overline M^\lambda$]
The smooth level cover and its normalization over $\Mbar_{g,n}$.

\item[$M_{\mathcal D}^\lambda,M^{c,\lambda}$]
The level covers over $\M_{\mathcal D}$ and $\M^c_{g,n}$;
$M_{\mathcal D}^{\circ,\lambda}$ is the restriction to
$\M_{\mathcal D}^{\circ}$.

\item[$W,L,\bar L$]
The Witt ring $W(k)$, its fraction field, and an algebraic closure
of that field, for $k$ of characteristic $p>0$.

\item[$\bar\eta,\bar\xi,\beta,\alpha_{\Zl},\alpha_{\Ql}$]
Generic- and special-fiber base points, an \'etale path between
them in the common smooth level locus, and the induced symplectic
transports on Jacobian fibers.

\item[$\operatorname{sp}_{\mathcal D}^{(\ell)},
\operatorname{sp}_{\mathcal D}^{\rel}$]
Specialization maps on the maximal pro-$\ell$ quotients of
level-cover groups and on relative pro-$\ell$ completions of
moduli groups, respectively.
\end{description}

\subsection*{AI disclosure}
The authors used OpenAI's ChatGPT to assist with literature
searches, check calculations and notation, and improve the
exposition of their original mathematical ideas and arguments.
All AI-assisted suggestions were independently reviewed and
verified by the authors, who take full responsibility for the
content of the paper.
\bibliographystyle{alpha}
\bibliography{references}
\end{document}